\documentclass[a4paper,reqno]{amsart} 

\usepackage{amsmath,amsthm,amssymb,amsfonts,mathrsfs,color,hyperref, mathtools,crop, graphicx, enumitem,tikz}
\usepackage[g]{esvect}
\usepackage[bottom=4cm]{geometry}
\theoremstyle{plain}
\begingroup
\newtheorem{Theorem}{Theorem}[section]
\newtheorem{theorem}[Theorem]{Theorem}
\newtheorem*{theorem*}{Theorem}
\newtheorem*{"theorem"}{``Theorem''}

\newtheorem{corollary}[Theorem]{Corollary}
\newtheorem{Proposition}[Theorem]{Proposition}

\newtheorem{Lemma}[Theorem]{Lemma}
\newtheorem{lemma}[Theorem]{Lemma}
\endgroup

\theoremstyle{definition}
\begingroup
\newtheorem{definition}[Theorem]{Definition}
\endgroup

\theoremstyle{remark}
\begingroup
\newtheorem{remark}[Theorem]{Remark}

\endgroup 

\numberwithin{equation}{section}
\newcommand\numberthis{\addtocounter{equation}{1}\tag{\theequation}}

\newenvironment{pde}{\left\{\begin{array}{rll} } {\end{array}\right.}

\newcommand{\N}{\mathbb N} 
\newcommand{\T}{\mathbb T} 
 
\newcommand{\R}{\mathbb R} 

\newcommand{\dist}{{\rm dist}}

\renewcommand{\div}{{\rm div \,}}

\newcommand{\id}{\mathrm{id}}
\newcommand{\sign}{\mathrm{sign}}

\renewcommand{\H}{{\mathcal H}}
\newcommand{\E}{{\mathcal E}}

\newcommand{\A}{{\mathcal A}}

\newcommand{\D}{{\mathcal D}}

\newcommand{\C}{{\mathcal C}}
\newcommand{\F}{{\mathcal F}}

\newcommand{\Eeff}{\E^{\mathsf{eff}}_\eps}
\newcommand{\Eeps}{\E_\eps}

\newcommand{\Ra} {\Rightarrow}
\newcommand{\wto}{\rightharpoonup}

\renewcommand{\d}{\,\mathrm{d}}

\newcommand{\dx}{\,\mathrm{d}x}
\newcommand{\dy}{\,\mathrm{d}y}
\newcommand{\dz}{\,\mathrm{d}z}
\newcommand{\ds}{\,\mathrm{d}s}
\newcommand{\dt}{\,\mathrm{d}t}
\newcommand{\dr}{\,\mathrm{d}r}

\newcommand{\cc}{\Subset}

\newcommand{\eps}{\varepsilon}
\newcommand{\curl}{\mathrm{curl} \,}
\newcommand{\tr}{\mathsf{tr}}
\renewcommand{\S}{\mathcal{S}}
\let\phi = \varphi

\newcommand{\com}[1]{\textcolor{black}{#1}}
 
 \allowdisplaybreaks
 
\begin{document}

\title[Dislocation motion in simplified elasticity]{On the motion of curved dislocations in three dimensions: Simplified linearized elasticity}

\author[I. Fonseca]{Irene Fonseca}
\address{Irene Fonseca\\Department of Mathematical Sciences\\Carnegie-Mellon University\\5000 Forbes Avenue\\ Pittsburgh, PA 15213}
\email{fonseca@andrew.cmu.edu}

\author[J. Ginster]{Janusz Ginster}
\address{Janusz Ginster\\Institut f\"{u}r Mathematik, Humboldt-Universit\"{a}t zu Berlin, Rudower Chaussee 25, 12489 Berlin}
\email{janusz.ginster@math.hu-berlin.de}

%\author[E. O'Brien]{Ethan O'Brien}
%\address{Ethan O'Brien\\Department of Mathematical Sciences\\Carnegie-Mellon University\\5000 Forbes Avenue\\ Pittsburgh, PA 15213}
%\email{eobrien2@andrew.cmu.edu}

\author[S. Wojtowytsch]{Stephan Wojtowytsch}
\address{Stephan Wojtowytsch\\Program in Applied and Computational Mathematics\\Princeton University\\205 Fine Hall\\ Princeton, NJ 08544}
\email{stephanw@princeton.edu}

\date{\today}


\maketitle

\begin{abstract}
\com{It is shown} that in core-radius cutoff regularized simplified elasticity (where the elastic energy depends quadratically on the full displacement gradient rather than its symmetrized version), the force on a dislocation curve by the negative gradient of the elastic energy asymptotically approaches the mean curvature of the curve as the cutoff radius \com{converges} to zero. Rigorous error bounds in H\"older spaces \com{are provided}.

As an application, convergence of dislocations moving by the gradient flow of the elastic energy to dislocations moving by the gradient flow of the arclength functional, when the motion law is given by an $H^1$-type dissipation, and convergence to curve shortening flow in co-dimension $2$ for the usual $L^2$-dissipation \com{is established}. In the second scenario, existence and regularity \com{are assumed} while the $H^1$-gradient flow is treated in full generality (for short time).

The methods developed here are a blueprint for the more physical setting of linearized isotropic elasticity.
\end{abstract}

\tableofcontents

\newpage

\section{Introduction}
While we typically consider crystalline materials as periodic grids, real crystals have defects. An important class among them are the line defects known as dislocations.

The creation and motion of crystal dislocations is the key mechanism for the plastic deformation of crystalline solids and has received substantial attention in both the mathematical and engineering communities. Except for the recent article \cite{hudson}, most rigorous efforts have been focused on special cases concerning either straight dislocations orthogonal to a plane (see, for example, \cite{MR2192291,Gi18,MR2989446,FaPaPo18} for the static setting and \cite{MR3323554,MR3639276} for the dynamics) or curved dislocations in a plane (plane models, e.g.,\ \cite{MR1935021}). \com{The study of the} the motion of crystal dislocations have been made on several different scales, from atomistic simulations  to the evolution of dislocation densities on a continuum scale \com{(see}\cite{MR3811366,MR3043571,MR3453936}\com{),} and in between, \com{see(}\cite{MR3621813}\com{)}.

We consider the intermediate regime of discrete dislocation lines in a crystal described by a continuum model of elasticity. \com{However, being an} atomistic phenomenon, \com{here the description of dislocations requires the introduction of a small parameter $\eps>0$} coupled to the grid scale, together with the postulate that the crystalline material is well described by a linear continuum theory of elasticity when further away from the dislocation line than $\eps$ and that the energy in the region of the material closer to the dislocation can be neglected (a core-radius regularization), see \cite{MR2338415} for a justification on the discrete level in the case of screw dislocations. Asymptotics for such energies as the regularizing parameter $\eps$ tends to $0$ have been obtained in the language of $\Gamma$-convergence in \cite{MR3396435,MR3375538}.

In this article we consider a model of {\em simplified linearized elasticity} in which the energy of a displacement $u$ of a body \com{with reference configuration} $\Omega$ is given by the Dirichlet energy 
\[
\E(u) := \int_\Omega |\nabla u|^2\dx.
\]
We will consider \com{the fully linearized isotropic} case in future work \cite{nextone}. The presence of dislocations is expressed as a non-zero curl, and instead of the deformation gradient $\nabla u$, we consider \com{as elastic strain} a tensor-field $\beta$ which satisfies $\curl \beta = \mu$ where $\mu$ is the Nye dislocation measure of the dislocation (see below). The measure $\mu$ is concentrated on the dislocation line $\gamma$, and thus $\beta$ cannot be $L^2$-regular and the elastic energy would be infinite. To compensate for this blow-up, we remove an $\eps$-tubular neighborhood around the curve $\gamma$ from our domain and consider the modified elastic energy
\[
\E_\eps(\beta) := \int_{\Omega\setminus B_\eps(\gamma)} |\beta|^2\dx \com{.}
\]
Assuming that the elastic energy is minimized given the dislocation line (elastic equilibrium under plastic side condition), we can \com{then} associate an elastic energy to the dislocation itself. \com{Under the assumption of a} slow movement of the dislocation $\gamma$ compared to the elastic relaxation time in the crystal, we can \com{expect} that $\gamma$ moves by the gradient flow of this elastic energy.

In this article we \com{find} an {\em effective energy} $\E_\eps^{\rm eff}(\gamma)$ which simplifies the non-local interaction between the boundary of the material body $\partial\Omega$ and the dislocation line, and then obtain asymptotic expansions for the energy $\E_\eps^{\rm eff}(\gamma)$ in $\eps$\com{.} \com{In turn, this allows us to deduce} the negative gradient of the effective energy with respect to the variation of the dislocation line, the so-called Peach-Koehler force $F^{PK,\eps}$. Precisely, we show that $\E_\eps^{\rm eff}(\gamma)$ \com{converges} to the length of the curve $\gamma$ (suitably rescaled) and that the Peach-Koehler force approaches the mean curvature of $\gamma$, also rescaled\com{, i.e.,}
\[
F^{PK,\eps}(\gamma) \approx \vv H_\gamma + R^{PK,\eps}
\]
where $R^{PK,\eps}$ is small in $L^\infty(\gamma)$ and bounded in $C^{0,\alpha}(\gamma)$ if $\gamma$ is a $C^{2,\alpha}$-curve -- the precise result can be found in Theorem \ref{theorem asymptotics force}. Unfortunately, the bounds we obtain are not strong enough to pass to the limit in the $L^2$-gradient flows of the energies $\E_\eps^{\rm eff}$ since the remainder term is only bounded in the critical H\"older space, but not necessarily small. However, we prove that if solutions to the gradient flow of $\E_\eps$ exist and remain regular up to a given time, do not approach the domain boundary $\partial\Omega$ or develop self-intersections, then these solutions \com{converge to} a solution to curve-shortening flow, which is the gradient flow of the energy limit.

Unconditionally, we prove \com{short-time} existence, regularity and convergence for solutions to a modified gradient flow where the dissipation is given by an $H^1$-type inner product similar to the one used in \cite{hudson}. This converts the PDE into a Banach-space valued ODE 
\[
\partial_t\Gamma = \left(1- \Delta_\Gamma\right)^{-1} F^{PK,\eps}(\Gamma)
\]
which can easily be solved since the inverse curve Laplacian regularizes the Peach-Koehler force sufficiently. With this regularized Peach-Koehler force, we can also include non-linear mobilities which distinguish between edge and screw dislocations (the tangent vector to the dislocation line $\gamma$ is orthogonal/parallel to a fixed vector $b\in \R^3$ respectively, the so called Burgers vector). The main results on the asymptotics for dislocation dynamics are given in Theorem \ref{theorem motion l2} and Corollary \ref{corollary motion h1}, with further extensions in Remarks \ref{remark physical dissipation} and \ref{remark physical h1}.

The proofs are based on a careful \com{decomposition $\beta$ as } $\beta = \S + \nabla u$ where $\S$ is the strain due to the presence of the dislocation in an infinite crystal and $u$ encodes the interaction of $\beta$ with the domain $\Omega$. We can write $\S$ as the convolution of an explicit kernel $k$ with the Nye measure $\mu$ on $\gamma$. Explicit formulas are obtained by using the second order Taylor expansion of $\gamma$ and estimating the error terms. The Peach-Koehler force is thus {\em doubly non-local} in that $\S$ depends non-locally on $\gamma$ and the force itself is an integral of terms involving $\S$ and $u$ over the boundary of the $\eps$-tubular neighborhood around $\gamma$.

This article is structured as follows. In Section \ref{section preliminaries} we collect some notations which will be used in the article and may not be standard. Section \ref{section energy} is devoted to the study of the elastic energy, the effective energy and its variation. In Section \ref{section asymptotics} we study the asymptotic behavior of the strain $\beta$ as $\eps\to0$ in $L^\infty$ and H\"older spaces for regular curves. \com{Then we} use our results to obtain an asymptotic expansion of the Peach-Koehler force. \com{In Section \ref{section evolution}} we show that also the gradient flows of the energies converge (with the aforementioned caveats), and \com{we} discuss our results and future work in Section \ref{section conclusion}.

\section{\com{Notations and Preliminaries}} \label{section preliminaries} Let us briefly list the notations and concepts used in this article which are slightly non-standard.

{\bf Embeddedness radius of a curve.} It is well-known that an embedded $C^2$-curve in $\R^3$ has a tubular neighborhood, i.e.\ a neighborhood of the form $B_r(\gamma):= \{x\in \R^3\:|\:\dist(x,\gamma)<r\}$, for some $r>0$, which is diffeomorphic to $S^1 \times D_1$ (where $D_1$ is the unit disk in $\R^2$) via a map
\[
\Phi: S^1 \times D_1 \to \R^3, \qquad \Phi(s, \rho,\sigma) := \gamma(s) + r\left(\rho\,n_1(s) + \sigma\,n_2(s)\right)
\]
for two normal vector fields to $\gamma$, \com{$n_1, n_2$}. We call the supremum of all radii for which such a diffeomorphism exists the {\em embeddedness radius of $\gamma$}. Further details can be found in Appendix \ref{appendix embeddedness} where we also show that the embeddedness radius is lower-semi continuous under $C^{2,\alpha}$-convergence of curves.

{\bf Notation.} \com{Below we list some notatins used in the sequel.}

\vspace{2ex}

\begin{tabular}{l|l}
$X(A; B)$ &\com{function} space $X$ (e.g.,\ $L^p, C^{0,\alpha}, H^k, \dots$) with domain $A$ and codomain $B$\\
$\D'$ &\com{space} of distributions\\
$a\wedge b$ &\com{cross} product of two vectors $a,b$ in $\R^3$\\
$A\cc B$ & $A$ is compactly contained in $B$, i.e., $A\subseteq\overline A\subseteq B$ and $\overline A$ is compact\\
$C^{2,\alpha}_{1, \alpha/2}$ &\com{parabolic} H\"older space of functions with two continuous spatial derivatives\\
	&  and one continuous time derivative such that $\partial_tu$ and $D^2u$ are simultaneously\\
	& $C^{0,\alpha}$-continuous in space and $C^{0,\alpha/2}$ in time\\
$f'$ &derivative of $f$ with respect to a spatial scalar variable (usually called $s$)\\
$\dot f$ &time-derivative of $f$\\
\com{$\Omega$} &\com{bounded Lipschitz-domain in $\R^3$}\\
\com{$b$} &\com{Burgers vector in $\R^3$ which is fixed throughout the paper}\\
$\gamma$ &dislocation curve embedded in a domain $\Omega$\\
$\tau$ &unit tangent vector $\frac{\gamma'}{|\gamma'|}$ of a curve $\gamma$\\
$\vv H$ & curvature vector of a curve\\
$\H^s$ & $s$-dimensional Hausdorff measure\\
$\mu$ &Nye dislocation measure $b\otimes \tau\cdot\H^1|_\gamma$\\
$\beta$ &elastic \com{strain}\\
$\S$ &singular strain associated to $\mu$\\
$u$ &function such that $\beta = \S + \nabla u$\\
$B_\eps(\gamma)$ &\com{$\eps$-tube} around the curve $\gamma$\com{, often} abbreviated as $B_\eps$\\
$\pi$ &\com{closest} point projection from $B_r(\gamma)$ to $\gamma$ \\
$\nu_\eps$ &exterior normal vector to the tubular neighborhood $\partial B_\eps(\gamma)$\\
$\Omega_\eps$ &the set $\Omega\setminus B_\eps(\gamma)$\\
$\Gamma$ &a family of space curves evolving in time\\
k &gradient of the Newton kernel $G$ in three dimensions
\end{tabular}

\vspace{2ex}

{\bf Neumann problems.} We recall a result on elliptic Neumann problems which we will need later on. Let $f\in L^2(\Omega)$ and $g\in L^2(\partial\Omega)$. We say that a function $u\in H^1(\Omega)$ solves the Neumann problem
\[
\begin{pde} - \Delta u &=f &\text{in }\Omega\\
	\partial_\nu u &= g &\text{on }\partial\Omega
\end{pde}
\]
on the bounded Lipschitz-domain $\Omega \com{\subset \R^3}$ if $\int_\Omega u= 0$ and one of the two equivalent conditions is met:
\begin{enumerate}
\item $u$ minimizes the energy $\int_\Omega \frac 12\,|\nabla u|^2 -uf\dx - \int_{\partial\Omega} ug\d\H^{2}$. \label{eq: Neumann variational}
\item $\int_\Omega \langle \nabla u, \nabla \phi\rangle\dx = \int_\Omega f\phi\dx + \int g\phi\d\H^{\com{2}}$ for all $\phi \in H^1(\Omega)$.
\end{enumerate}
Considering the energy competitor $v=0$ \com{in \eqref{eq: Neumann variational}} shows that the minimum energy is non-positive, and therefore 
\begin{align}
\int_\Omega |\nabla u|^2\dx &\com{\leq} 2\int_\Omega uf \dx + 2\int_{\partial\Omega}ug\d\H^{\com{2}} \nonumber\\
	&\leq 2\,\|u\|_{L^2(\Omega)}\|f\|_{L^2(\Omega)} + 2\|u\|_{L^2(\partial\Omega)}\|g\|_{L^2(\partial\Omega)} \nonumber\\
	&\leq C\left[\|f\|_{L^2(\Omega)} + \|g\|_{L^2(\partial\Omega)}\right]\,\|\nabla u\|_{L^2(\Omega)}, \label{eq: Neumann}
\end{align}
\com{and so} $\|u\|_{H^1(\Omega)}\leq C\left[\|f\|_{L^2(\Omega)} + \|g\|_{L^2(\partial\Omega)}\right]$, where the constant $C$ depends on the Poincar\'e- and trace-constants of the domain $\Omega$.

If the boundary of $\Omega$ is $C^2$-regular, standard regularity estimates imply that $u\in H^2(\Omega)$ and $\partial_\nu u= g$ in the sense of traces.

{\bf Further conventions.} The vector calculus operators $\curl$ and $\div$ are applied row-wise to matrices. In the following, we will not distinguish between a parameterized curve, its reparameterizations and its trace.

\com{In addition, our proofs we need several results which we introduce in appendices:}
\begin{itemize}
\item \com{For the derivation of the effective energy, we require uniform trace and Poincar\'e constants on the domains $\Omega_\eps = \Omega\setminus B_\eps(\gamma)$ given by the physical body $\Omega$ without a tubular neighborhood of the dislocation. This is the content of Appendix \ref{appendix domain with hole}.}

\item \com{Both in the proof of these results and obtaining the asymptotic expansion of the Peach-Koehler force, we require cylindrical coordinates for the tubular neighborhood around $\gamma$ which we remove and particularly its boundary. These coordinates are introduced in Appendix \ref{appendix embeddedness}.}

\item \com{Finally, in Appendix \ref{appendix l2 h1} we collect a few results on $L^2$- and $H^1$-gradient flows and the equation $(1-\delta\Delta) u= f$ for small $\delta$.}
\end{itemize}

\section{Elastic energy and Peach-Koehler force}\label{section energy}

\subsection{The elastic energy}

We consider the minimum of the elastic energy of deformations with a prescribed topological defect in the form of a fixed dislocation loop $\gamma$ in a bounded, \com{simply-connected} Lipschitz domain $\Omega\subset\R^3$. For technical reasons, we will in addition assume that $\Omega$ is simply connected. 

Let $\gamma: S^1 \rightarrow \Omega$ be a regular, closed Lipschitz curve , and let $b\in \R^3$ be the Burgers vector of the dislocation. We denote by $\gamma = \gamma(S^1)$ also the trace of the curve in $\R^3$ and we do not distinguish between reparametrizations. The Nye dislocation measure of the dislocation is then given by
\[
\mu_{\gamma} := b \otimes \tau \, \H^1_{\com{|}\gamma}
\]
where $\tau: \gamma \rightarrow S^2$, $\tau_{\gamma(t)} := \frac{\gamma'(t)}{|\gamma'(t)|}$ is the unit tangent to $\gamma$. We define the set of associated admissible strains of the crystal $\Omega$ by
\[
\mathcal{A}(\mu_{\gamma}) := \{  \beta \in L^1(\Omega; \R^{3\times 3}) \cap L^2_{\com{\rm{loc}}}(\Omega \setminus \gamma; \R^{3\times 3}): \curl \beta = \mu_{\gamma} \text{ in }\D'\left(\Omega\times \R^{3\times 3}\right) \}\com{.}
\]
Note that \com{the} condition $\curl \beta = \mu_{\gamma}$ is not compatible with \com{prescribing $\beta \in L^2(\Omega;\R^{3\times3})$}. 
In fact, we will show later that \com{if} $\beta \in \mathcal{A}(\mu_{\gamma})$ \com{then} $\int_{\Omega \setminus B_{\eps}(\gamma)} |\beta|^2 \, dx \gtrsim |\log \eps|$ \com{for every $\eps>0$}. 
\com{On the other hand}, a classical linearized elastic energy is quadratic in the strain\com{, so in order to obtain a finite} energy we use a standard core-radius cut-off approach. We consider a `simplified linearized elasticity' for the stored elastic energy induced by a dislocation given by the dislocation density $\mu_{\gamma}$, i.e., we set
\begin{align}\label{eq: energyeps}
\Eeps (\mu_{\gamma}) := \inf\bigg\{ \frac12 \int_{\Omega_{\eps}(\gamma)} |\beta|^2 \dx\:\bigg|\: \beta \in \mathcal{A}(\mu_{\gamma}) \bigg\}\com{,}
\end{align}
where 
\[
\Omega_{\eps}(\gamma) := \Omega \setminus B_{\eps}(\gamma) \text{ \com{with} } \com{B_{\eps}(\gamma) = }\{x\in \Omega \:|\: \dist(x,\gamma)<\eps\}.
\]
When no confusion is possible, we will also simply write $\Omega_\eps = \Omega_\eps(\gamma)$. Note that since $\gamma$ is a compact curve in $\Omega$, there is a positive distance between $\gamma$ and $\partial \Omega$, \com{and so} for small enough $\eps>0$ also $B_\eps(\gamma)\cc\Omega$ \com{and} $\partial\Omega_\eps = \partial \Omega \cup \partial B_\eps(\gamma)$.

In this energy \com{\eqref{eq: energyeps}}, $\beta$ plays the role of the deformation gradient, but cannot be a gradient for topological reasons due to the presence of the dislocation $\gamma$ where plastic `slip' occurs, encoded by the non-trivial curl. Heuristically, we can imagine an extra half-plane of atoms wedged in on one side of $\gamma$, but not the other one, see Figure \ref{fig: straightdislocationline}.

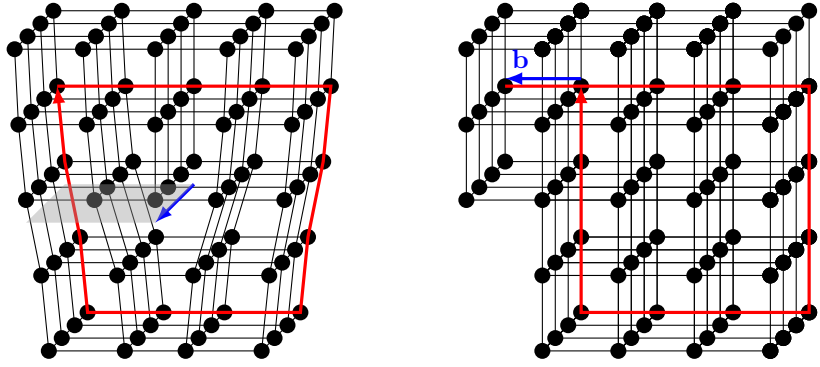
\begin{figure}[h]
 \centering
\begin{tikzpicture}[scale =1]

\foreach \t in {0,-0.17,-0.34,-0.51}
{
\draw (1+\t,1+\t) -- (4+\t,1+\t);
\fill (1+\t,1+\t) circle (3pt);
\fill (2+\t,1+\t) circle (3pt);
\fill (3+\t,1+\t) circle (3pt);
\fill (4+\t,1+\t) circle (3pt);

\draw (1.1+\t,0+\t) -- (3.9+\t,0+\t);

\draw (1.1+\t,0+\t) -- (1+\t,1+\t);
\fill (1.1+\t,0+\t) circle (3pt);

\draw (2.1+\t,0+\t) -- (2+\t,1+\t);
\fill (2.1+\t,0+\t) circle (3pt);

\draw (2.9+\t,0+\t) -- (3+\t,1+\t);
\fill (2.9+\t,0+\t) circle (3pt);

\draw (3.9+\t,0+\t) -- (4+\t,1+\t);
\fill (3.9+\t,0+\t) circle (3pt);

\draw (0.9+\t,2+\t) -- (4.3+\t,2+\t);
\fill (0.8+\t,2+\t) circle (3pt);
\draw (0.8+\t,2+\t) -- (1+\t,1+\t);
\fill (1.7+\t,2+\t) circle (3pt);
\fill (2.5+\t,2+\t) circle (3pt);
\draw (1.7+\t,2+\t) -- (2+\t,1+\t);
\fill (3.3+\t,2+\t) circle (3pt);
\draw (3.3+\t,2+\t) -- (3+\t,1+\t);
\fill (4.2+\t,2+\t) circle (3pt);
\draw (4.2+\t,2+\t) -- (4+\t,1+\t);

\draw (0.7+\t,3+\t) -- (4.4+\t,3+\t);
\fill (0.7+\t,3+\t) circle (3pt);
\draw (0.7+\t,3+\t) -- (0.8+\t,2+\t);
\fill (1.6+\t,3+\t) circle (3pt);
\draw (1.6+\t,3+\t) -- (1.7+\t,2+\t);
\fill (2.5+\t,3+\t) circle (3pt);
\draw (2.5+\t,3+\t) -- (2.5+\t,2+\t);
\fill (3.4+\t,3+\t) circle (3pt);
\draw (3.4+\t,3+\t) -- (3.3+\t,2+\t);
\fill (4.3+\t,3+\t) circle (3pt);
\draw (4.3+\t,3+\t) -- (4.2+\t,2+\t);

\draw (0.65+\t,4+\t) -- (4.35+\t,4+\t);
\fill (0.65+\t,4+\t) circle (3pt);
\draw (0.65+\t,4+\t) -- (0.7+\t,3+\t);
\fill (1.55+\t,4+\t) circle (3pt);
\draw (1.55+\t,4+\t) -- (1.6+\t,3+\t);
\fill (2.5+\t,4+\t) circle (3pt);
\draw (2.5+\t,4+\t) -- (2.5+\t,3+\t);
\fill (3.45+\t,4+\t) circle (3pt);
\draw (3.45+\t,4+\t) -- (3.4+\t,3+\t);
\fill (4.35+\t,4+\t) circle (3pt);
\draw (4.35+\t,4+\t) -- (4.3+\t,3+\t);
}

\draw (1,1) -- (0.49,0.49);
\draw (2,1) -- (1.49,0.49);
\draw (3,1) -- (2.49,0.49);
\draw (4,1) -- (3.49,0.49);

\draw (1.1,0) -- (1.1-0.51,-0.51);
\draw (2.1,0) -- (2.1-0.51,-0.51);
\draw (2.9,0) -- (2.9-0.51,-0.51);
\draw (3.9,0) -- (3.9-0.51,-0.51);

\draw (0.8,2) -- (0.8-0.51,2-0.51);
\draw (1.7,2) -- (1.7-0.51,2-0.51);
\draw (2.5,2) -- (2.5-0.51,2-0.51);
\draw (3.3,2) -- (3.3-0.51,2-0.51);
\draw (4.2,2) -- (4.2-0.51,2-0.51);

\draw (0.7,3) -- (0.7-0.51,3-0.51);
\draw (1.6,3) -- (1.6-0.51,3-0.51);
\draw (2.5,3) -- (2.5-0.51,3-0.51);
\draw (3.4,3) -- (3.4-0.51,3-0.51);
\draw (4.3,3) -- (4.3-0.51,3-0.51);

\draw (0.65,4) -- (0.65-0.51,4-0.51);
\draw (1.55,4) -- (1.55-0.51,4-0.51);
\draw (2.5,4) -- (2.5-0.51,4-0.51);
\draw (3.45,4) -- (3.45-0.51,4-0.51);
\draw (4.35,4) -- (4.35-0.51,4-0.51);

\draw[very thick,blue,>=latex,<-] (2.5-0.51,1.7-0.51) -- (2.5,1.7);
\draw[red,very thick,>=latex,<-] (0.7,3) -- (0.8,2) -- (1,1) -- (1.1,0) -- (3.9,0) -- (4,1) -- (4.2,2) -- (4.3,3) -- (0.7,3);
\fill[gray!80!white,opacity=0.4] (2.5-0.51,1.7-0.51) -- (2.5,1.7) -- (0.8,1.7) -- (0.8-0.51,1.7-0.51) -- cycle;
\end{tikzpicture}
\quad \quad \quad \quad
\begin{tikzpicture}[scale=1]
\foreach \t in {0,-0.17,-0.34,-0.51}{
\foreach \x in {1,2,3}
\foreach \y in {0,1,2,3}{
\draw (\x+\t,\y+\t) -- (\x+\t,\y+\t+1);
\draw (\x+\t,\y+\t) -- (\x+\t+1,\y+\t);
\draw (\x+\t+1,\y+\t) -- (\x+1+\t,\y+1+\t);
\draw (\x+\t,\y+1+\t) -- (\x+1+\t,\y+1+\t);
\fill (\x+\t,\y+\t) circle (3pt);
\draw (\x,\y) -- (\x-0.51,\y-0.51);
\fill (4+\t,\y+\t) circle (3pt);
\draw (4,\y) -- (4-0.51,\y-0.51);
\fill (\x+\t,4+\t) circle (3pt);
\draw (\x,4) -- (\x-0.51,4-0.51);
}
\fill (0+\t,4+\t) circle (3pt);
\fill (0+\t,3+\t) circle (3pt);
\fill (0+\t,2+\t) circle (3pt);
\draw (0+\t,2+\t) -- (0+\t,4+\t);
\draw (0+\t,4+\t) -- (1+\t,4+\t);
\draw (0+\t,3+\t) -- (1+\t,3+\t);
\draw (0+\t,2+\t) -- (1+\t,2+\t);
\fill (4+\t,4+\t) circle (3pt);
}
\draw (4,4) -- (4-0.51,4-0.51);
\draw (0,4) -- (0-0.51,4-0.51);
\draw (0,3) -- (0-0.51,3-0.51);
\draw (0,2) -- (0-0.51,2-0.51);

\draw[very thick, red, ->,>=latex] (0,3) -- (4,3) -- (4,0) -- (1,0) -- (1,3);
\draw[very thick, blue, ->,>=latex] (1,3.1) -- (0,3.1);
\draw (0.2,3.1) node[anchor=south, blue,very thick] {$\mathbf{b}$};
\end{tikzpicture}
\caption{Sketch of a straight edge dislocation line in a three-dimensional cubic lattice. Left: Deformed configuration with dislocation. Right: Reference configuration with Burgers vector $b$.} \label{fig: straightdislocationline}
\end{figure}

\com{In physical units, the Burgers vector is proportional to the lattice, the vector $b$ is the normalized Burgers vector and of order $O(1)$. So the physical energy associated to a single dislocation with normalized Burgers vector $b$ would be multiplied by the square of the lattice constant, a scalar multiple of $\eps$.}

\begin{remark}
\com{In} linear elasticity theories for crystal dislocations \com{the energy becomes}
\[
\Eeps (\mu_{\gamma}) := \inf\bigg\{ \frac12 \int_{\Omega_{\eps}(\Gamma)} \C\beta:\beta \dx: \beta \in \mathcal{A}(\mu_{\gamma}) \bigg\}
\]
for a tensorial map $\C:\R^{3\times 3}\to \R^{3\times 3}$ which associates the stress $\C\beta$ to the strain $\beta$. In our simplified case, $\C=\id$\com{, while} isotropic elastic tensors have the form
\[
\C_{ijkl} = \lambda\,\delta_{ij}\delta_{kl} + \mu\big(\delta_{ik}\delta_{jl} + \delta_{il}\delta_{jk}\big)
\]
leading to an elastic energy
\[
\Eeps^{\lambda,\mu} (\mu_{\gamma}) := \inf\bigg\{ \frac12 \int_{\Omega_{\eps}(\gamma)} \left(2\mu\,\left|\frac{\beta+\beta^T}2\right|^2 + \lambda\,\tr(\beta)^2\right) \dx: \beta \in \mathcal{A}(\mu_{\gamma}) \bigg\}.
\]
Assuming \com{that} the {\em Lam\'e moduli} $\lambda,\mu$ satisfy the ellipticity conditions $\mu>0, \lambda+2\mu>0$, this energy is bi-Lipschitz equivalent to that of simplified linearized elasticity \com{\eqref{eq: energyeps}} due to Korn's inequality, and many of the same methods are applicable. This setting will be the topic of future \com{work} \cite{nextone}.
\end{remark}

\com{As a first step toward understanding the energy \eqref{eq: energyeps}} we recall how to construct a solution to the Euler-Lagrange equations to the whole space problem
\[
\begin{pde} \div \S &= 0 &\text{ in } \R^3, \\ \curl \S &= \mu_{\gamma} &\text{ in } \R^3.\end{pde}
\]
Notice that the closedness of $\gamma$ implies that $\div \mu_{\gamma} = 0$ in the distributional sense.
Then, due to \cite{MR2293957,MR3375538}, a distributional solution $\S \in L^{\frac32}(\R^3;\R^{3\times 3}) \cap L^{2}_{\com{\rm{loc}}}(\R^3\setminus \gamma;\R^{3\times 3})$ is given by
$\S := (-\Delta)^{-1} \curl \mu_{\gamma}$, which is verified in a formal computation
\begin{align*}
\div\S &= \div  (-\Delta)^{-1}\curl \mu_\gamma
	= (-\Delta)^{-1}\div\curl\mu_\gamma
	= (-\Delta)^{-1}\,0
	=0
\end{align*}
and
\begin{align*}
\curl\S &= \curl (-\Delta)^{-1}\curl \mu_\gamma\\
	&= (\curl\circ\curl)\,(-\Delta)^{-1}\mu_\gamma\\
	&= (\nabla \circ \div - \Delta)(-\Delta)^{-1}\mu_\gamma\\
	&= \com{(-\Delta)^{-1}\nabla\left(\div \mu_\gamma + (-\Delta)(-\Delta)^{-1}\mu_\gamma \right)}\\
	&= \com{(-\Delta)^{-1}\,\nabla \,0 + \mu_\gamma} \\
	&= \mu_\gamma.
\end{align*}
It is well known that the inverse of the negative Laplacian (vanishing at infinity) is given by convolution with the Newtonian kernel $G(x) = \frac{1}{4\pi\,|x|}$, and thus, formally,
\begin{align}
\S(x) &= (-\Delta)^{-1} \curl \mu_{\gamma} \nonumber\\
	&= \curl (-\Delta)^{-1}\mu_\gamma \nonumber\\
	&= \curl_x \int_{\R^3}G(x-y)\d\mu_\gamma(y) \nonumber\\
	&= \curl_x \int_{\gamma} G(x-y)\, b\otimes \tau_y \d\H^1_y \nonumber\\
	&= b\otimes \curl_x \int_\gamma G(x-y)\,\tau_y\d\H^1_y \nonumber\\
	&= b\otimes \int_\gamma \nabla G(x-y)\wedge\tau_y\d\H^1_y \nonumber\\
	&=b\otimes \int_{\gamma} k(x-y) \wedge \tau_y \d\mathcal{H}^1_y \label{eq: Sconvolution}
\end{align}
where $k(x) := \nabla G(x) = \frac{- x}{4\pi\,|x|^3}$.
 Clearly, the strain $\S$ is $C^\infty$ away from the curve $\gamma$. We will investigate the structure of $\S$ below in great detail in the case that $\gamma\in C^{2}(S^1;\R^3)$, which will also give a more precise information on the nature of the singularity.

We can now express the energy $\E_{\eps}(\mu_{\gamma})$ in \com{terms of $\S$}.
First we notice that every function $\beta \in \mathcal{A}(\mu_{\gamma})$ can be written as $\beta = \S + \nabla u$ for some $u \in W^{1,1}(\Omega; \R^3) \cap H^1(\Omega_{\eps}(\gamma);\R^3)$, and vice versa since $\S$ fixes the curl of $\beta$ and we assumed $\Omega$ to be simply connected.
\com{We have}
\begin{align*}
\int_{\Omega_{\eps}(\gamma)} \frac12|\beta|^2 \dx &= \frac12 \int_{\Omega_{\eps}(\gamma)} |\S|^2 \dx + \frac12 \int_{\Omega_{\eps}} |\nabla u|^2 \dx + \int_{\Omega_\eps(\gamma)} \langle \S, \nabla u\rangle\dx\\
	&= \frac12 \int_{\Omega_{\eps}(\gamma)} |\S|^2 \dx + \frac12 \int_{\Omega_{\eps}(\gamma)} |\nabla u|^2 \dx + \int_{\Omega_\eps(\gamma)} \left(\div(\S u) - \langle \div \S,u\rangle \right) \dx\\
	&= \frac12 \int_{\Omega_{\eps}(\gamma)} |\S|^2 \dx + \frac12 \int_{\Omega_{\eps}(\gamma)} |\nabla u|^2 \dx + \int_{\partial \Omega} u \cdot \S \nu \d\mathcal{H}^2 - \int_{\partial B_{\eps}(\gamma)} u \cdot \S \nu_{\eps} \d\mathcal{H}^2,
\end{align*}
where $\nu_{\eps}$ denotes the exterior unit normal to $B_\eps(\gamma)$ on $\partial B_\eps(\gamma)$ (i.e., the interior unit normal to $\Omega_\eps$), $\nu$ the exterior normal to $\Omega$ on $\partial \Omega$\com{, and where we used the fact that $\div \S = 0$}. The integrals are well-defined since $\S$ is smooth and bounded away from $\gamma$ and $u \in H^1(\Omega_\eps)$ has a trace on the boundary.
This motivates the definition of the domain-dependent elastic energy
\[
I_{\eps}(u) := \frac12 \int_{\Omega_{\eps}(\gamma)} |\nabla u|^2 \, dx + \int_{\partial \Omega} u \cdot \S \nu \, d\mathcal{H}^2 - \int_{\partial B_{\eps}(\gamma)} u \cdot \S \nu_{\eps} \, d\mathcal{H}^2
\]
for $u\in H^1(\Omega_\eps(\gamma); \R^3)\cap W^{1,1}(\Omega;\R^3)$. On the boundary $\partial \Omega_\eps = \partial\Omega \cup \partial B_\eps(\gamma)$, the function has to be understood in the sense of traces. Moreover, note that the values of $u$ and $\nabla u$ in $B_{\eps}(\gamma)$ do not contribute to the energy $I_{\eps}(u)$. On the other hand, every function $u \in H^1(\Omega_{\eps}(\gamma);\R^2)$ can be extended to a function $u \in H^1(\Omega;\R^2) \subseteq W^{1,1}(\Omega;\R^2)$. 
Therefore we may write
\begin{equation} \label{eq: I eps}
\E_{\eps}(\mu_{\gamma}) = \frac12 \int_{\Omega_{\eps}(\gamma)} |\S|^2 \, dx + \inf_{u \in H^1(\com{\Omega;\R^3})} I_{\eps}(u).
\end{equation}

\begin{Theorem}\label{theorem: existenceminimizer}
Let $b \in \R^3$ and \com{let} $\gamma$ \com{be} a closed, regular $C^2$-curve. Then \com{for all $\eps>0$}
\[
 \inf_{u \in H^1(\Omega_{\eps}(\gamma);\R^3)} I_{\eps}(u) =  \min_{u \in H^1(\Omega_{\eps}(\gamma); \R^3)} I_{\eps}(u).
\]
\com{Furthermore,} the minimizer $u_\eps$ of $I_\eps$ is unique up to the addition of a constant, and satisfies the Euler-Lagrange equation \com{with boundary conditions}
\begin{equation} \label{eq: ELE eps}
\begin{pde}
 -\Delta u_\eps &= 0 &\text{ in } \Omega_{\eps},\\
\com{\partial_{\nu_{\eps}}} u_\eps  &= -\S \nu_{\eps} & \text{ on } \partial B_{\eps}(\gamma),\\
\com{\partial_{\nu}} u_\eps &= -\S \nu &\text{ on } \partial \Omega.
\end{pde}
\end{equation}
\end{Theorem}

\begin{proof}
The proof \com{uses} a standard argument. First notice that since $\div \S = 0$, the energy $I_{\eps}$ is invariant under the transformation $u \rightarrow u + c$ for $c \in \R^3$. Hence, it is enough to consider the minimization in $\{u \in H^1(\Omega_{\eps};\R^3): \int_{\Omega_{\eps}} u = 0 \}$. Let $v \in H^1(\Omega_{\eps};\R^3)$ be such that $\int_{\Omega_{\eps}} v = 0$. Then
\begin{align*}
I_{\eps}(v) &\geq  \com{\frac12} \int_{\Omega_{\eps}} |\nabla v|^2 \, dx - \| v \|_{L^2(\partial \Omega_{\eps};\R^3)} \| \S \nu_{\eps} \|_{L^2(\partial \Omega_{\eps};\R^3)}\\
	&\geq c_{\eps}  \| v \|_{H^1(\Omega_{\eps};\R^3)}^2 - C_{\eps} \left(\| \S \nu_{\eps} \|_{L^2(\partial B_\eps(\gamma);\R^3)} + \| \S\nu\|_{L^2(\partial\Omega, \R^3)}\right)\, \| v \|_{H^1(\Omega_{\eps};\R^3)}.
\end{align*}
Here, $v$ \com{on $\partial \Omega_{\eps}$} has to be understood in the sense of traces. The constants $c_{\eps}, C_{\eps}>0$ stem from Poincar\'e's inequality and the trace operator, respectively. 

Hence, for fixed $\eps>0$ the functional $I_{\eps}$ is coercive. As \com{$I_{\eps}$ is a} strongly continuous, convex function, it is weakly lower-semicontinuous, and thus a minimizer exists. The equation \eqref{eq: ELE eps} is simply the corresponding Euler-Lagrange-equation whose solution is unique in the subspace of functions with vanishing average \com{(cf.~estimate \eqref{eq: Neumann})}, and thus \com{it is unique in $H^1(\Omega_{\eps};\R^3)$} up to the addition of a constant vector. Since $\S$ is $C^\infty$-smooth away from $\gamma$\com{, cf.~\eqref{eq: Sconvolution},} standard regularity theory for Neumann problems on $C^2$-domains implies that $u_\eps$ is an $H^2$-function up to the boundary and the boundary condition is therefore well-defined in the sense of traces.
\end{proof}

\subsection{The effective energy}

We are now interested in the behavior of the energy $\Eeps$ as $\eps\to 0$. Since $\S$ only depends on the curve $\gamma$, the dependence of the elastic energy on the domain $\Omega$ is solely encoded in the function $u_\eps$ and the energy $I_{\eps}$, \eqref{eq: I eps}. To simplify matters, we will take a partial limit $\eps\to 0$ only in $I_\eps$ (i.e.\ $u_\eps$). The same strategy was successfully used in \cite{MR3639276,MR2192291}.  Consider 
\begin{equation} \label{eq: def I}
I:H^1(\Omega;\R^3)\to \R, \qquad I(u) := \int_{\Omega}\frac12 |\nabla u|^2 \, dx + \int_{\partial \Omega} u \cdot \S \nu \,\mathcal{H}^2.
\end{equation}

\begin{Theorem}\label{theorem boundary term limit}
There exists a \com{unique (up to the addition of a constant)} function $u \in H^1(\Omega;\R^3)$ such that 
\[
\lim_{\eps \to 0} \min_{w \in H^1(\Omega_{\eps}(\gamma);\R^3)}I_{\eps}(w) =  I(u), 
\]
The function $u$ minimizes $I$, satisfies \com{the Neumann problem}
\[
\begin{pde}
-\Delta u & = 0 &\text{ in } \Omega, \\ \com{\partial_{\nu}} u &= - \S \nu &\text{ on } \partial \Omega
\end{pde}
\]
and the solutions $u_\eps$ to the $\Omega_\eps$-problem \eqref{eq: ELE eps} converge to $u$ strongly in $H^1(\Omega',\R^3)$ for all $\Omega'\cc\overline \Omega\setminus \gamma$. 
\end{Theorem}

Before proving Theorem \ref{theorem boundary term limit}, we gather a few results.

\begin{Proposition}\label{proposition boundary orthogonal}
Assume that $\gamma \in C^{2}(S^1;\com{\Omega})$. Then
\[
\lim_{\eps \to 0} \| \S \nu_{\eps} \|_{L^2(\partial B_{\eps}(\gamma);\R^3)} \rightarrow 0.
\]
\end{Proposition}

\begin{proof}
\com{We anticipate the results of Theorem \ref{thm: expansionK} where we obtain a rigorous asymptotic expansion of $\S$ \eqref{eq: expansionS}} which shows that 
\[
\|\S\nu_\eps\|_{L^\infty(\partial B_\eps(\gamma);\R^3)} = O(|\log\eps|).
\]
\com{The proof of the estimate is postponed until the next section. Let us emphasize that $\S = O(\eps^{-1})$ on $\partial B_\eps(\gamma)$ and that it is only the product $\S\nu_\eps$ which is bounded in $L^2(\partial B_\eps(\gamma))$. In addition, we claim that $\H^2(\partial B_\eps(\gamma)) = O(\eps)$.
This follows from the fact that $\gamma$ locally looks like a straight line and consequently $\partial B_\eps(\gamma)$ locally looks like a cylinder of radius $\eps$ -- in fact, more precisely $\lim_{\eps \to 0} \eps^{-1} \H^2(\partial B_\eps(\gamma)) = 2\pi\, \H^1(\gamma)$. \\
It follows that $\| \S \nu_{\eps} \|_{L^2(\partial B_{\eps}(\gamma);\R^3)} = O(\eps |\log \eps|)$.}
\end{proof}

\com{We will now argue that certain trace and extension constants on $\Omega_\eps$ are uniform in $\eps$. The key idea in proving the next two propositions is that, on a small scale, $\gamma$ looks like a straight line and $B_\eps(\gamma)$ looks like a cylinder. Proving the trace and extension result for a cylinder in three dimensions is very similar to proving them for a disk in two dimensions, where the Dirichlet energy is scale-invariant, leading to constants uniform in the radius of the cylinder.}

\com{Recall that the {\em embeddedness radius} of a curve $\gamma$ is in heuristic terms the largest radius for which the tubular neighborhood $\{\dist(\cdot,\gamma)< r\}$ is diffeomorphic to a torus -- see also Definition \ref{definition embeddedness radius} in the appendix.}

\begin{Proposition}\label{proposition uniform trace}
There exist $\eps_0>0$ \com{ and $C>0$} such that for all $0 < \eps < \eps_0$ the trace operator $H^1(\Omega_{\eps};\R^3) \rightarrow L^2(\partial B_{\eps}(\gamma);\R^3)$ satisfies
\[
\| u \|_{L^2(\partial B_{\eps}(\gamma);\R^3)} \leq C \| u \|_{H^1(\Omega_{\eps};\R^3)}.
\]
The constants $C$ and $\eps_0$ depend on the $C^2$-norm of $\gamma$ and its embeddedness radius.
\end{Proposition}

\com{Of course, since $\gamma$ does not touch $\partial\Omega$ and examining the local nature of the proof of trace inequalities, also the constants of the trace operator $H^1(\Omega_\eps) \to L^2(\partial\Omega)$ are uniform in $\eps<\eps_0$.}

\begin{Proposition}\label{proposition uniform extension}
There exist $\eps_0>0$ \com{and $C>0$} such that for all $0 < \eps < \eps_0$ there exists an extension operator $T_\eps: H^1(\Omega_{\eps};\R^3) \rightarrow H^1(\Omega ;\R^3)$ such that 
\[
\| T_\eps u \|_{H^1(\Omega;\R^3)} \leq C \| u \|_{H^1(\Omega_{\eps};\R^3)}
\]
for all $u\in H^1(\Omega_{\eps};\R^3)$. The constants $C$ and $\eps_0$ depend on the $C^2$-norm of $\gamma$ and its embeddedness radius.
\end{Proposition}

Similarly, if we remove a finite union of discs from a domain in $\R^2$, the Poincar{\'e}-Friedrich constant remains uniform -- see e.g.\ \cite[Proposition A.1]{MR2192291}. Also this transfers to our case.

\begin{Proposition}\label{proposition uniform poincare}
There exist $\eps_0>0$ and $C>0$ such that for all $0<\eps<\eps_0$ and $u\in H^1(\Omega_\eps;\R^3)$ the inequality
\[
\int_{\Omega_\eps} \left| u(x) - \frac{\int_{\Omega_\eps} u }{|\Omega_\eps|}\right|^2\dx \leq C \int_{\Omega_\eps}|\nabla u_\eps|^2\dx
\]
holds. The constants $C$ and $\eps_0$ depend on the $C^2$-norm of $\gamma$ and its embeddedness radius.
\end{Proposition}

\com{The proofs of \ref{proposition uniform trace}, \ref{proposition uniform extension} and \ref{proposition uniform poincare} can be found in the appendix \ref{appendix domain with hole}. We are now ready to prove this section's main result.}

\begin{proof}[Proof of Theorem \ref{theorem boundary term limit}]
Let $u_{\eps}$ be \com{the} minimizer of $I_{\eps}$ satisfying $\int_{\Omega_{\eps}} u_{\eps} \, dx = 0$. \com{As $v=0$ is always a competitor for $I_{\eps}$, it follows that $\sup_{\eps} I_{\eps}(u_{\eps}) \leq 0$.}
Using Proposition \ref{proposition uniform trace}, we observe that 
\begin{align*}
I_\eps(u_\eps) &= \frac12\int_{\Omega_\eps} |\nabla u_\eps|^2 \dx  + \int_{\partial\Omega} \langle u_\eps, \S\nu\rangle \d\H^2 - \int_{\partial B_\eps(\gamma)} \langle u_\eps, \S\nu_\eps\rangle\d\H^2\\
	&\geq \frac12 \int_{\Omega_\eps} |\nabla u_\eps|^2 \dx - \|u_\eps\|_{L^2(\partial\Omega)}\,\|\S\nu\|_{L^2(\partial\Omega)} - \|u_\eps\|_{L^2(\partial B_\eps)} \|\S\nu_\eps\|_{L^2(\partial B_\eps)}\\
	&\geq \frac12 \|\nabla u_\eps\|_{L^2(\Omega_\eps)}^2 - C\left(\|\S\nu\|_{L^2(\partial\Omega)} + \|\S\nu_\eps\|_{L^2(\partial B_\eps(\gamma))}\right)\,\| u_\eps\|_{H^1(\Omega_\eps)}.
\end{align*}
Now, applying Propositions \ref{proposition uniform poincare} and \ref{proposition boundary orthogonal}, we \com{have fo some $c>0$ independent from $\varepsilon$}
\[
I_\eps(u_\eps) \geq \left[c\,\|u_\eps\|_{H^1(\Omega_\eps)} - C\left(\|\S\nu\|_{L^2(\partial\Omega)} + 1\right)\right]\,\|u_\eps\|_{H^1(\Omega_\eps)}
\]
from which we obtain a uniform bound on $\|u_\eps\|_{H^1(\Omega_\eps)}$ independent of $0<\eps<\eps_0$. Using Proposition \ref{proposition uniform extension}, we extend $u_\eps$ to $\Omega$ in such a way that
\[
\|T_\eps u_\eps\|_{H^1(\Omega;\R^3)} \leq C\,\|u_\eps\|_{H^1(\Omega_\eps; \R^3)}.
\]
Since the sequence $T_\eps u_\eps$ is uniformly bounded in $H^1(\Omega;\R^3)$, there exists a function $u\in H^1(\Omega;\R^3)$ such that \com{(up to a subsequence)} $T_\eps u_\eps \wto u$ in $H^1(\Omega;\R^3)$. In particular, $\int_{\Omega} u \, dx = 0$.

We will now show that $u$ minimizes $I$. \com{Arguing by contradiction, assume that} there exists $v\in H^1(\Omega)$ such that $\int_\Omega v \dx = 0$ and $I(v)< I(u)$. Using the strict inequality, we can find $r>0$ such that
\[
I(v) < \frac12 \int_{\Omega_r}|\nabla u|^2 \dx + \int_{\partial\Omega}\langle u, \S\nu\rangle \d\H^2
\]
\com{where in $I(u)$} we replace the whole domain $\Omega$ by the smaller set $\Omega_r$. \com{As $u_{\eps}$ minimizes $I_{\eps}$ for $0<\eps<\eps_0$, we have}
\begin{align*}
\com{\limsup_{\eps \to 0} I_{\eps}(u_{\eps})} \leq \lim_{\eps\to 0} I_\eps(v) &= I(v)\\
	&< \frac12 \int_{\Omega_r}|\nabla u|^2 \dx + \int_{\partial\Omega}\langle u, \S\nu\rangle \d\H^2\\
	&\leq \liminf_{\eps\to0} \left[\frac12 \int_{\Omega_r}|\nabla u_\eps |^2 \com{\dx}+ \int_{\partial\Omega}\langle u_\eps, \S\nu\rangle \d\H^2\right]\\
	&\leq \liminf_{\eps\to0} \left[\frac12 \int_{\Omega_\eps}|\nabla u_\eps |^2\com{\dx} + \int_{\partial\Omega}\langle u_\eps, \S\nu\rangle \d\H^2\right]\\
	&\leq \liminf_{\eps\to0} I_\eps(u_\eps).
\end{align*}
\com{Above we used the facts that $u_\eps\to u$ weakly in $H^1(\Omega_r)$, $u_\eps\to u$ strongly in $L^2(\partial\Omega)$, $\|\S\nu\|_{L^2(\partial B_\eps)} \to0$ by Proposition \ref{proposition boundary orthogonal}, and thus}
\[
\com{\left|\int_{\partial B_\eps}\langle u, \S\nu\rangle \d\H^2\right| \leq C\,\|u_\eps\|_{H^1(\Omega_\eps)}\,\|\S\nu\|_{L^2(\partial B_\eps)} \to 0}.
\]
\com{This contradiction allows us to deduce} that $u$ is a minimizer of $I$. Since $I$ is convex and coercive, the minimizer $u$ is unique and  satisfies the corresponding Euler-Lagrange equations
\[
\begin{pde}
-\Delta u & = 0 &\text{ in } \Omega, \\ \partial_{\nu} u &= - \S \nu &\text{ on } \partial \Omega.
\end{pde}
\]
The same proof as above shows that $I(u) = \lim_{\eps\to 0}I_\eps(u_\eps)$. \com{In particular we have} $\int_{\Omega'} |\nabla u_\eps|^2\, dx \to \int_{\Omega'} |\nabla u|^2 \, dx$ for all $\Omega'\cc\overline \Omega\setminus \gamma$, \com{thus implying the strong convergence of $u_{\eps} \to u$ in $H^1(\Omega';\R^3)$.}
\end{proof}

\com{In view of Theorem \ref{theorem boundary term limit} we define the effective energy of a dislocation by}
\[
\Eeff (\mu_{\gamma})= \min_{u\in H^1(\Omega)} \left[\int_{\Omega_{\eps}(\gamma)} |\S|^2 \dx + I(u)\right].
\]
\com{In the future, we will always assume that $\S$ is associated to $\gamma$ through \eqref{eq: Sconvolution} and  $u=u_\gamma$ is the unique minimizer with zero mean of $I$, so}
\begin{equation} \label{eq: def eff energy}
\Eeff (\mu)= \int_{\Omega_{\eps}(\gamma)} \frac12|\S|^2 \dx + I(u).
\end{equation}

\subsection{Variation of the effective energy}

In this section we derive an expression of the self-force as the (negative) quasi-static variation of the effective energy with respect to the dislocation curve $\gamma$. \com{Physically, this means that we consider variations of the energy associated to the equilibrium stress-field $\beta$ given a defect measure $\mu$, i.e.\ variations along which the relaxation to equilibrium of the crystal is much faster than the movement of the dislocation curve.}
For this we consider a closed $C^2$-curve $\varphi: S^1\rightarrow \R^3$ and the associated dislocation densities
\[
\mu_t := b \otimes \com{\tau_t} \, \mathcal{H}^1|_{\gamma_t},
\]
where $\gamma_t = (\gamma + t \varphi)$ and $\tau_{\com{t}}(\gamma(s) + t \varphi(s)) = \frac{\gamma'(s) + t \varphi'(s)}{|\gamma'(s) + t \varphi'(s)|}$. 
In addition, we denote the associated functions by $\S_t := \S_{\gamma_t}$ and $u_t := u_{\gamma_t}$. Note that due to the regularity of $\gamma$ and $\phi$, the domain
\[
\Omega_T := \bigcup_{t\in (-T,T)} \{t\}\times \big(\Omega\setminus \gamma_t\big)
\]
is open. In this section, the variable $t$ plays the role of the variation parameter, while below it will be used for physical time.

Our first result is for the domain-independent singular strain $\S$. The variation of $\S$ with respect to $\gamma$ changes the curl of $\S$ in such a gradual way that it does not appear in the derivative. \com{We show that} the derivative of $\S$ with respect to the curve $\gamma$ is given by the spatial gradient of a function which depends on the direction of the variation. 

\com{First, we prove the following simple lemma.}

\begin{Lemma}\label{lemma wedge stuff}
Let $U\subset\R^3$ be open and $f\in C^1(U,\R^3)$, $v, w\in \R^3$. Then
\begin{align*}
\left[D(f\wedge v)\right] w - \left[D(f\wedge w)\right] v 
	%&= \frac12\nabla \left[(f\wedge v)\cdot w - (f\wedge w)\cdot v\right]' - \div(f)\,v\wedge w\\
	&= \nabla \left[(f\wedge v)\cdot w\right] - \div(f)\,v\wedge w\\
	&= - \nabla \left[(f\wedge w)\cdot v\right] - \div(f)\,v\wedge w\\
	&= \nabla \left[f\cdot (v\wedge w)\right] - \div(f)\,v\wedge w.
\end{align*}
\end{Lemma}

\begin{proof}
It suffices to satisfy the identity when $v, w\in \{e_1, e_2, e_3\}$ are basis vectors in $\R^3$ since both sides \com{of the string of equations} are bilinear \com{with respect to} the constant vectors. If $v=w$, both sides trivially reduce to \com{zero}, so without loss of generality, we may assume that $v=e_1$ and $w=e_2$. \com{We have}

\begin{tabular}{rll}
$f\wedge v$ & $= \begin{pmatrix} f_1\\ f_2\\ f_3\end{pmatrix} \wedge \begin{pmatrix} 1\\ 0 \\ 0\end{pmatrix}= \begin{pmatrix}0\\ f_3\\ -f_2\end{pmatrix}$\\
$\left[D(f\wedge v)\right]\cdot w$ &$= \begin{pmatrix} 0 &0&0\\ \partial_1f_3 &\partial_2 f_3&\partial_3f_3\\ -\partial_1f_2 & -\partial_2f_2 & -\partial_3f_2  \end{pmatrix} \begin{pmatrix} 0\\ 1\\ 0\end{pmatrix}
=\begin{pmatrix} 0\\ \partial_2f_3 \\ -\partial_2f_2\end{pmatrix}$\\
$\nabla \left[(f\wedge v)\cdot w\right]$ &$= \nabla\left[ \begin{pmatrix}0\\ f_3\\ -f_2 \end{pmatrix} \cdot \begin{pmatrix} 0\\ 1\\ 0 \end{pmatrix}\right] = \nabla f_3$.
\end{tabular}

\com{Interchanging the roles of $v$ and $w$, we obtain }
\begin{align*}
\left[D(f\wedge v)\right]\cdot w - \left[D(f\wedge w)\right] \cdot v &= \begin{pmatrix} \partial_1f_3\\ \partial_2 f_3\\ -\partial_1f_1 - \partial_2f_2\end{pmatrix} &&= \begin{pmatrix} \partial_1f_3\\ \partial_2 f_3\\ \partial_3f_3\end{pmatrix} - \begin{pmatrix}0 \\ 0 \\ \partial_1f_1 + \partial_2f_2 + \partial_3f_3\end{pmatrix}\\
	&= \nabla f_3 - \div(f)\,e_3 &&= \nabla (f\cdot(v\wedge w)) - \div(f)\,(v\wedge w).
\end{align*}
\com{The remaining equalities follow simply from $(f \wedge v) \cdot w = f\cdot(v\wedge w) = -(f\wedge w)\cdot v$}.
\end{proof}

\begin{Proposition}\label{prop: derivativeS}
Denote $\Omega_T := \bigcup_{t\in (-T,T)} \{t\}\times \big(\Omega\setminus \gamma_t\big)$. There exists $T >0$ such that the functions
\[
\S^\phi: \Omega_T \rightarrow \R^{3\times 3}, \quad \S^\phi(t,x) = \S_t(x) 
\]
and
\[
u^T: (-T,T) \to H^1(\Omega;\R^{3\times 3}\com{)}, \quad u^T(t) = u_t
\]
are $C^1$-smooth.
\com{We have the explicit formula}
\[
\dot\S^\phi(x) := \frac{d}{dt}\bigg|_{t=0} \S_t(x) =   -b\otimes \nabla w^\phi(x), \qquad w^\phi(x) := \int_{\gamma}  \left\langle k(x-y) \wedge \tau_y, \: \varphi(\gamma^{-1}(y)) \right\rangle \, d\mathcal{H}^1_y,
\]
\com{where $k$ is the same as in \eqref{eq: Sconvolution}.}
Moreover, $\dot u^\phi := \frac {d}{dt}\big|_{t=0}u^T$ satisfies
\[
\begin{pde} \label{eq: NeumannDerivative}
 -\Delta \dot u^\phi  &= 0 &\text{ in } \Omega,\\
 \com{\partial_{\nu}} \dot u^\phi &= -\dot \S^\phi \nu& \text{ on } \partial \Omega.
\end{pde}
\]
\end{Proposition}

\begin{proof}
\com{In view of \eqref{eq: Sconvolution} it holds}
\begin{align*}
\S_t(x) &= b\otimes \int_0^1 k\left(x- \left(\gamma(s) + t\varphi(s)\right)\right) \wedge \left( \gamma'(s) + t\varphi'(s)\right)\d s.
\end{align*}
\com{We remark that for every vector field $f: \Omega \to \R^3$ and vectors $a,b \in \R^3$ it holds for every $x \in \Omega$
\begin{equation} \label{eq: derivative and wedge}
(Df(x)\, a )\wedge b = D(f\wedge b)(x) \,a.
\end{equation}
Using Lemma \ref{lemma wedge stuff} and \eqref{eq: derivative and wedge}, we compute}
\begin{align*} 
\com{\dot S^\phi(x)} &= \com{\frac{d}{dt}\bigg|_{t=0}\S_t(x)}\\
	 &= \com{b\otimes \int_0^1 \left(\left[Dk(x-\gamma(s))(-\varphi(s))\right]\wedge \gamma'(s) + k(x-\gamma(s))\wedge \varphi'(s)\right)\d s}\\
	&=  \com{b\otimes \int_0^1 \left(\left[Dk(x-\gamma(s))(-\varphi(s))\right]\wedge \gamma'(s) - \left[ Dk(x-\gamma(s))(-\gamma'(s))\right] \wedge \varphi(s)\right)\d s}\\
	&= \com{b\otimes \int_0^1 \left(\left[ Dk(x-\gamma(s))\gamma'(s)\right] \wedge \varphi(s) - \left[Dk(x-\gamma(s))\varphi(s)\right]\wedge \gamma'(s) \right)\d s} \\
	&= \com{b\otimes \int_0^1 D_x\left[k(x-\gamma(s))\wedge \varphi(s)\right] \gamma'(s) - D_x\left[k(x-\gamma(s)) \wedge \gamma'(s)\right]\cdot\varphi(s)\d s}\\
	&=\com{ b\otimes \int_0^1 \nabla_x \left[k(x-\gamma(s))\cdot \left( \varphi(s) \wedge \gamma'(s)\right) \right] - \left(\varphi(s) \wedge \gamma'(s)\right) \, \div_x(k(x-\gamma(s))\d s}\\
	&= \com{b\otimes \nabla_x \int_0^1 \left[k(x-\gamma(s)) \wedge \varphi(s)\right]\cdot\gamma'(s)\d s} \\
	&=\com{b\otimes \nabla_x \int_{\gamma} \left[ k(x-y) \wedge \varphi(\gamma^{-1}(y)) \right] \cdot \tau(y) \, d\mathcal{H}^1(y)},
\end{align*}
\com{where} we used the fact that $k$ is the gradient of the Newtonian kernel and thus divergence free away from its singularity.

\com{Let us fix $t \in (-T,T)$ and $t' \to t$. In view of formula \eqref{eq: Sconvolution}, we find that the difference quotients $\frac{\S_t - \S_{t'}}{t-t'}$ are uniformly bounded in space (again, away from $\gamma$)}. Since by \com{Theorem \ref{theorem boundary term limit}} $\frac{u_t - u_{t'}}{t-t'}$ solves the Neumann problem
\[
\begin{pde}
 - \Delta \left(\frac{u_t - u_{t'}}{t-t'}\right) &=0 &\text{in }\Omega\\
 \partial_\nu \left(\frac{u_t - u_{t'}}{t-t'}\right) &= - \left(\frac{\S_t - \S_{t'}}{t-t'}\right)\nu &\text{on }\partial\Omega
 \end{pde}
\]
we have
\[
\left\|\frac{u_t - u_{t'}}{t-t'}\right\|_{H^1(\Omega)} \leq C_\Omega\,\left\|\frac{\S_t - \S_{t'}}{t-t'}\right\|_{L^2(\partial\Omega)}
\]
such that the difference quotients $\frac{u_t - u_{t'}}{t-t'}$ have an $H^1$-weakly convergent subsequence with a limit $\dot u\in H^1(\Omega)$. \com{Now, fix $\varphi \in C^{\infty}(\overline{\Omega})$.} By weak convergence, we can pass to the limit on the left side of the equation
\[
\int_\Omega \left\langle \nabla \frac{u_t - u_{t'}}{t-t'}, \nabla \phi\right\rangle \dx = \int_{\partial\Omega} \left\langle \frac{\S_t - \S_{t'}}{t-t'}, \phi\right\rangle\d\H^{2}
\] 
and \com{since $t\to S_t(x)$ is differentiable for every $x \in \partial \Omega$}, we \com{can use the dominated convergence theorem to} pass to the limit on the right hand side as well. \com{We deduce that $\dot{u}$ has zero average and satisfies \eqref{eq: NeumannDerivative}. As the solution to \eqref{eq: NeumannDerivative} is unique (up to the addition of a constant) the limit $\dot{u}$ does not depend on the subsequence and thus $\frac{u_t - u_{t'}}{t-t'} \rightharpoonup \dot{u}$. This shows that $u^T$ is differentiable in time with values in $H^1(\Omega;\R^3)$ (equipped with the weak topology)}.

\com{Finally, we show that $u^T$ is also differentiable in time with values in $H^1(\Omega;\R^3)$ (equipped with the strong topology.} Similarly to the proof of Theorem \ref{theorem boundary term limit} we find that
\[
\int_\Omega \left|\nabla \frac{u_t - u_{0}}{t-0}\right|^2 \dx \to \int_\Omega |\nabla \dot u|^2\dx
\]
since if
\[
\int_\Omega |\nabla \dot u|^2\dx < \liminf_{t\to 0} \int_\Omega \left|\nabla \frac{u_t - u_{0}}{t}\right|^2 \dx 
\]
we could show that
\[
\int_\Omega |\nabla \dot u|^2\dx - \int_{\partial\Omega} \left\langle\dot u, \frac{\S_t(x) - \S_{0}(x)}{t}\right\rangle\d\H^2 < \int_\Omega \left|\nabla \frac{u_t - u_{0}}{t}\right|^2\dx - \int_{\partial\Omega} \left\langle \nabla \frac{u_t - u_{0}}{t},  \frac{\S_t(x) - \S_{0}(x)}{t}\right\rangle \d\H^2
\]
for some small $t$ since $H^1(\Omega)$ embeds compactly in $L^2(\partial\Omega)$. This would contradict the characterization of solutions to the Neumann problem as minimizers of an energy functional. It follows that the difference quotients converge strongly in $H^1$ and thus that $u^T$ is differentiable with values in $H^1$.
\end{proof}

The following Proposition is the main result of this section.

\begin{Proposition}\label{prop: variation}
We have
\[
\frac{\d \Eeff(\mu_t)}{dt}\bigg|_{t=0}= -\int_{\partial B_\eps(\gamma)}\frac12\,|\S|^2 \, \phi \cdot \nu_\eps - w^{\phi}\, b\cdot (\S+\nabla u)\nu_\eps -  u \cdot \dot{\S}^\phi\nu_\eps\d\H^2.
\]
\end{Proposition}

\begin{proof}
\com{We use the Reynolds transport theorem \cite[Equation (15.23)]{antman} for $\Omega_{\eps}(\gamma+ t \varphi)$. The corresponding velocity at $t=0$ on $\partial B_{\eps}(\gamma)$ is given by $\varphi$.
We obtain }

\begin{align*}
\frac{\d\E_\eps(\mu_t)}{\dt} = &\frac{\d}{\d t}_{|t=0} \left[ \int_{\Omega_{\eps}(\gamma + t \varphi)} \frac12 |\S|^2 \d x + \int_{\Omega} \frac12 |\nabla u|^2 \d x + \int_{\partial \Omega} u\cdot \S\nu \d \H^2  \right]  \\
= &\com{-\int_{\partial B_{\eps}(\gamma)} \frac12 |\S|^2 (\varphi \cdot \nu_{\eps}) \d \H^2 + \int_{\Omega_{\eps}(\gamma)} {S} : \dot{\S}^{\varphi} \d x + \int_{\Omega} \nabla u: \nabla \dot{u}^{\varphi} \d x} \\
&\com{+ \int_{\partial \Omega} \dot{u}^{\varphi} \cdot \S\nu + u \cdot \dot{\S}^{\varphi} \nu \d \H^2} \\
= &\com{-\int_{\partial B_{\eps}(\gamma)} \frac12 |\S|^2 (\varphi \cdot \nu_{\eps}) \d \H^2 - \int_{\Omega_{\eps}(\gamma)} {S} :(b \otimes \nabla w^{\varphi}) \d x }\\
&\com{+ \int_{\partial \Omega} \left(\dot{u}^{\varphi} \cdot \partial_{\nu} u+ \dot{u}^{\varphi} \cdot \S\nu + u \cdot \dot{\S}^{\varphi} \nu\right) \d \H^2 }\\
=&\com{\int_{\partial B_{\eps}(\gamma)} \left(-\frac12 |\S|^2 (\varphi \cdot \nu_{\eps}) + w^{\varphi} b\cdot \S \nu_{\eps} \right) \d \H^2 -  \int_{\partial \Omega} \left( w^{\varphi}b \cdot \S \nu + u \cdot \dot{\S}^{\varphi} \nu\right) \d \H^2. }
\end{align*}
\com{Here, we used that $\div \S = \Delta u = 0$, $\partial_{\nu} u = -\S \nu$ on $\partial \Omega$ and the fact that the outer normal of $\partial \Omega_{\eps}$ on $\partial B_{\eps}(\gamma)$ is given by the negative of the outer normal $\nu_{\eps}$ of $\partial B_{\eps}(\gamma)$. Now, we notice that}
\begin{align*}
-\int_{\partial \Omega} w^{\varphi}b \cdot \S \nu \d \H^2 = &\com{\int_{\partial \Omega} w^{\varphi} b \cdot \partial_{\nu} u \d \H^2 }\\
= &\com{\int_{\Omega_{\eps}} \left(\Delta u \cdot bw^{\varphi} + \nabla u : b\otimes \nabla w^{\varphi}\right) \d x + \int_{\partial B_{\eps}(\gamma)} w b \cdot \partial_{\nu_{\eps}} u \d \H^2 }\\
= &\com{-\int_{\Omega_{\eps}} u \cdot b \Delta w \d x + \int_{\partial B_{\eps}(\gamma)}  \left(w b \cdot \partial_{\nu_{\eps}} u - u \cdot(b\otimes \nabla w^{\varphi}) \nu_{\eps}\right)\d \H^2} \\ 
&\com{+ \int_{\partial \Omega} u \cdot (b\otimes \nabla w^{\varphi}) \nu \d \H^2}\\
=&\com{\int_{\partial B_{\eps}(\gamma)} \left(w b \cdot \nabla u \nu_{\eps} - u \cdot\dot{\S}^{\varphi} \nu_{\eps}\right)\d \H^2 + \int_{\partial \Omega} u \cdot \dot{\S}^{\varphi} \nu \d \H^2.}
\end{align*}
\com{Consequently, we find that}
\begin{align*}
\com{\frac{\d\E_\eps(\mu_t)}{\dt} = \int_{\partial B_{\eps}(\gamma)} \left(-\frac12 |\S|^2 (\varphi \cdot \nu_{\eps}) + w^{\varphi} b\cdot \left(\S+\nabla u\right) \nu_{\eps} - u\cdot \dot{\S}^{\varphi} \nu_{\eps}\right) \d \H^2.}
\end{align*}
\end{proof}

\section{Asymptotic expansions}\label{section asymptotics}

In this chapter we obtain rigorous asymptotic expansions for the singular strain $\mathcal{S}$, \com{see \eqref{eq: Sconvolution}},  and the Peach-K\"{o}hler force. 

\subsection{The singular strain}

Denote \com{by $r>0$} the embeddedness radius of $\gamma$. \com{We} recall the following.
\begin{enumerate}
\item $\pi:B_r(\gamma)\to \gamma$ \com{is} the closest point projection, which is well-defined and $C^{1}$-smooth ($C^{1,\alpha}$ if $\gamma \in C^{2,\alpha}$),
\item $\dist(x,\gamma)$ is $C^{2}$-smooth \com{in} $B_r(\gamma)$ ($C^{2,\alpha}$ smooth if $\gamma\in C^{2,\alpha}$),
\item $\nu := \nabla \dist(\cdot,\gamma)$ denotes the exterior normal field to the tubular neighborhood $B_{\rho}(\gamma)$ for $0<\rho<r$.
\item $\vv H:\gamma\to\R^3$ \com{stands for} the curvature vector field of $\gamma$, i.e.\ if $\gamma$ is parametrized by unit length then  $\vv H_{\gamma(s)} = \gamma''(s)$.
\end{enumerate}

\begin{Theorem}\label{thm: expansionK}
Let $\gamma \in C^2(S^1; \R^3)$ and $x\in B_r(\gamma)$. Then 
\begin{equation} \label{eq: expansionS}
\S(x) = \frac1{2\pi} \,b \otimes \left(\frac1{\dist(x,\gamma)}\, \tau_{\pi(x)}\wedge \nu_x + \frac12|\log\dist(x,\gamma)|\, \tau_{\pi(x)}\wedge  \vv H_{\pi(x)} + R(x) \right)
\end{equation}
where $R(x)$ satisfies
\begin{enumerate}
\item $\|R\|_{L^\infty(\partial B_\eps(\gamma))} \leq C\big(1+ \|\gamma\|_{C^2}^2 + |\log\eps|\big)$ for some $C$ which depends on $\gamma$ only through the embeddedness radius,

\item $\lim_{\eps\to 0} \frac{\| R\|_{L^\infty(\partial B_\eps(\gamma))}}{|\log\eps|}=0$. This estimate is not uniform in $C^2$ and depends on the modulus of continuity of $\gamma''$,

\item $\|R\|_{L^\infty(\partial B_\eps(\gamma))} \leq C\left(1+ \|\gamma\|_{C^2}^2 + \|\gamma\|_{C^{2,\alpha}}\right)$ for some $C$ if $\gamma$ is $C^{2,\alpha}$-smooth. The constant $C$ depends on $\gamma$ only through the embeddedness radius.
\end{enumerate}
\end{Theorem}

\begin{proof}
\com{First we} establish the identity \eqref{eq: expansionS} for a Taylor approximation of the curve $\gamma$. In a second step, we estimate the error terms. 

\com{We denote $\eps:= \dist(x,\gamma)$}. Without loss of generality, we assume that 
\[
\gamma(0) = 0, \qquad \tau_{\gamma(0)} = \gamma'(0) = e_3, \qquad x = \eps\,e_1 \quad (\Ra \nu_x = e_1).
\]
Recall that
\begin{align*}
\S(x) =  b\otimes \int_{\gamma} k(x-y)\wedge \tau_y \d\H^1_y = b\otimes \int_{0}^1 k(x-\gamma(s))\wedge {\gamma'(s)} \d s.
\end{align*}
We fix $\tilde{r}r>0$ smaller than the embeddedness radius $r$ and such that $|\langle \tau_{\gamma(s)}, \tau_{\gamma(s')}\rangle|\geq \frac12$ whenever $|\gamma(s)-\gamma(s')|\leq 2\tilde{r}$. 

{\bf Step 1. Taylor approximation.}
\com{We begin by replacing $\gamma$ by its second order Taylor polynomial and consider}
\[
P_\gamma:(-\tilde{r},\tilde{r})\to \R^3, \qquad
P_\gamma(s) = 0 + s\,e_3 + \frac{s^2}2 \vv H =  \left(\frac{\gamma_x}2\,s^2 , \frac{\gamma_y}2\,s^2, s\right),
\]
\com{and we} calculate the associated approximation of $\S(x)$ as
\begin{align}
\widetilde \S(x) &= b\otimes \int_{P_\gamma} k(x-y) \wedge \tau_y\d\H^1_y \nonumber\\
	&= b\otimes\int_{-\tilde r}^{\tilde r} k\big(x - P_\gamma(s)\big) \wedge P_\gamma'(s)\ds \nonumber\\
	&= b\otimes\int_{-\tilde r}^{\tilde r} \frac{-1}{4\pi} \frac{x-P_\gamma(s)}{|x- P_\gamma(s)|^3} \wedge P_\gamma'(s)\ds \nonumber\\
	&= \frac{1}{4\pi} b\otimes \int_{-\tilde r}^{\tilde r} \frac{1}{\left[\left(\frac{\gamma_x}2 s^2 - \eps\right)^2 + \left(\frac{\gamma_y}2 s^2\right)^2 + s^2\right]^{3/2}} \begin{pmatrix} \frac{\gamma_x}2\,s^2 - \eps \nonumber \\ \frac{\gamma_y}2\,s^2\\ s\end{pmatrix}\wedge \begin{pmatrix} \gamma_x s\\ \gamma_ys\\ 1\end{pmatrix}\ds \nonumber\\
	&= \frac{1}{4\pi} b\otimes \int_{-\tilde r}^{\tilde r} \frac{1}{\left[\left(\frac{\gamma_x}2 s^2 - \eps\right)^2 + \left(\frac{\gamma_y}2 s^2\right)^2 + s^2\right]^{3/2}} \begin{pmatrix} \frac{\gamma_y}2 s^2 - \gamma_y s^2\\ \gamma_xs^2 - \left(\frac{\gamma_x}2s^2 - \eps\right)\\ \frac{\gamma_x}2\gamma_y\,s^3 -\gamma_y\eps s - \gamma_x\,\frac{\gamma_y}2\,s^3\end{pmatrix}\ds \nonumber\\
	&= b\otimes\frac1{4\pi}\int_{-\tilde r}^{\tilde r} \frac{s^2/2}{\left[\left(\frac{\gamma_x}2 s^2 - \eps\right)^2 + \left(\frac{\gamma_y}2 s^2\right)^2 + s^2\right]^{3/2}}\ds \begin{pmatrix} -\gamma_y\\ \gamma_x\\ 0\end{pmatrix} \nonumber \\
	&\qquad + b\otimes\frac1{4\pi}\int_{-\tilde r}^{\tilde r} \frac{\eps}{\left[\left(\frac{\gamma_x}2 s^2 - \eps\right)^2 + \left(\frac{\gamma_y}2 s^2\right)^2 + s^2\right]^{3/2}}\ds \begin{pmatrix}0\\ 1\\ 0\end{pmatrix} \label{eq: tildeS Taylor}
\end{align}
since the odd integrals in $e_3$ direction cancel out due to anti-symmetry. 
Note that $\vv H_{\gamma(0)} = \vv H_{P_\gamma(0)} = (\gamma_x, \gamma_y,0)$, and therefore
\begin{equation}\label{eq: relations Taylor}
(-\gamma_y, \gamma_x, 0) = e_3 \wedge (\gamma_x, \gamma_y,0) = \tau_{\pi(x)} \wedge \vv{H}_{\pi(x)}, \qquad 
e_2 = e_3\wedge e_1 = \tau_{\pi(x)} \wedge \nu_x.
\end{equation}
\com{In view of \eqref{eq: expansionS}, \eqref{eq: tildeS Taylor} and \eqref{eq: relations Taylor}, we claim that}
\begin{equation} \label{eq: tildeS logeps}
\int_{-\tilde r}^{\tilde r} \frac{s^2/2}{\left[\left(\frac{\gamma_x}2 s^2 - \eps\right)^2 + \left(\frac{\gamma_y}2 s^2\right)^2 + s^2\right]^{3/2}}\ds = |\log \eps| + O(1)
\end{equation}
and
\begin{equation} \label{eq: tildeS eps}
\int_{-\tilde r}^{\tilde r} \frac{\eps}{\left[\left(\frac{\gamma_x}2 s^2 - \eps\right)^2 + \left(\frac{\gamma_y}2 s^2\right)^2 + s^2\right]^{3/2}}\ds = \frac2{\eps} + O(1).
\end{equation}
\com{Indeed, we obtain with a change of variables}
\begin{align}
\int_{-\tilde r}^{\tilde r}& \frac{s^2/2}{\left[\left(\frac{\gamma_x}2 s^2 - \eps\right)^2 + \left(\frac{\gamma_y}2 s\right)^2 + s^2\right]^{3/2}}\ds \nonumber\\
	&=\int_{-\tilde r}^{\tilde r} \frac{s^2/2}{\left[ \eps^2 - \gamma_x\eps\,s^2 + \frac14 \left(\gamma_x^2 + \gamma_y^2\right)s^4 + s^2\right]^{3/2}}\ds \nonumber\\
	&= \frac12 \int_{-\tilde r}^{\tilde r} \frac{\eps^2 (s/\eps)^2}{\,\eps^3\left[1 + \frac{|\vec H|^2\,\eps^2}4 (s/\eps)^4 + (1-\gamma_x\eps)\,(s/\eps)^2\right]^{3/2}}\ds \nonumber\\
	&= \frac12\int_{-\tilde r/\eps}^{\tilde r/\eps} \frac{s^2}{\left[1 + \frac{|\vv H|^2 \,\eps^2}4\,s^4 + (1-\gamma_x\eps)s^2\right]^{3/2}}\ds. \label{eq: tildeS term1 cov}
\end{align}
\com{Setting}
\[
\com{I^1(H, t, \tilde{r},\eps):= \int_{-\tilde r/\eps}^{\tilde r/\eps} \frac{s^2}{\left[1 + \frac{H \,\eps^2}4\,s^4 + (1-t\eps)s^2\right]^{3/2}}\ds},
\]
\com{\eqref{eq: tildeS term1 cov} reduces to $\frac12 I^1(|\vv H|^2, \gamma_x, \tilde r,\eps)$, and we have}
\begin{align*}
I^1\left(0,0,r,\eps\right) &= \int_{-\tilde r/\eps}^{\tilde r/\eps} \frac{s^2}{\left[1 + s^2\right]^{3/2}}\ds\\
%	&= \int_{\arsinh(-r/\eps)}^{\arsinh(r/\eps)}  \frac{\sinh^2(\theta)}{[1+\sinh^2(\theta)]^{3/2}}\,\cosh(\theta)\d\theta\\
%	&= \int_{\arsinh(-r/\eps)}^{\arsinh(r/\eps)} \frac{\sinh^2(\theta)\,\cosh(\theta)}{\cosh^3(\theta)}\d\theta\\
%	&= \int_{\arsinh(-r/\eps)}^{\arsinh(r/\eps)}\tanh^2(\theta)\d\theta\\
%	&= \arsinh(r/\eps) + \tanh(\arsinh(r/\eps)) - \arsinh(-r/\eps) - \tanh(\arsinh(-r/\eps))\\
%	&=2 \arsinh(r/\eps) + 2\, \frac{r/\eps}{\sqrt{1+(r/\eps)^2}}\\
%	&\approx 2\left(\log(2r/\eps) + \frac{r/\eps}{\sqrt{1+r/\eps)^2}}\right)\\
%	&\approx 2\left(|\log\eps| + \log(2r) + 1\right)
\partial_{H} I^1\left(|\vv H|^2, \gamma_x, r,\eps\right) &= \int_{-\tilde r/\eps}^{\tilde r/\eps}\left(-\frac32\right) \frac{s^2} {\left[1 + \frac{|\vv H|^2 \,\eps^2}4\,s^4 + (1-\gamma_x\eps)s^2\right]^{5/2}} \frac{ \eps^2\,s^4}4\ds\\
	&= \frac{-3\,\eps^2}8\int_{-\tilde r/\eps}^{\tilde r/\eps}\frac{s^6} {\left[1 + \frac{|\vv H|^2 \,\eps^2}4\,s^4 + (1-\gamma_x\eps)s^2\right]^{5/2}}\ds,\\
\partial_{t} I^1\left(|\vv H|^2, \gamma_x, r,\eps\right)% &= \int_{-\tilde r/\eps}^{\tilde r/\eps}\left(-\frac32\right) \frac{s^2} {\left[1 + \frac{|\vv H|^2 \,\eps^2}4\,s^4 + (1-\gamma_x\eps)s^2\right]^{5/2}} \big(-\eps\,s^2\big)\ds\\
	&= \frac{3\eps}2 \int_{-\tilde r/\eps}^{\tilde r/\eps} \frac{s^4} {\left[1 + \frac{|\vv H|^2 \,\eps^2}4\,s^4 + (1-\gamma_x\eps)s^2\right]^{5/2}} \ds.
\end{align*}
\com{A calculation using hyperbolic functions and their identities shows that}
\[
\lim_{\eps\to 0} \left[ I^1(0,0,\tilde r,\eps) - 2\,|\log\eps|\right] = 2\left[\log(2\tilde r) +1\right].
\]
The second and third integral do not have any singularities, so it is only the behavior of the integrand at infinity which determines the behavior of the integral. In the derivative with respect to $|\vv H|^2$, the integrand diverges linearly and thus the integral diverges quadratically. Together with the prefactor, this means that $\partial_{H}I^1 = O(1)$. In the second derivative, the integrand decays to zero as $1/s$, so the integral diverges logarithmically. This means that $\partial_{t}I^1 = O(\eps|\log\eps|)$ and thus in total
\[
\int_{-r}^r \frac{s^2/2}{\left[\left(\frac{\gamma_x}2 s^2 - \eps\right)^2 + \left(\frac{\gamma_y}2 s\right)^2 + s^2\right]^{3/2}}\ds = |\log\eps| + O(1)
\]
\com{proving \eqref{eq: tildeS logeps}.}

\com{To obtain \eqref{eq: tildeS eps} we write}
\begin{align*}
I^0(|\vv{H}|^2,\gamma_x, r,\eps) &:= \int_{-\tilde r}^{\tilde r} \frac{\eps}{\left[\left(\frac{\gamma_x}2 s^2 - \eps\right)^2 + \left(\frac{\gamma_y}2 s^2\right)^2 + s^2\right]^{3/2}}\ds\\
	%&=  \int_{-\tilde r}^{\tilde r} \frac{\eps}{\eps^2\,\left[1 + \frac{|\vv H|^2 \,\eps^2}4\,s^4 + (1-\gamma_x\eps)s^2\right]^{3/2}} \frac1\eps \ds\\
	&= \frac 1\eps \int_{-\tilde r/\eps}^{\tilde r/\eps} \frac{1}{\left[1 + \frac{|\vv H|^2 \,\eps^2}4\,s^4 + (1-\gamma_x\eps)s^2\right]^{3/2}}\ds.
\end{align*}
A simple calculation shows that
\begin{align*}
I^0(0,0,r,\eps) &= \frac1\eps \int_{-\tilde r/\eps}^{\tilde r/\eps} \frac{1}{[1+s^2]^{3/2}}\ds\\
	&= \frac2\eps \frac{(\tilde r/\eps)}{\sqrt{1+ (\tilde r/\eps)^2}}\\
	&= \frac2\eps + O(\eps),\\
%	&= \frac2\eps \frac{r}{\sqrt{\eps^2 + r^2}}\\
%	&= \frac 2\eps + \frac{2\eps}{\sqrt{r^2+\eps^2}} \frac{\sqrt{r^2} - \sqrt{r^2+\eps^2}}{\eps^2}\\
%	&\approx \frac2\eps + \frac2r\,\frac{1}{2\sqrt{r}}\eps\\
\partial_{H} I^0\left(|\vv H|^2, \gamma_x, r,\eps\right) &= \frac{-3 \eps}{8} \int_{-\tilde r/\eps}^{\tilde r/\eps} \frac{s^4}{\left[1 + \frac{|\vv H|^2 \,\eps^2}4\,s^4 + (1-\gamma_x\eps)s^2\right]^{5/2}}\ds \\
	&= O(\eps\,|\log\eps|),\\
\text{ \com{and} } \qquad \partial_{t} I^0\left(|\vv H|^2, \gamma_x, r,\eps\right) &= \frac{3}{2} \int_{-\tilde r/\eps}^{\tilde r/\eps} \frac{s^2}{\left[1 + \frac{|\vv H|^2 \,\eps^2}4\,s^4 + (1-\gamma_x\eps)s^2\right]^{5/2}}\ds\\
	&= O(1).
\end{align*}
\com{Hence $I^0(|\vv H|^2, \gamma_x, \tilde r,\eps) = \frac2{\eps} + O(1)$ and we conclude \eqref{eq: tildeS eps}.} \\

{\bf Step 2. Estimate of the remainder for $C^2$-curves.} So far we have shown that 
\begin{align*}
\com{\widetilde S(x) = \frac{b}{2\pi} \otimes} &\com{\bigg[ \left(\eps^{-1} + f_{r,\eps}^{(0)}\left(|\vv H_{\pi(x)}|^2, \gamma_x\right)\right) \tau_{\pi(x)}\wedge \nu_x } \\ 
 &\com{+ \left(\frac{|\log\eps|}{2}+ f_{r,\eps}^{(1)}\left(|\vv{H}_{\pi(x)}|^2, \gamma_x\right)\right) \:\tau_{\pi(x)} \wedge \vv{H}_{\pi(x)} \bigg],}
\end{align*}
where the functions $f_{r,\eps}^0$ \com{and} $f^1_{r,\eps}$ converge to a finite value at the origin and have uniformly bounded Lipschitz constants. \com{We include these terms in the} remainder term $R$ \com{in \eqref{eq: expansionS}}. It remains to estimate how large \com{is the error} stemming from the use of the Taylor approximation \com{in place of $\gamma$}.

\com{We} decompose $\S_\gamma(x)$ into a local part and a nonlocal part. Assume that $\gamma$ is parametrized by unit speed and set
\[
\com{\S(x) = \S^{loc}(x) + \S^{nl}(x),}
\]
\com{where}
\[
\com{S^{loc}(x) := \int_{-\tilde r}^{\tilde r} k\big(x-\gamma(s)\big)\wedge\gamma'(s)\ds \text{ and } S^{nl}(x):= \int_{S^1\setminus[-\tilde r,\tilde r]} k\big(x-\gamma(s)\big)\wedge\gamma'(s)\ds. }
\]
\com{As $|k(x)| = \frac1{4\pi |x|^2}$, the} non-local term is easily estimated as
\begin{equation} \label{eq: nonlocalestimate}
|\S^{nl}(x)| \leq \frac{\H^1(\gamma)}{4\pi\,\tilde{r}^2}
\end{equation}
when $\tilde{r}$ is small enough (e.g.\ the embeddedness radius) so we focus on the local term.
\com{Using that (assuming that $\eps$ is small enough)}
\begin{equation}\label{eq: estimatesTaylorgamma}
\com{|x-\gamma(s)|^2 \geq (\com{s^2}+\eps^2)/2 \text{ and } |x - P_{\gamma}(s)| \geq (\com{s^2} + \eps^2)/2}
\end{equation}
\com{we find that }
\begin{align}
\big|\S^{loc}(x) &- \widetilde \S(x)\big| %\leq \left| b\otimes \int_\gamma k(x-y)\wedge\tau_y\d\H^1_y - b\otimes \int_{P_\gamma} k(x-y) \wedge \tau_y\d\H^1_y\right|\nonumber\\
	\leq |b| \left|\int_{-\tilde r}^{\tilde r} \left(k\big(x - \gamma(s)\big) \wedge \gamma'(s) - k\big( x-  P_\gamma(s)\big) \wedge  P_\gamma'(s)\right)\ds \right|\nonumber\\
	&\leq |b| \int_{-\tilde r}^{\tilde r} \left(\left| \left[k\big(x - \gamma(s)\big) -  k\big( x-  P_\gamma(s)\big)\right]\wedge \gamma'(s) \right| 
			+ \left| k\big(x- P_\gamma(s)\big) \wedge \left[\gamma'(s) -  P_\gamma'(s)\right] \right| \right)\ds\nonumber\\
	&\leq |b| \int_{-\tilde r}^{\tilde r} \left(\left|\frac{-1}{4\pi}\,\frac{ x-\gamma(s)}{|x-\gamma(s)|^3} + \frac{1}{4\pi}\frac{x- P_\gamma(s)}{|x- P_\gamma(s)|^3}\right|\,\|\gamma'\|_{L^\infty} + \frac{1}{4\pi\,|x- P_\gamma(s)|^2}|\gamma'(s)- P_\gamma'(s)| \right)\ds\nonumber\\
	&\leq\frac{|b|}{4\pi} \int_{-\tilde r}^{\tilde r}  \left(\left|\frac{ P_\gamma(s)-\gamma(s)}{|x-\gamma(s)|^3}\right| + \left|\big(x- P_\gamma(s)\big)\left[ \frac1{|x-\gamma(s)|^3} - \frac{1}{|x- P_\gamma(s)|^3}\right]\right| + \frac{|\gamma'(s)-  P_\gamma'(s)|}{|x- P_\gamma(s)|^2} \right)\ds \nonumber\\
	&\leq\frac{|b|}{4\pi} \int_{-\tilde r}^{\tilde r} \left(\frac{ |P_\gamma(s)-\gamma(s)|}{(s^2/2+\eps^2/2)^{3/2}} + \left|\big(x- P_\gamma(s)\big)\left[ \frac1{|x-\gamma(s)|^3} - \frac{1}{|x- P_\gamma(s)|^3}\right]\right| + 2\frac{|\gamma'(s)-  P_\gamma'(s)|}{s^2+ \eps^2}\right)\ds. \label{eq: comptildeSSloc}
\end{align}
\com{Denote} by 
\[
\omega(s):= \sup\{|\gamma''(s_1) - \gamma''(s_2)|\::\:|s_1-s_2|<s\}
\]
the modulus of continuity of $\gamma''$, and observe that 
\begin{align*}
\gamma'(s) - P_\gamma'(s) &= \left[\gamma' - P_\gamma' \right](s) - \left[\gamma' - P_\gamma'\right](0)\\
	&= \int_0^s \left(\gamma''(\theta) - P_\gamma''(\theta)\right)\d\theta\\
	&= \int_0^s \left(\left[\gamma'' - P_\gamma''\right](\theta) - \left[\gamma'' - P_\gamma''\right](0)\right)\d\theta\\
	&= \int_0^s \left(\gamma''(\theta) - \gamma''(0)\right)\d\theta
\end{align*}
since $P_\gamma''$ is constant as $P_\gamma$ is a quadratic polynomial. It follows that
\begin{equation}\label{eq: comptildeSSloc2}
|\gamma'(s) - P_\gamma'(s)| \leq \int_0^s\omega(|\theta|)\d\theta \leq s\,\omega(|s|)
\end{equation}
and, similarly, that
\begin{equation}\label{eq: comptildeSSloc3}
|\gamma(s) - P_\gamma(s)| \leq s^2\,\omega(|s|).
\end{equation}
Finally, we note that 
\begin{equation}\label{eq: comptildeSSloc4}
\left|\frac{1}{|x-\gamma(s)|^3} - \frac{1}{|x-P_\gamma(s)|^3}\right| = \left|\frac{-3}{\xi_s^4}\cdot \big(|x-\gamma(s)| - |x-P_\gamma(s)|\big)\right| \leq \frac{3}{\xi_s^4}\,\big |\gamma(s) - P_\gamma(s)\big|
\end{equation}
where $\xi_s$ is a point between $|\gamma(s)-x|$ and $|P_\gamma(s)-x|$, so $\xi_s\geq \sqrt{\frac{s^2+\eps^2}{2}}$. \com{In view of \eqref{eq: comptildeSSloc}, \eqref{eq: comptildeSSloc2}, \eqref{eq: comptildeSSloc3} and \eqref{eq: comptildeSSloc4} we obtain}
\begin{align} \label{eq: estimatemodcont}
\big|\S^{loc}(x) &- \widetilde \S(x)\big| \leq \frac{|b|\,\omega(r)}{4\pi} \int_{-r}^r \frac{ s^2}{(s^2/2+\eps^2/2)^{3/2}} + \big(\eps^2 + s^2 + |\vv H|^2s^4 \big)^\frac12\frac{3}{\xi_s^4}\,s^2 + 2\frac{|s|}{s^2+ \eps^2}\ds\\
	&\leq \frac{|b|\,\omega(r)}{2\pi} \bigg[\int_0^\eps \frac{\sqrt{8}\,s^2}{\eps^3} + \frac{12\,(\eps^2 + s^2 + |\vv H|^2 s^4)^\frac12}{\eps^4}\,s^2 + \frac{2\,|s|}{\eps^2}\ds \nonumber\\
	&\hspace{1.5cm}+\int_\eps^r \frac{\sqrt{8} \,s^2}{s^3} + \frac{12 \, \left(\eps^2 + s^2+ |\vv H|^2 s^4\right)^\frac12}{s^4}s^2\,+ 2\frac{|s|}{s^2}\ds \bigg]  \nonumber\\
	&\leq \frac{|b|\,\omega(r)}{2\pi}\,C\left(1 + |\vv H|^2 + |\log\eps|\right) \nonumber
\end{align}
where $C$ is a universal constant. It follows together with \eqref{eq: nonlocalestimate} -- even without a modulus of continuity estimate -- that
\[
\big|\widetilde \S(x) - \S(x)\big| \leq C\,|b| \left[|\log\eps|+1\right]
\]
where $C$ depends only on the embeddedness radius and the $C^{1,1}$-norm of $\gamma$. This can be improved by choosing $\tilde{r}= \tilde{r}_\eps$ such that
\[
\tilde{r}_\eps \to 0, \qquad \frac{1}{|\log\eps|\,\tilde{r}_\eps^2}\to 0
\]
and  observing that
\[
\frac{|\widetilde{\S}(x) - \S(x)|}{|\log\eps|} \leq \frac{|\widetilde \S(x) - \S^{loc}(x)| + |\S^{nl}(x)|}{|\log\eps|} \leq \frac{C\,\omega(\tilde{r}_\eps)\,(1+|\log\eps|)}{|\log\eps|} + \frac{C}{|\log\eps|\,\tilde{r}_\eps^2} \to 0
\]
since $\omega(\tilde{r}_\eps)\to 0$. 
As $\frac{\tilde{r}_\eps}\eps\to \infty$ \com{all the above} analysis for $\widetilde S$ \com{is still valid under this dependence} of $\tilde{r}$ on $\eps$. It follows that $\|R\|_{L^\infty(\partial B_\eps)} = o(|\log\eps|)$. 

{\bf \com{Step 3. Estimate of} the remainder for $C^{2,\alpha}$-curves.} \com{Thus} far we have not made much use of the modulus of continuity estimate in calculations. \com{We can make use of \eqref{eq: comptildeSSloc}, \eqref{eq: comptildeSSloc2}, \eqref{eq: comptildeSSloc3} and \eqref{eq: comptildeSSloc4} to obtain instead of \eqref{eq: estimatemodcont} the estimate}
\begin{align*}
\com{|\S^{loc}-\widetilde{\S}|} &\com{\leq \frac{|b|}{4\pi} \int_{-\tilde r}^{\tilde r} \frac{s^2 \omega(s)}{(s^2/2 + \eps^2 / 2)^{3/2}} + \big(\eps^2 + s^2 + |\vv H|^2s^4 \big)^\frac12\frac{3}{\xi_s^4}\,\omega(s) s^2 + 2\frac{|s| \omega(s)}{s^2+ \eps^2}\ds} \\
&\com{\leq C \int_{-\tilde r}^{\tilde r} \frac{\omega(s)}{|s|}\ds}
\end{align*}
\com{If $\omega$ is such that $\int_{-\tilde r}^{\tilde r} \frac{\omega(s)}{|s|}\ds <\infty$, we find that $\|R\|_{L^\infty(B_\eps)} = O(1)$. This is in particular the case in H\"older spaces $C^{2,\alpha}$. }
\end{proof}

\com{
Before moving on, let us briefly relate $\S$ to physical slip and consider the situation with applied boundary forces.
}

\begin{remark}\label{remark forcing}
\com{So far} we \com{have} assumed that the only stresses in the crystal are due to the presence of the dislocation $\gamma$ and that no external forces are applied, such as volume forces coupling to the displacement \com{in} $\Omega$ (e.g.\ electro-magnetic forces), or boundary forces coupling to the energy on $\partial\Omega$. Boundary forces could be included in the model as (potentially time-dependent) linear terms in the energy.

\com{To be precise}, assume for the moment that $\Omega$ is a strictly star-shaped domain, i.e., \ $0\in\Omega$ and $r\Omega\cc\Omega$ for all $0<r<1$. Since $\gamma\cc\Omega$, there exists $r\in (0,1)$ such that $\gamma\subset r\Omega$ and by construction $\beta = \S + \nabla u$ is curl-free \com{in} the domain $\Omega\setminus\overline{r\Omega}$. Since $\Omega$ \com{is} strictly star-shaped, \com{$\Omega \setminus \overline{r\Omega}$} is homeomorphic to the (simply connected) domain $B_1(0) \setminus \overline{B_{1/2}(0)}$, and so $\beta$ is a gradient \com{in} $\Omega\setminus\overline{r\cdot\Omega}$, say $\beta = \nabla \beta^{\rm{lift}}$. The boundary forces are then included, as usual, by 
\[
\E(\mu) := \inf_{\beta = \S+ \nabla u\in \A_\mu} \left\{ \int_\Omega \frac12\,|\beta|^2 \dx - \int_{\partial\Omega} \langle \beta^{\rm{lift}},f^{surf} \rangle\dx\right\}.
\]
and act on $u$ as boundary conditions. \com{Physically}, a dislocation loop is the boundary of a slip plane $\Sigma$. Assuming that $\Sigma\cc\Omega$ and $\gamma = \partial\Sigma$, we obtain \com{using Stokes' Theorem} that for any $V\in \R^3$ we have
\begin{align*}
\left\langle\int_\gamma k(x-y)\wedge\tau_y\d\H^1, V\right\rangle 
	&= \int_\gamma\left\langle k(x-y)\wedge\tau_y, V\right\rangle\d\H^1\\
	&= \int_\gamma\left\langle V\wedge k(x-y), \tau_y\right\rangle\d\H^1\\
	&= \int_\Sigma \left\langle \curl_y\big(V\wedge k(x-y)\big), \nu_{\Sigma,y}\right\rangle \d\H^2\\
	&= \int_\Sigma \left\langle (\div k)\,V - \partial_Vk(x-y), \:\nu_{\Sigma,y}\right\rangle \d\H^2\\
	&= \int_\Sigma \left\langle 0 - D_yk(x-y) V, \:\nu_{\Sigma,y}\right\rangle \d\H^2\\
	&= \left(\int_\Sigma \nu_{\Sigma,y}^T D_xk(x-y)\d \H^2\right)\cdot V\\
	&= \left(\nabla_x \int_\Sigma k(x-y) \cdot \nu_{\Sigma,y} \d\H^2\right)\cdot V
\end{align*}
if $x\notin \overline{\Sigma}$ since then $\div k = \Delta G=0$ because $k$ is the gradient of the Newton potential. \com{Recall from \eqref{eq: Sconvolution} that $\S(x) = b \otimes \int_\gamma k(x-y)\wedge\tau_y\d\H^1$ and by definition $\nabla \beta^{\rm{lift}} = \S + \nabla u$ outside $r\Omega$. Thus, we find (up to an additive constant) }
\[
\beta^{\rm{lift}}(x) = u + b \int_\Sigma  k(x-y) \cdot \nu_{\Sigma,y}\dy.
\]
The slip-plane $\Sigma$ of $\beta^{\rm{lift}}$ is non-unique, but if $\Sigma, \Sigma'$ are two slip-surfaces, then $\beta^{\rm{lift}}$ is well-defined on connected components of $\Omega\setminus \big(\Sigma\cup \Sigma')$ up to additive constants since $k$ is divergence-free away from $\gamma$. We note that in the special case where $\gamma$ is the unit circle and $\Sigma$ is the unit disk, we can calculate $\beta^{\rm{lift}}$ at $x= (0,0,x_3)$ as
\begin{align*}
\beta^{\rm{lift}}(x) &= u+ b \int_\Sigma k(x-y) \cdot \nu_{\Sigma,y} \d \H^2\\
	&=u+ \frac{b}{4\pi} \int_{\{|y|<1\}}\left\langle\begin{pmatrix}0\\0\\1\end{pmatrix}, \frac{1}{[y_1^2+y_2^2 + x_3^2]^{3/2}}\begin{pmatrix}y_1\\ y_2\\ - x_3\end{pmatrix}\right\rangle \dy_1\dy_2\\
	%&= u-\frac{b}{4\pi} \int_{\{|y|<1\}} \frac{x_3}{[y_1^2+y_2^2 + x_3^2]^{3/2}}\dy_1\dy_2\\
	&= u-\frac{b}{4\pi} \int_{\{|y|<1\}}\frac{x_3}{\,|x_3|^3\,[\left(\frac{y_1}{x_3}\right)^2+\left(\frac{y_2}{x_3}\right)^2 + 1]^{3/2}}\dy_1\dy_2\\
	&= u-\frac{b\,\sign(x_3)}{4\pi} \int_{\{|y'|< \frac1{|x_3|}\}}\frac{1}{[1+(y_1')^2 + (y_2')^2}\dy_1\dy_2\\
	&= u-\frac{b\,\sign(x_3)}{4\pi}\int_0^{|x_3|^{-1}} [1+r^2]^{-3/2}\,2\pi r\dr\\
	%&= u-\frac{b\,\sign(x_3)}2 \left[1+r^2\right]^{-1/2}\bigg|^{r= |x_3|^{-1}}_{r=0}\\
	&\to u-\frac{b\,\sign(x_3)}{2}
\end{align*}
as $|x_3|\to 0^\pm$. This shows that the deformation $\beta^{\rm{lift}}$ jumps by precisely the Burger's vector across the surface $\Sigma$ since the symmetry only simplified the calculation of the integral. Similar computations can be made for curved surfaces by Taylor approximation. 
\end{remark}

From \com{\eqref{eq: Sconvolution}}, it is \com{clear that $\S\in C^\infty(\R^3\setminus \gamma; \R^{3\times 3})$.}
We show that if $\gamma\in C^{2,\alpha}$ \com{then} the remainder term $R$ \com{in \eqref{eq: expansionS}} is $C^{0,\alpha}$-H\"older continuous.

Let $A$ be a subset of $C^{2}(S^1; \R^3)$ such that all curves in $A$ are uniformly bounded in $C^2$ and have similar embeddedness radii as well as a uniform lower bound on $|\gamma'|$. \com{Using the construction described in Appendix \ref{appendix : cylindrical coordinates}} We choose normal fields $n_1, n_2$ for cylindrical coordinates $\psi_{\gamma;r}$ such that the maps
\[
A \to C^1(S^1;\R^3), \qquad \gamma \mapsto n_{i}\quad (i=1,2)
\]
are Lipschitz continuous.

\begin{lemma}\label{eq: estimaterepara}
The map
\[
\Upsilon_r: A \to C^0(S^1\times S^1), \qquad \gamma\mapsto R_\gamma\circ \psi_{\gamma;r}
\]
is Lipschitz-continuous with Lipschitz constant $L_\eps \leq L(|\log\eps|+1)$ for some $L>0$ depending on $A$ through the uniform bound on the $C^2$-norm, the uniform lower bound on the embeddedness radius, and a uniform lower bound on $|\gamma'|$.
\end{lemma}

\begin{proof}
{\bf Step 1.} \com{Assume first} that $\gamma, \eta:(-r,r)\to \R^3$ are curves \com{in $A$} such that
\[
\gamma(0) = \eta(0) = 0, \qquad \gamma'(0) = \eta'(0) = e_3, \qquad x = \eps\,e_1.
\]
\com{Replacing $P_{\gamma}$ by $\eta$ in Step 2 in the proof of Theorem \ref{thm: expansionK} one immediately derives the analogous estimates to \eqref{eq: comptildeSSloc}, \eqref{eq: comptildeSSloc2}, \eqref{eq: comptildeSSloc3} and \eqref{eq: comptildeSSloc4} with $\omega(s) = \| \gamma'' - \eta'' \|_{L^{\infty}}$. Then we obtain exactly in \eqref{eq: estimatemodcont} that}
\[
\com{|\S^{loc}_\gamma(x) - \S^{loc}_\eta(x)| \leq C\,(|\log\eps|+1)\,\|\gamma'' - \eta''\|_{L^\infty}}
\]
The non-local contribution to $\S$ can be estimated in a straight-forward way by $\|\gamma-\eta\|_{C^1}$. Since the $\tau\wedge\nu$-term in $\S_\gamma$ and $\S_\eta$ cancels, this means that
\[
\left|\left[|\log\eps| \vv H_\gamma(x) + R_\gamma(x) \right] - \left[|\log\eps|\,\vv H_\eta(x) +R_\eta(x)\right]\right| \leq C\,(|\log\eps|+1)\,\|\gamma'' - \eta''\|_{L^\infty}
\]
and using the dependence of $\vv H$ on the curve, we deduce that in fact
\[
\left|R_\gamma(x) - R_\eta(x)\right| \leq C\,(|\log\eps|+1)\,\|\gamma - \eta\|_{C^2}.
\]

{\bf Step 2.} \com{Consider} two curves $\gamma, \eta$ \com{in $A$}, and $s,\theta \in S^1$. \com{There} exists an Euclidean motion
\[
M:\R^3\to \R^3, \qquad Mx = m_0 + Ox
\]
with $O \in SO(3)$ such that $|m_0| \leq \|\gamma - \eta\|_{C^0}$, $\|O-I\|_{\R^{3\times 3}} \leq \|\gamma-\eta\|_{C^1}$, and
\[
M(\psi_\eta(s,\theta)) = \psi_\gamma(s,\theta) \com{\qquad \text{ and } \qquad \gamma'(s) = O \eta'(s).}
\]
We calculate
\begin{align*}
\S_{M\eta}(Mx) &= \com{b\otimes \int_{M\eta} k\left(Mx - y\right)\wedge \tau^{M\eta}_y\d\H^1_y}\\
	&\com{= b\otimes \int_\eta k\left(Mx - My\right) \wedge O\tau^\eta_y\d\H^1_y}\\
	&\com{= b\otimes \int_\eta O\left[k(x-y)\right] \wedge O\,\tau^\eta_y\d\H^1_y}\\
	&\com{= b\otimes \int_\eta O\left[k(x-y)\wedge \tau^{\eta}_y\right]\d\H^1_y}\\
	&\com{= b\otimes O \int_\eta k(x-y)\wedge\tau^{\eta}_y\d\H^1_y.}
\end{align*}
since $O$ is an isometry of $\R^3$ (so measure preserving for $\H^1$) and preserves the orientation of $\R^3$, so it commutes with the cross-product. When we write $\S = b\otimes \widehat \S$ \com{and} $R= b\otimes \widehat R$, we can use the first step \com{for $\gamma$ and $M\eta$} to find for $x= \psi_{\gamma}(s,\theta)$ that
\begin{align*}
\left|R_\gamma(x) - R_\eta(x)\right| &= \com{|b|}\,\big|\widehat R_\gamma- \widehat R_\eta\big|\\
	&\leq\com{ |b| \left(\left|O\widehat R_\eta - \widehat R_\eta\right| + |\widehat R_{M\eta} - \widehat R_\gamma|\right)}\\
	&\leq \com{|b| \left(\|M-I\|_{\R^{3\times 3}}\,\|\widehat R_\eta\|_{L^\infty} + C\,(|\log\eps|+1)\,\|\gamma - M\eta\|_{C^2}\right)}\\
	&\leq \com{ C|b| \left(\,\|\gamma-\eta\|_{C^1} (1+|\log\eps|) \|\eta\|_{C^2} + C\,(|\log\eps|+1)\,\left[\|\gamma - \eta\|_{C^2} + \|\eta - M\eta\|_{C^2}\right]\right)}\\
	&\leq  \com{C|b|\,(|\log\eps|+1) \left[\|\gamma-\eta\|_{C^1}\|\eta\|_{C^2} + \|\gamma - \eta\|_{C^2} + \|\gamma-\eta\|_{C^1} \|\eta\|_{C^2}\right]}\\
	&\leq \com{C|b|\,\left(|\log\eps|+1\right)\,\|\gamma-\eta\|_{C^2}.}
\end{align*}
The constant $C$ depend on the set $A$ through the uniform $C^2$-bound and the uniform lower bound on the embeddedness radius.

\end{proof}

\begin{corollary}\label{corollary lipschitz estimate}
If $\gamma\in C^{2,\alpha}(S^1;\R^3)$ is embedded \com{then it holds for the remainder term in \eqref{eq: expansionS} that} $R\in C^{0,\alpha}(\partial B_\eps(\gamma);\R^{3\times 3})$ and \com{for $(s,\theta),(s',\theta') \in S^1 \times S^1$}
\[
\big|R\circ\psi_{\eps,\gamma}(s,\theta) - R\circ\psi_{\eps,\gamma}(s',\theta')\big| \leq C\,|\log\eps|\,\|\gamma\|_{C^{2,\alpha}}\big[|s-s'|^\alpha + |\theta-\theta'|\big]
\]
where the constant $C$ depends on $\gamma$ through its $C^2$-norm and embeddedness radius.
\end{corollary}

\begin{proof}
Assume that $\gamma$ is parametrized by unit speed and set 
\[
A:= \{\gamma_{s_0}\:|\:\gamma_{s_0}(s) := \gamma(s+s_0)\}.
\]
\com{By Lemma \ref{eq: estimaterepara} we have}
\begin{align*}
\big|R_\gamma(\psi_\gamma(s,\theta)) - R_\gamma(\psi_\gamma(s_0+s,\theta))\big| &= \big|R_\gamma(\psi_\gamma(s,\theta)) - R_{\gamma_{s_0}}(\psi_{\gamma_{s_0}}(s,\theta)) \big|\\
	&\leq C\,\com{(|\log\eps|+1)\, \|\gamma - \gamma_{s_0}\|_{C^2}}\\
	&\leq C\,\com{(|\log\eps|+1)\, [\gamma'']_{C^{0,\alpha}}\,s_0^\alpha}
\end{align*}
for fixed $\theta\in S^1$. \com{On the other hand, for fixed $s$ we may just use rotations as above to establish Lipschitz continuity. Indeed,} assume that $\gamma(s) = 0, \gamma'(s) = e_3$, $\psi(s,\theta) = \eps e_1$. \com{Let $M$ be} a rotation such that $M\psi(s,\theta) = M\psi(s,\theta')$. Then by Lemma \ref{eq: estimaterepara}
\begin{align*}
\big|R(\psi(s,\theta)) - R(\psi(s,\theta'))| &= \big| R_\gamma (\eps e_1) - R_{\gamma}(M\,\eps e_1)\big|\\
	&= \big| R_\gamma (\eps e_1) - R_{M\,M^{-1}\gamma}(M\,\eps e_1)\big|\\
	&\leq C\,\left(|\log\eps|+1\right)\,\|\gamma- M^{-1}\gamma\|_{C^2}\\
	&\leq C\,\left(|\log\eps|+1\right)\,\|\gamma\|_{C^2} \,\|M^{-1}-I\|\\
	&= C\,\left(|\log\eps|+1\right)\,\|\gamma\|_{C^2} \,\|I-M\|
\end{align*}
for the standard matrix norm on $\R^3$, which is rotation invariant. The distance between the rotation $M$ and the identity matrix is proportional to the angle by which $M$ rotates around the $e_3$ axis, i.e., $\theta-\theta'$. We conclude that
\[
\big|R_\gamma(\psi_\gamma(s,\theta)) - R_\gamma(\psi_\gamma(s',\theta'))\big| \leq C\,\big(|\log\eps|+1\big)\,\|\gamma\|_{C^{2,\alpha}}\big[|s-s'|^\alpha + |\theta-\theta'|\big].
\]
\end{proof}

\subsection{The effective energy}

Asymptotics for true elastic energies in the language of $\Gamma$-convergence are available due to \cite{MR3375538} in the more general framework of lattice-valued currents. \com{Here, we present a heuristic argument to justify the asymptotic effective energy.}

\begin{theorem}\label{theorem: limitenergy}
Let $\gamma\in C^{2}(S^1;\Omega)$ be an embedded curve, $b\neq 0$ a Burgers vector and $\mu$ the Nye dislocation measure. Then there exists a constant $E\in \R$ such that 
\[
\lim_{\eps\to 0} \left(\Eeff (\mu) - \frac{|b|^2\,\H^1(\gamma)}{4\pi}\,|\log\eps|\right) =E.
\]
\end{theorem}

\begin{proof}
\com{We use the cylindrical coordinates $\Psi$ adapted to $\gamma$ in a tubular neighborhood $B_r(\gamma)$ introduced in Appendix \ref{appendix embeddedness}}. \com{Let $u$ be the unique (up to the addition of a constant) minimizer of $I$ associated to $\gamma$ indirectly through $\S$ (see Theorem \ref{theorem boundary term limit}). We have using the expansions for $\S$, \eqref{eq: expansionS}, and $|\det D\Psi|$, \eqref{eq: expansiondetpsi},  }
\begin{align*}
\Eeff (\mu) &= \frac12 \int_{\Omega_{\eps}(\gamma)} |\S|^2 \dx + I(u)\\
	&= \frac12 \int_{\Omega_\eps\setminus\Omega_r}|\S|^2 \dx + \frac12 \int_{\Omega_r}|\S|^2 \dx + I(u)\\
	&= \frac12 \int_{\Omega_\eps\setminus\Omega_r}|\S|^2\dx + \E^{\mathrm{eff}}_r(\mu)\\
	&= \frac12 \int_{S^1} \int_\eps^r \int_{S^1} |\S|^2 (\Psi(s,\rho,\theta))\,|\det D\Psi|\d\theta\d\rho\ds + \E^{\mathrm{eff}}_r(\mu)\\
	&= \frac{|b|^2}{8\pi^2} \int_{S^1} \int_\eps^r \int_{S^1} \left[\frac1{\rho^2} + O\left(\frac{|\log \rho|}\rho\right)\right] |\gamma'(s)| \left[ \rho+ O(\rho^2)\right]\d\theta\d\rho\ds 
		+ \E^{\mathrm{eff}}_r(\mu)\\
	&=\frac{|b|^2}{8\pi^2}\int_{S^1} |\gamma'(s)| \int_\eps^r \int_{S^1}\frac1\rho +O(|\log \rho|) \d\theta\d\rho\ds + \E^{\mathrm{eff}}_r(\mu)\\
	&= \frac{|b|^2}{8\pi^2} \left(2\pi\cdot \H^1(\gamma) \left[\log r - \log\eps\right] + O(r\,|\log r|)\right) + \E^{\mathrm{eff}}_r(\mu)\\
	&= \frac{|b|^2\,|\log\eps|}{4\pi}\H^1(\gamma) + \left[ \E^{\mathrm{eff}}_r(\mu) + \frac{|b|^2\,\log(r)}{2\pi}\H^1(\gamma)\right] + O(r\,|\log r|).
\end{align*}
Taking the first term on the right hand side to the left and taking the limit $\eps\to0$ on both sides, we obtain the statement. The $O(r\log r)$ term has only a mild dependence on $\eps$ and the limit exists.
\end{proof}

\begin{remark}
Note that the growth of the effective energy is quadratic in $|b|$. This \com{may explain} why typically dislocations with a large Burgers vector split into two (or more) dislocations of very similar length with smaller Burgers vectors $b_1, b_2$ such that $b_1 +b_2 = b$. The physical growth rate therefore is $|b|$ -- see \cite{MR3375538}.

We do not see this in our model since the variations we consider do not account for the possibility of splitting. However, the result is reasonable for all Burgers vectors $b$ in the crystallographic lattice close enough to the origin, namely all Burgers vectors which cannot be decomposed into several lattice vectors $b_1, \dots, b_N$ such that $|b_1|^2 + \dots + |b_N|^2 \leq |b|^2$.
\end{remark}

\subsection{The Peach-Koehler force} 

\com{The Peach-Koehler force on a dislocation line is defined as the negative gradient of the effective energy under a variation of the dislocation curve. In this section, we obtain an asymptotic expansion of this force. Recalling Proposition \ref{prop: variation} we have}
\[
\frac{\d \Eeff(\mu_t)}{dt}\bigg|_{t=0}= -\int_{\partial B_\eps(\gamma)}\left[\frac12\,|\S|^2 \langle \phi,\nu\rangle - w^\phi\langle b, (\S+\nabla u)\nu\rangle - \langle u, \dot \S^\phi\nu\rangle \right]\d\H^2.
\]
For future purposes, we define the renormalized Peach-Koehler force on the dislocation line $\gamma$ at $y\in \gamma$ by
\begin{equation}\label{eq Peach-Koehler force}
F^{PK,\eps}(x) := - \frac{f^{PK, \eps, (1)}(x) + f^{PK,\eps, (2)}(x) + f^{PK, \eps, (3)}(x)}{|\log\eps|},
\end{equation}
where 
\begin{align*}
f^{PK,\eps,(1)}(y):= &\com{-\int_{\partial B_\eps(\gamma)}\frac12\,|\S|^2 \langle \phi,\nu\rangle \d \H^2,} \\
f^{PK,\eps,(2)}(y):= &\com{\int_{\partial B_\eps(\gamma)} w^\phi\langle b, (\S+\nabla u)\nu\rangle \d\H^2,}\\
f^{PK,\eps,(3)}(y)= &\com{\int_{\partial B_\eps(\gamma)} \langle u, \dot \S^\phi\nu\rangle \d\H^2.}
\end{align*}

In the following we calculate the three terms separately, using the adapted cylindrical coordinates $\psi_\eps$ for $\partial B_\eps(\gamma)$ developed in Appendix \ref{appendix embeddedness}. We have seen \com{in Theorem \ref{theorem: limitenergy}} that the energy is, to highest order, given by the arclength functional, and the main result in this chapter is that, to highest order, its variation is given by curvature (which is the variation of the arclength functional).

\begin{theorem}\label{theorem asymptotics force}
The Peach-Koehler force is
\begin{align*}
F^{PK,\eps}
	&= \frac{|b|^2}{4\pi}\,\vv H + R^{PK,\eps}
\end{align*}
where 
\begin{equation}\label{eq: PKestimates}
\lim_{\eps\to 0} \|R^{PK,\eps}\|_{L^\infty(\gamma)} = 0, \qquad \|R^{PK,\eps}\|_{C^{0,\alpha}}\leq C.
\end{equation}
\com{The} limit is uniform in bounded subsets of $C^{2,\alpha}$ with a uniform embededness radius whereas the constant $C$ depends on $\gamma$ through the $C^{2,\alpha}$-norm and the embeddedness radius.
\end{theorem}
\begin{proof}
Theorem \ref{theorem asymptotics force} follows from the upcoming Lemmas \ref{lemma: PKterm1}, \ref{lemma: PKterm2}, and \ref{lemma: PKterm3}.
\end{proof}

\com{In the next section we often write $F^{PK,\eps}_\gamma$ and $R^{PK,\eps}_\gamma$ when the curve $\gamma$ is allowed to evolve.}

\com{
\begin{remark}
It is not clear whether or not our $C^{0,\alpha}$-estimate, \com{\eqref{eq: PKestimates}}, which is only uniform modulo $|\log\eps|$ is optimal, and an improvement in this would likely allow significant advances in the time-dependent setting where an estimate like $\|R^{PK,\eps}\|_{C^{0,\alpha}} \leq C\,|\log\eps|^{-1}$ would allow for a number of new techniques to be applied.
\end{remark}
}

\com{
\begin{remark}
As we have seen, it is fairly involved to obtain estimates on the remainder term even in relatively weak spaces. In some special situations, \com{this} term is much better understood, for example for infinite straight parallel lines. In that \com{case}, there is no curvature and long-ranged interactions between distinct dislocations and between dislocations and the boundary dominate the picture. This is the setting investigated, for example, in \cite{MR2192291,MR3639276}.
\end{remark}
}

\begin{remark}
\com{While we have only estimated $F^{PK,\eps}$ in the $C^{0,\alpha}$-norm and can do no better in a uniform fashion since the leading order term $\vv H$ is only $C^{k-2,\alpha}$ for $C^{k,\alpha}$-curves, we note that $F^{PK,\eps}$ is more regular. 
Indeed, in view of Proposition \ref{prop: variation} we have}
\begin{align*}
F^{PK,\eps}(y) &= \frac{1}2\int_{S^1} |\S|^2\big(\psi_\eps(\pi(y) ,\theta)\big)\, \sqrt{\det g_{\psi_{\eps}} \big(\pi(y),\theta\big)\,}\nu \big(\pi(y),\theta\big) \d\theta\\
	&\qquad - \left[\int_{\partial B_\eps(\gamma)}\big \langle b, (S+\nabla u)\nu_\eps \big\rangle\,k(x-y) \d\H^2_x\right] \wedge \tau_y\\
	&\qquad - \left[\int_{\partial B_\eps(\gamma)} \langle u(x), b\rangle  \left(\partial_{\nu_x}k\right)(x-y)\d\H^2_x\right]\wedge \tau_y.
\end{align*}
\com{Note that the terms involving $S, u, \nabla u$ and $k$ are $C^\infty$-smooth and the regularity of the expression is dictated by the regularity of the closest point projection $\pi$ and its composition with $\psi_\eps, \det g_{\psi_{\eps}}, \nu$, as well as the regularity of the unit tangent. Since all of these functions are of class $C^{k-1,\alpha}$ if $\gamma$ is an embedded $C^{k,\alpha}$-curve, we see that $F^{PK,\eps}\in C^{k-1,\alpha}$, although the bounds degenerate as $\eps\to0$.}
\end{remark}

\begin{lemma}\label{lemma: PKterm1}
\com{The first contribution to the Peach-Koehler force is}
\begin{align*}
f^{PK, \eps, (1)} &:= \com{\frac{1}2\int_{S^1} |\S|^2\big(\psi_\eps(s,\theta)\big)\, \sqrt{\det g_{\psi_{\eps}}(s,\theta)\,}\nu_\eps(s,\theta) \d\theta}\\
	&= \com{-\frac{|b|^2}{8\pi}|\log\eps|\vv H + R^{PK, \eps, (1)},}
\end{align*}
\com{where the remainder term $R^{PK, \eps, (1)}$ is uniformly bounded in $L^\infty(\gamma; \R^3)$ and $|\log\eps|^{-1}\,R^{PK,\eps,(1)}$ is uniformly bounded in $C^{0,\alpha}(\gamma;\R^3)$.
The bounds depend on $\gamma$ through its embeddedness radius, the $C^2$-norm, and its $C^{2,\alpha}$-norm, respectively.}
\end{lemma}

Note that it is $-f^{PK,\eps,(1)}$ that contributes to the variation since we $f^{PK,\eps,(1)}$ expresses the (positive) variation of energy.

\begin{proof}
We use the adapted cylindrical coordinates on $\partial B_{\eps}(\gamma)$ as introduced in Appendix \ref{appendix embeddedness}. 
Then we find that
\begin{align*}
\int_{\partial B_\eps(\gamma)}\frac12\,|\S|^2 \langle \phi,\nu_\eps\rangle \d\H^2 &= \int_{S^1}\int_{S^1}|\S|^2\big(\psi_\eps(s,\theta)\big)\,\langle \phi(s), \nu_\eps(s,\theta)\rangle \,\sqrt{\det g_{\psi_{\eps}}(s,\theta)\,}\d\theta\ds\\
	&= \int_{S^1} \left\langle \int_{S^1} |\S|^2\big(\psi_\eps(s,\theta)\big)\,\sqrt{\det g_{\psi_{\eps}}(s,\theta)}\,\nu_\eps(s,\theta) \d\theta, \:\phi(s)\right\rangle \ds,
\end{align*}
and we define
\[
f^{PK, \eps, (1)}(s) := -\frac{1}2\int_{S^1} |\S|^2\big(\psi_\eps(s,\theta)\big)\, \sqrt{\det g_{\psi_{\eps}}(s,\theta)\,}\nu_\eps(s,\theta) \d\theta.
\]
Using the asymptotic expansion of $\S$ in \eqref{eq: expansionS}, we observe that
\begin{align*}
|\S(x)|^2 &= \left|\frac1{2\pi} \,b \otimes \left(\frac1{\eps}\, \tau_{\pi(x)}\wedge \nu_x + \frac{|\log\eps|}{2}\, \tau_{\pi(x)}\wedge  \vv H_{\pi(x)} + R(x) \right)\right|^2\\
	&= \frac{|b|^2}{4\,\pi^2}\left( \frac1{\eps^2} + \frac{|\log\eps|}\eps \langle \vv{H}_{\pi(x)}, \nu_x\rangle + \frac2\eps \langle \tau_{\pi(x)} \wedge\nu_x, R(x)\rangle \right. \\
	& \hspace{1.5cm} \left. + \frac{|\log\eps|^2}4\,|\vv{H}_{\pi(x)}|^2 + |\log\eps|\,\langle \tau\wedge \vv{H}_{\pi(x)}, R(x)\rangle + |R(x)|^2 \right)\\
	&= \frac{|b|^2}{4\,\pi^2\,\eps} \left(\frac1\eps + |\log\eps|\,\langle \vv{H}_{\pi(x)},\nu_x\rangle + \widetilde R_\eps(x)\right)
\end{align*}
where, \com{by Theorem \ref{thm: expansionK}}, $\widetilde R_\eps$ is uniformly bounded in $L^\infty$ and $C^{0,\alpha}$-H\"older continuous with H\"older constant blowing up no faster than $O(|\log\eps|)$. 
Without loss of generality we may assume that $\gamma$ is parametrized by arclength. Then we note that \com{(see \eqref{eq: areadistortion} in Appendix \ref{appendix : cylindrical coordinates})}
\[
\sqrt{\det g_{\psi_{\eps}}(s,\phi)} = \eps - \eps^2 \delta g_{\psi_{\eps}}(s,\phi)
\]
where $\delta g_{\psi_{\eps}}$ is a uniformly bounded and uniformly $C^{0,\alpha}$-H\"older continuous function (with bounds depending on the curvature of $\gamma$). It follows that
\begin{align*}
f^{PK, \eps, (1)}(s) &= -\frac12 \int_{S^1} \frac{|b|^2}{4\,\pi^2\,\eps} \left(\frac1\eps + |\log\eps|\,\langle \vv{H}_{\gamma(s)},\cos\theta\,n_1(s) + \sin\theta\,n_2(s)\rangle + \widetilde R_\eps\right)\\
		&\hspace{4cm}\cdot \left(\eps - \eps^2 \delta g_{\psi_{\eps}}(s,\theta)\right) \left(\cos\theta\,n_1(s) + \sin\theta\,n_2(s)\right)\d\theta\\
	&= \frac{-|b|^2}{8\,\pi^2} \int_{S^1}\frac{\cos\theta\, n_1 + \sin\theta\, n_2}\eps \\
		& \hspace{2.5cm}+ |\log\eps|\langle \vv{H}_{\gamma(s)}, \cos\theta \,n_1 + \sin\theta\,n_2\rangle \,(\cos\theta \,n_1 + \sin\theta\, n_2) + O(1)\d\theta,
\end{align*}
where the $O(1)$-term is uniformly bounded in $L^\infty$ and is $C^{0,\alpha}$-H\"older continuous with H\"older constant growing no faster than $|\log\eps|$. The first term in the integral vanishes due to symmetry, so 
\begin{align*}
f^{PK, \eps, (1)}(s) &= \frac{-|b|^2}{8\,\pi^2} \int_{S^1} |\log\eps|\langle \vv{H}_{\gamma(s)}, \cos\theta \,n_1 + \sin\theta\,n_2\rangle \,(\cos\theta \,n_1 + \sin\theta\, n_2) \d\theta + R^{PK, \eps, (1)}\\
	&= \frac{-|b|^2}{8\pi}|\log\eps|\vv{H}_{\gamma(s)} + R^{PK, \eps, (1)}
\end{align*}
where the error term satisfies the same estimates as the $O(1)$ term before. To see that the last identity holds, use \com{the fact that $\int_{S^1}\cos(\theta)^2 \, d\theta = \int_{S^1}\sin(\theta)^2 \, d\theta =\pi$.}
\end{proof}

\begin{lemma}\label{lemma: PKterm2}
\com{The second contribution to the Peach-Koehler force is}
\begin{align*}
f^{PK, \eps, (2)} &:= \com{\left[\int_{\partial B_\eps(\gamma)}\big \langle b, (\S+\nabla u)\nu_x \big\rangle\,k(x-y) \d\H^2_x \right] \wedge \tau_y}\\
	&= \com{-\frac{|b|^2}{8\pi}|\log\eps|\vv H + R^{PK, \eps, (2)}},
\end{align*}
\com{where the remainder term $R^{PK, \eps, (2)}$ is uniformly bounded in $L^\infty(\gamma; \R^3)$ and $|\log\eps|^{-1}\,R^{PK,\eps,(2)}$ is uniformly bounded in $C^{0,\alpha}(\gamma;\R^3)$.
The bounds depend on $\gamma$ through its embeddedness radius, the $C^2$-norm, and its $C^{2,\alpha}$-norm, respectively.}
\end{lemma}

\begin{proof}
We note, \com{using Proposition \ref{prop: derivativeS},} that
\begin{align*}
\int_{\partial B_\eps(\gamma)}&w^\phi\langle b, (\S+\nabla u)\nu \rangle \d\H^2\\
	 &= \int_{\partial B_\eps(\gamma)} \left( \int_\gamma \left[k(x-y)\wedge \tau_y\right]\cdot \phi(y)\d\H^1_y\right) \big\langle b, (\S+\nabla u)(x)\,\nu_x\big\rangle \d\H^2_x\\
	 &= \int_\gamma \left(\int_{\partial B_\eps(\gamma)}\big \langle b, (\S+\nabla u)(x)\nu_x\big\rangle\,k(x-y) \d\H^2_x \wedge \tau_y \right)\cdot \phi(y) \d\H^1_y.
\end{align*}
We set
\[
f^{PK, \eps, (2)}(y) := \left[\int_{\partial B_\eps(\gamma)}\big \langle b, (\S+\nabla u)\nu_x \big\rangle\,k(x-y) \d\H^2_x \right] \wedge \tau_y.
\]
Using the asymptotic expansion of $\S$\com{, \eqref{eq: expansionS}}, we see that 
\begin{align*}
\S(x)\nu_x &= \frac1{2\pi} \,\left[b \otimes \left(\frac1{\eps}\, \tau_{\pi(x)}\wedge \nu_x + \frac{|\log\eps|}2\, \tau_{\pi(x)}\wedge  \vv{H}_{\pi(x)} + R(x) \right)\right]\nu_x\\
	&= \frac1{2\pi}\left(\frac{\langle \tau_{\pi(x)} \wedge\nu_x, \nu_x\rangle}\eps + \frac{|\log\eps|}{2}\,\langle \tau\wedge \vv{H}_{\pi(x)}, \nu_x\rangle + \langle R(x),\nu_x\rangle\right) b\\
	&= \frac{1}{2\pi} \left(\frac{|\log\eps|}2\,\langle \tau_{\pi(x)}\wedge \vv{H}_{\pi(x)}, \nu_x\rangle + \langle R(x),\nu_x\rangle\right) b,
\end{align*}
and so
\begin{align*}
f^{PK, \eps, (2)}(y) &= \frac{|b|^2}{2\pi} \left[\int_{\partial B_\eps(\gamma)}\left(\frac{|\log\eps|}2\,\langle \tau_{\pi(x)}\wedge \vv H_{\pi(x)}, \nu_x\rangle + \langle R(x),\nu_x\rangle\right)\,k(x-y)\d\H^2_x\right]\wedge\tau_y\\
	&\qquad + \left[\int_{\partial B_\eps(\gamma)}\left\langle b, \nabla u(x)\,\nu_x\right\rangle\,k(x-y) \d\H^2_x\right]\wedge\tau_y\\
	%&\qquad +  \int_{\partial B_\eps(\gamma)}\left\langle b, \left(\nabla u(x) - \nabla u(\pi(x))\right)\,\nu_x\right\rangle\,k(x-y) \d\H^2_x\wedge\tau_y.
	&=\frac{|b|^2}{2\pi}\left[\int_{\partial B_\eps(\gamma)}\left(\frac{|\log\eps|}2 \,\langle \tau_{\pi(x)}\wedge \vv H_{\pi(x)}, \nu_x\rangle + \langle R(x),\nu_x\rangle\right)\,k(x-y)\d\H^2_x\right]\wedge\tau_y\\
	&\qquad + \left[\int_{\partial B_\eps(\gamma)}\left\langle \nabla u^Tb, \nu_x\right\rangle\,k(x-y) \d\H^2_x\right]\wedge\tau_y.
\end{align*}
To compute the \com{latter two} integrals, we again make simplifying assumptions. Without loss of generality, we may assume that $y=0$ and $\tau_y= e_3$, \com{and we replace $\gamma$ by its first order Taylor polynomial $P_\gamma(s) = s\,e_3$ when calculating the integral to leading order.} This allows us to explicitly \com{select} normal fields $n_1=e_1, n_2=e_2$ and \com{evaluate} for a fixed vector $A$
\begin{align*}
&\left[\int_{\partial B_\eps(P_\gamma)} \langle A,\nu_x\rangle\, k(x-y)\d\H^2\right] \wedge \tau_y\\
	=&-\left[ \int_{-r}^r\int_{S^1} \langle A, \nu_{\psi_\eps(s,\theta)}\rangle \,\frac{1}{4\pi}\,\frac{1}{[\eps^2+s^2]^{3/2}}
		\begin{pmatrix}\eps\cos\theta\\ \eps\sin\theta\\ s\end{pmatrix}\eps\d\theta\ds\right] \wedge \tau_y\\ 
	=& -\frac{1}{4\pi}\left[\int_{-r}^r \int_{S^1}\left\langle A, \begin{pmatrix}\cos\theta\\ \sin\theta\\ 0\end{pmatrix}\right\rangle 
		\frac1{[\eps^2 + s^2]^{3/2}} \begin{pmatrix}\cos\theta\\ \sin\theta\\ 0\end{pmatrix}\,\eps^2\d\theta\ds \right]\wedge \tau_y\\
	=& -\left[\frac1{4\pi} \left(\int_{-r}^r \frac{\eps^2}{[\eps^2 + s^2]^{3/2}} \ds\right)\left( \int_{S^1} \left\langle A, \begin{pmatrix}\cos\theta\\ \sin\theta\\ 0\end{pmatrix}\right\rangle \begin{pmatrix}\cos\theta\\ \sin\theta\\ 0\end{pmatrix}\d\theta\right)\right]\wedge \tau_y\\
	=& -\frac1{4\pi} \left(\int_{-r}^r \frac{\eps^2}{\eps^3\,[1+ (s/\eps)^2]^{3/2}} \eps\,\frac1\eps\ds\right) \, \pi A\wedge \tau_y\\
	=& -\frac14 \left(\int_{-r/\eps}^{r/\eps} \frac1{[1+s^2]^{3/2}}\ds\right) \:A\wedge \tau_y\\
	=& -\frac14\left(2\,\frac{r/\eps}{\sqrt{1+(r/\eps)^2}}\right) A\wedge\tau_y\\
	= &-\frac12 \big(1+O(\eps)\big)\,A\wedge\tau_y.
\end{align*}
If now $A$ is a $C^{0,\alpha}$-continuous vector field \com{with $A_0 = A(\gamma(0))$}, then
\begin{align*}
\bigg|\int_{\partial B_\eps(P_\gamma)} &\langle A_x,\nu_x\rangle\, k(x-y)\d\H^2\wedge \tau_y - \int_{\partial B_\eps(P_\gamma)} \langle A_0,\nu_x\rangle\, k(x-y)\d\H^2\wedge \tau_y\bigg|\\
	&\leq \frac1{4\pi}\int_{-r}^r\int_{S^1}\big|A_{\psi_\eps(s,\theta)} - A_0\big|\,\frac1{|\psi_\eps(s,\theta)|^2}\eps\d\theta\ds\\
	&\leq \frac1{4\pi}\cdot 2\pi\eps \int_{-r}^r [A]_{C^{0,\alpha}}\,\left(s^2 + \eps^2\right)^\frac\alpha2 \,\frac1{s^2+\eps^2}\ds\\
	&= \frac{[A]_{C^{0,\alpha}}}{2} \int_{-r}^r \eps^{\alpha -2} \left[(s/\eps)^2 +1\right]^{\frac\alpha2 -1} \eps^2\,\frac1\eps\ds\\
	&= \frac{[A]_{C^{0,\alpha}}}{2}\eps^{\alpha} \int_{-r/\eps}^{r/\eps} \frac{1}{[1+s^2]^{1-\frac\alpha2}}\ds\\
	&\leq \frac{C_\alpha\,[A]_{C^{0,\alpha}}}{2}\eps^{\alpha}
\end{align*}
since the integral converges for $\alpha<1$. As before in the asymptotic expansion of $\S$, we can show that the error we make when approximating $\gamma$ by the Taylor polynomial is small in $\eps$ -- note that coefficient integrals $I^1, I^2$ are, to leading order, independent of $\eps$, so it can be \com{shown} that the {\em first order} approximation is sufficient here. Above, mean curvature entered the picture through the expansion of $k$ and $\S$ respectively, not the area element where the curvature term appears only multiplied by $\eps$. The details are similar to the calculations above, using in addition that 
\[
\big|P_\gamma'(s) - \gamma'(s)\big| \leq \|\gamma''\|_{L^\infty}\,s, \quad\big|P_\gamma(s) - \gamma(s)\big| \leq \|\gamma''\|_{L^\infty}\,s^2, 
\] 
which imply that 
\[
\big|\psi_{\gamma,\eps}(s,\theta) - \psi_{P_\gamma,\eps}(s,\theta)\big| \leq C\,\|\gamma''\|_{L^\infty} s 
\]
when the fields $n_1, n_2$ are chosen appropriately. It follows that
\begin{align*}
f^{PK, \eps, (2)}(y) &= \frac{|b|^2}{2\pi} \left( -\frac12\right) \left[\frac{|\log\eps|}2 \left(\tau_y \wedge \vv H_y\right) \wedge\tau_y + \tilde R_y\wedge \tau_y + \left(\nabla u^T(y)\right)b + O(\eps^{\alpha}|\log\eps|)\right]\\
	&= \frac{|b|^2}{4\pi} \left[ - \frac{|\log\eps|}2 \,\vv H_y \com{-} \tilde R(y) \wedge \tau_y \com{-} \nabla u^T(y)\cdot b + O(\eps^{\alpha}|\log\eps|)\right],
\end{align*}
where $\tilde R$ is associated to the remainder term $R$ in $\S$, see Theorem \ref{thm: expansionK}. To see that $\left(\tau \wedge \vv H\right)\wedge \tau = \vv H$, note that the two vectors are orthogonal and wlog aligned with $e_1,e_2$. The same arguments as above show that 
\[
\big\|f^{PK, \eps, (2)}_\gamma - f^{PK, \eps, (2)}_\eta\big\|_{L^\infty} \leq C\,|\log\eps|\,\|\gamma-\eta\|_{C^2},
\]
which implies the $C^{0,\alpha}$-H\"older continuity.
\end{proof}

The third term heavily involves the function $u$, which encodes the interaction of the dislocation line with the domain boundary. As this term does not contribute to the energy to leading order, we \com{also} do not expect it to contribute to the Peach-Koehler force to leading order.

\begin{lemma}\label{lemma: PKterm3}
\com{The third contribution to the Peach-Koehler force}
\[
\com{f^{PK, \eps, (3)}(y)  = - \left[\int_{\partial B_\eps(\gamma)} \langle u(x), b\rangle  \left(\partial_{\nu_x}k\right)(x-y)\d\H^2_x\right]\wedge \tau_y}
\]
\com{is uniformly bounded in $L^\infty(S_L^1; \R^3)$ and $|\log\eps|^{-1}\,f^{PK,\eps,(3)}$ is uniformly bounded in $C^{0,\alpha}(S^1_L;\R^3)$. 
The bounds depend on $\gamma$ through its embeddedness radius, the $C^2$-norm, and its $C^{2,\alpha}$-norm, respectively.}
\end{lemma}

\begin{proof}
\com{By Proposition \ref{prop: derivativeS} we have}
 \begin{align*}
\int_{\partial B_\eps(\gamma)}& \langle u, \dot \S^\phi\nu_\eps\rangle \d\H^2
	 =\int_{\partial B_\eps(\gamma)} \langle u, -(b\otimes \nabla w^\phi) \nu_\eps\rangle \d\H^2\\
	 &= - \int_{\partial B_\eps(\gamma)} \langle u, b\rangle \,\langle \nabla w^\phi, \nu_\eps\rangle \d\H^2\\
	 &= - \int_{\partial B_\eps(\gamma)} \langle u, b\rangle \,\left\langle \nabla_x \int_\gamma \left[k(x-y)\wedge \tau_y\right]\cdot\phi(y)\d\H^1_y,\: \nu_\eps(x)\right\rangle\d\H^2_x\\
	 &= - \int_{\partial B_\eps(\gamma)} \langle u, b\rangle \int_\gamma \left[\left(\partial_{\nu_x}k\right)(x-y)\wedge \tau_y\right]\cdot\phi(y)\d\H^1_y\d\H^2_x\\
	 &= -\int_\gamma \left\langle \left[\int_{\partial B_\eps(\gamma)} \langle u(x), b\rangle  \left(\partial_{\nu_x}k\right)(x-y)\d\H^2_x\right]\wedge \tau_y, \:\phi(y)\right\rangle\d\H^1_y,
\end{align*}
and we set
\begin{align*}
 f^{PK, \eps, (3)}(y) &:= - \left[\int_{\partial B_\eps(\gamma)} \langle u(x), b\rangle  \left(\partial_{\nu_x}k\right)(x-y)\d\H^2_x\right]\wedge \tau_y\\
 	&= - \left[\int_{\partial B_\eps(\gamma)} \langle u(x), b\rangle\:  \frac{-1}{4\pi\,|x-y|^3}\left(I - 3\,\frac{x-y}{|x-y|}\otimes \frac{x-y}{|x-y|}\right)\nu_x \d\H^2_x\right]\wedge \tau_y\\
	&= \frac1{4\pi} \left[\int_{\partial B_\eps(\gamma)} \frac{\langle u(x), b\rangle}{|x-y|^3} \left(\nu_x - 3\left\langle \frac{x-y}{|x-y|}, \nu_x\right\rangle\:\frac{x-y}{|x-y|}\right)\d\H^2_x\right]\wedge\tau_y
%	&= \frac1{4\pi} \int_{\partial B_\eps(\gamma)} \frac{\langle u(\pi(x)), b\rangle}{|x-y|^3} \left(\nu_x - 3\left\langle \frac{x-y}{|x-y|}, \nu_x\right\rangle\:\frac{x-y}{|x-y|}\right)\d\H^2_x\wedge\tau_y\\
%		&\qquad + \frac1{4\pi} \int_{\partial B_\eps(\gamma)} \frac{\langle u(x) - u(\pi(x)), b\rangle}{|x-y|^3} \left(\nu_x - 3\left\langle \frac{x-y}{|x-y|}, \nu_x\right\rangle\:\frac{x-y}{|x-y|}\right)\d\H^2_x\wedge\tau_y.
\end{align*}
To calculate this integral, as before we replace $\gamma$ \com{locally} by its first order Taylor expansion $P_\gamma(s) = s\,e_3$ and assume that $y=0$. \com{The remainder term is as before in $O(1)$ (we skip the proof here).}
For the approximated curve we use the adapted cylindrical coordinates $\psi_{\eps}$ as introduced in Appendix \ref{appendix embeddedness} and note that
\begin{align*}
f^{PK,\eps,(3)}(y) &= \frac1{4\pi} \left[\int_{-r}^r\int_{S^1} \frac{\langle u(\psi_\eps(s,\theta)), b\rangle}{[s^2+\eps^2]^{3/2}}\bigg[\begin{pmatrix} \cos\theta\\ \sin\theta\\ 0\end{pmatrix} \right.\\
	&\hspace{0.1cm}- \left. 3\left\langle \frac{1}{\sqrt{s^2+\eps^2}}\begin{pmatrix} \eps \cos\theta\\ \eps\sin\theta\\ s\end{pmatrix}, \begin{pmatrix}\cos \theta\\ \sin\theta\\ 0\end{pmatrix}\right\rangle \,\frac{1}{\sqrt{s^2+\eps^2}}\begin{pmatrix} \eps \cos\theta\\ \eps\sin\theta\\ s\end{pmatrix}\bigg]\,\eps\d\theta\ds\right] \wedge\tau_y \com{+O(1)}\\
	&= \frac1{4\pi} \left[\int_{-r}^r\int_{S^1} \frac{\langle u(\psi_\eps(s,\theta)), b\rangle}{[s^2+\eps^2]^{3/2}} \begin{pmatrix} \left[1 - \frac{3\eps^2}{s^2+\eps^2}\right] \cos\theta\\  \left[1 - \frac{3\eps^2}{s^2+\eps^2}\right] \sin\theta\\ \com{-\frac{3s}{s^2+\eps^2}}\end{pmatrix}\eps\d\theta\ds\right] \wedge\begin{pmatrix}0\\0\\1\end{pmatrix} + \com{O(1)}.
\end{align*}
We can ignore the last component of the integral which will disappear when taking the wedge product with \com{$\tau_y = \begin{pmatrix}
0\\0\\1 \end{pmatrix}$}. \com{Next, we notice that since $\gamma\cc\Omega$, we can find a smoothly bounded $\Omega'\cc\Omega$ such that $B_r(\gamma) \subset \Omega'$. Classical elliptic regularity theory then yields (see Theorem \ref{eq: Neumann})}
\[
\|u\|_{C^{k,\alpha}(\overline{\Omega'})} \leq C(k,\alpha,\Omega',\Omega)\,\|u\|_{H^1(\Omega)} \leq C\|\S\nu\|_{L^2(\partial\Omega)} \leq C
\]
\com{for any $k\in\N$ and $\alpha\in[0,1]$. When choosing $k=2$, this allows us to use the approximation }
\[
u(\psi_\eps(s,\theta)) = u(\gamma(s)) + (\nabla u)_{\gamma(s)}\,\eps\nu_{\psi_\eps(s,\theta)} + O(\eps^2),
\]
\com{where the $O(\eps^2)$-term is uniform in $s$ and $\theta$, so that}
\[
\langle u(\psi_\eps(s,\theta)), b\rangle = \left\langle u(\gamma(s)), b\right\rangle + \eps \left\langle (\nabla u)_{\gamma(s)}^T b, \nu_{\psi_\eps(s,\theta)} \right\rangle + O(\eps^2).
\]
Consequently,
\begin{align*}
f^{PK,\eps,(3)}(y) &= \frac1{4\pi} \left[\int_{-r}^r\int_{S^1} \frac{\left\langle u(\gamma(s)), b\right\rangle}{[s^2+\eps^2]^{3/2}} \begin{pmatrix} \left[1 - \frac{3\eps^2}{s^2+\eps^2}\right] \cos\theta\\  \left[1 - \frac{3\eps^2}{s^2+\eps^2}\right] \sin\theta\\ 0\end{pmatrix}\eps\d\theta\ds \right] \wedge\begin{pmatrix}0\\0\\1\end{pmatrix} \\
		&+ \frac1{4\pi} \left[\int_{-r}^r\int_{S^1} \frac{\eps\left\langle (\nabla u)_{\gamma(s)}^T b, \nu_{\psi_\eps(s,\theta)} \right\rangle}{[s^2+\eps^2]^{3/2}} \begin{pmatrix} \left[1 - \frac{3\eps^2}{s^2+\eps^2}\right] \cos\theta\\  \left[1 - \frac{3\eps^2}{s^2+\eps^2}\right] \sin\theta\\ 0\end{pmatrix}\eps\d\theta\ds\right] \wedge\begin{pmatrix}0\\0\\1\end{pmatrix} + \com{O(1)}\\
		&= \frac1{4\pi} \int_{-r}^r \frac{\eps^2\,\left[1 - \frac{3\eps^2}{s^2+\eps^2}\right]}{[s^2+\eps^2]^{3/2}} \left(\int_{S^1}\left\langle (\nabla u)_{\gamma(s)}^Tb, \nu_{\psi_\eps(s,\theta)}\right\rangle\,\nu_{\psi_\eps(s,\theta)}\d\theta\right)\ds\\
		&= \frac1{4\pi} \int_{-r}^r \frac{\eps^2\,\left[1 - \frac{3\eps^2}{s^2+\eps^2}\right]}{[s^2+\eps^2]^{3/2}}\, \pi\,(\nabla u)_{\gamma(s)}^Tb\ds,
\end{align*}
since the first integral, integrated with respect to $\theta$ first, vanishes. 
%The integral over the circle is calculated as above, and we have calculated a number of similar integrals above, so we leave it to the reader to check that the final integral is uniformly finite in $\eps$. As this integral is finite, the error term from the Taylor expansion of $u$ which was one order smaller in $\eps$ is $\eps$-small in this expression. 

Finally, we remark that similar symmetry arguments can be employed if we do not linearize $\gamma$, but keep the non-linear cylindrical coordinates and reflect directly inside the integral by shifting the integration domain:
\[
\int_0^{2\pi}f(\theta)\d\theta = \frac12\int_0^{2\pi} f(\theta) + f(\pi+\theta)\d\theta.
\]
The remaining terms are again estimated by the $C^2$-norm of $\gamma$. A notable difference to previous estimates is that the $C^2$-norm enters the estimate despite our use of a {\em first order} Taylor expansion because we are estimating the modulus of continuity of $n_1, n_2$ which are only as smooth as $\gamma'$.
\end{proof}

\section{Convergence of the evolution equations}\label{section evolution}

\subsection{$L^2$-dissipation}

In this section, we provide a convergence result for the gradient flows of the energies $\Eeff$ to codimension $2$ curve shortening flow, which is conditional on the existence and regularity of the gradient flows for $\Eeff$. The problem in this section can be summarized as follows.

\begin{enumerate}
\item If $\gamma\in C^{2,\alpha}$, then $F^{PK,\eps}_\gamma$ approaches mean curvature in $C^{0,\beta}$ for all $\beta<\alpha$. The convergence is uniform on suitable subsets of $C^{2,\alpha}$.

\item From this, we can deduce that if the equation $\Gamma^{\eps}_t =F^{PK,\eps}_\gamma$ has a solution and $\Gamma^\eps$ in the parabolic H\"older space $C^{2,\alpha}_{1,\alpha/2}$ uniformly in $\eps$, then $\Gamma^\eps\to \Gamma$ in $C^{2,\beta}_{1,\beta/2}$ for all $\beta<\alpha$ where $\Gamma$ is the solution of curve shortening flow.

\com{Recall from Section \ref{section preliminaries} that} $C^{2,\alpha}_{1,\alpha/2}$ is the space of functions which have one derivative in time and two in space such that $\partial_tu$ and $D^2u$ are H\"older-continuous of order $\alpha/2$ in time and $\alpha$ in space.

\item The proof of existence and regularity is complicated by the fact that we are not dealing directly with a \com{PDE}, but a non-local equation which approximates a PDE in the singular limit. When expanding the near-field term in the singular part $\S$ of the strain (and by proxy the Peach-Koehler force), the dependence on mean curvature is fully non-linear through integrals of the form
\[
\int_{-r/\eps}^{r/\eps} \frac{1}{[1+s^2 + \eps^2|\vv{H}|^2\,s^4]^{3/2}}\d s.
\]
\com{F}or the error estimate we need the $C^{2,\alpha}$-semi-norm of the space curve, which only makes the error term small in the $C^{0,\beta}$-norm for $\beta<\alpha$.

\end{enumerate}

We will not provide a full proof of existence, regularity and convergence but only show convergence conditional on existence and regularity. In Section \ref{section h1 dissipation} we will prove existence for a gradient flow with an $H^1$-type dissipation similar to \cite{hudson}.

\begin{theorem}\label{theorem motion l2}
Let $\gamma_0\in C^{2,\alpha}(S^1; \Omega)$ be an embedded curve. Assume that there exist $T>0$, $\eps_0>0$ such that the evolution equations
\[
\begin{pde}
\partial_t \Gamma &= \frac{|b|^2}{4\pi} \vv H &t>0,\\
\Gamma &= \gamma_0 &t=0,
\end{pde},
\qquad\qquad
\begin{pde}
\partial_t\Gamma^\eps &= F^{PK,\eps}_\Gamma &t>0,\\
\Gamma^\eps &= \gamma_0 &t=0,
\end{pde}
\]
have solutions in the parabolic H\"older space $C^{2,\alpha}_{1,\alpha/2}([0,T]\times S^1; \R^3)$ such that 
\begin{enumerate}
\item the curves $\Gamma(t,\cdot)$ are embedded in $\Omega$ for all $0< t<T$, and
\item there exists a constant $K>0$ independent of $0<\eps<\eps_0$ such that
\[
\|\Gamma^\eps\|_{C^{2,\alpha}_{1,\alpha/2}([0,T]\times S^1; \R^3)} \leq K.
\]
\end{enumerate}
Then $\Gamma^\eps\to \Gamma$ in $C^{2,\beta}_{1,\beta/2}([0,T']\times S^1; \R^3)$ for all $T'<T$ and $\beta<\alpha$.\end{theorem}

\begin{proof}
Since $\|\Gamma^\eps\|_{C^{2,\alpha}_{1,\alpha/2}}\leq K$, using the closed compact embeddings of H\"older spaces \com{we find} $\tilde \Gamma \in C^{2,\alpha}_{1,\alpha/2}$ such that $\Gamma_\eps \to \tilde \Gamma$ in $C^{2,\beta}_{1,\beta/2}$ for all $\beta<\alpha$. Since $\tilde \Gamma (0, \cdot) = \gamma_0$ \com{is embedded}, it follows that $\Gamma(t,\cdot)$ is an embedded curve for all small enough $0<t< \tilde T$ with a uniform lower bound on the embeddedness radius \com{by continuity}. Due to the $C^{2,\beta}$-convergence, the same is true for $\Gamma_\eps$ for all small enough $\eps$ and \com{all} $T'<\tilde T$.

By Theorem \ref{theorem asymptotics force}, we find that 
\[
F^{PK,\eps}_{\Gamma^\eps(t,\cdot)} - \frac{|b|^2}{4\pi}\vv H_{\tilde \Gamma(t,\cdot)} = F^{PK,\eps}_{\Gamma^\eps(t,\cdot)} - \frac{|b|^2}{4\pi}\vv H_{\tilde \Gamma^\eps(t,\cdot)} +\frac{|b|^2}{4\pi}\vv H_{\tilde \Gamma^\eps(t,\cdot)} - \frac{|b|^2}{4\pi}\vv H_{\tilde \Gamma(t,\cdot)}
\]
converges to $0$ in $L^\infty$. It follows that 
\[
\tilde \Gamma_t = \lim_{\eps\to 0} \Gamma^\eps_t = \lim_{\eps\to 0} F^{PK,\eps}_{\Gamma^\eps} = \frac{|b|^2}{4\pi}\vv H_{\com{\tilde \Gamma}},
\]
so that $\tilde \Gamma=\Gamma$ on $[0,T']$ since the curve shortening flow has a {\em unique} solution (see, for example, \cite{MR664497}). Iterating the argument, we can choose any $T' < \tilde T= T$.
\end{proof}

Clearly, the proof hinges on regularity and existence of the solution to the evolution problem for the effective energy $\Eeff$. However, assuming that such a priori bounds hold, we could prove the same result for more physically relevant problems.

\begin{remark}\label{remark physical dissipation}
It is widely accepted that dislocations move differently depending on whether they are of edge type, i.e., $b\bot \tau$, or screw type, i.e., $b\parallel \tau$. Edge dislocations move only in the plane spanned by $b$ and $\tau$ while screw dislocations may move in a set of crystallographic directions $s_1,\dots, s_N$, one of which is chosen by a {\em maximal dissipation criterion}. Real dislocations mix both types, and except on segments where $b$ and $\tau$ are almost parallel, the motion \com{resembles} that of edge dislocations (but see \cite{liu2014interfacial} for a more accurate picture).

This motion law is non-unique when two directions are equally favourable for energy dissipation. Consider the simple situation \com{in which} two infinite straight dislocations, modelled by locations $x_1, x_2$, in the plane attract each other with a force $f \parallel (x_2 -x_1)$. If $x_1(0) = -e_1$, $x_2(0) = e_1$, and the crystallographic directions are $s_1 = e_1+e_2$, $s_2 = e_1 -e_2$, then moving in either direction is equally advantageous energetically. In fact, it is possible to alternate between the two directions quickly -- this gives rise to so-called fine cross slip where the the macroscopic motion is only in the convex hull of the maximally dissipating directions \cite{MR3323554}.

Formally, $(x_1, x_2)\in \R^4$ solves the differential inclusion $(\dot x_1, \dot x_2) \in -\mathsf{conv}\left(\partial \E(x_1,x_2)\right)$ \com{which means that the time derivative of $(x_1,x_2)$ lies in the convex hull of the directions of maximal dissipation}. It is easy to see that 
\[
x_1(t) := \begin{pmatrix} -1 + t\\ \lambda t\end{pmatrix}, \qquad x_2(t) := \begin{pmatrix} 1- t\\ \lambda t\end{pmatrix}, \qquad t\in \R
\]
satisfy the differential inclusion for any $\lambda\in[-1,1]$, and therefore that solutions are generally not unique. It is highly unlikely that such an irregular motion law is going to produce the high degree of spatial regularity we require. We propose a more regular motion law of the form 
\begin{equation}\label{eq physical}
\begin{pde}
\partial_t\Gamma &= m\left(\frac{\gamma'}{|\gamma'|}, b, F^{PK,\eps}_\Gamma\right) &t>0,\\
\Gamma &= \gamma_0 &t=0,
\end{pde}
\end{equation}
where 
\[
m(\tau, b, f) = \beta\left(\left\langle \tau, \frac{b}{|b|}\right\rangle\right)\langle b^\tau, f\rangle\, b^\tau + \left[1- \beta\left(\left\langle \tau, \frac{b}{|b|}\right\rangle\right)\right] \frac{\sum_{i=1}^N \big|\langle f, s_i\rangle\big|^p\,\langle f,s_i\rangle \,s_i}{\sum_{j=1}^N\big|\langle f, s_j\rangle\big|^p}
\]
for some large $p$. Here 
\[
b^\tau := \frac{b- \langle b,\tau\rangle\,\tau}{\big|b- \langle b,\tau\rangle\,\tau\big|}
\]
denotes the projection of $b$ onto the component normal to $\tau$, which is the only direction in which an edge dislocation can effectively move (since tangential movement is irrelevant due to reparametrization). The function $\beta:[-1,1]\to [0,1]$ should be chosen even, smooth and monotone on the positive numbers to switch over between the two types of dynamics and satisfy $\beta(0) = 1, \beta(\pm 1) = 0$. When $p$ is large and there is a unique $s_i$ which maximizes $|\langle s_i, f\rangle|$, then the quotient will converge to $\langle f, s_i\rangle\,s_i$ as $p\to \infty$, otherwise it will pick out the average of the directions which are most aligned with the force.

This ``best'' direction lies in the convex hull of the maximally dissipating crystallographic directions and a simple argument shows that this case occurs exactly if $f$ is parallel to the average of all these directions, so that energy dissipation is truly maximized. 

By the same proof as above, if solutions to \eqref{eq physical} satisfy uniform bounds in suitable spaces, they converge to solutions of
\[
\begin{pde}
\partial_t\Gamma &= m\left(\frac{\gamma'}{|\gamma'|}, b, \vv H_\Gamma\right) &t>0,\\
\Gamma &= \gamma_0 &t=0.
\end{pde}
\]
As before, the main issue is that of existence and uniform regularity, which seems difficult for such a highly non-linear and only degenerately parabolic equation.
\end{remark}

\begin{remark}
The following toy example shows that the properties of $F^{PK,\eps}$ \com{obtained so far} are not sufficient to prove existence of the gradient flow. Let 
\[
f(z) := \max\{-\delta, \min \{\lambda z, \delta\}\}
\]
\com{for some $\lambda\in\R$ and $\delta>0$.
C}onsider the equation
\[
(\partial_t-\Delta) u = f(-\Delta u).
\]
This equation has the same \com{properties that} we isolated in the gradient flow of the Peach-Koehler force with the non-linearity $f$ playing the role of $R$ and the Laplacian replacing the mean curvature. \com{Specifically, for $u\in C^{2,\alpha}$ we find that $f(u)$ is small in $C^0$ and bounded in $C^{0,\alpha}$,} but if we choose an initial condition $u^0$ such that $-\delta < -\lambda\Delta u^{\com0} < \delta$, then the equation becomes
\[
u_t = (1-\lambda)\,\Delta u
\]
close to $u^0$, which is a backwards heat equation if $\lambda>1$. If $u^0$ is not $C^\infty$-smooth, this does not have \com{a} solution since solutions to the heat equation become instantaneously smooth. 
\end{remark}

In the next section we give \com{an} argument that does not require these \com{a priori existence and regularity} assumptions, but \com{which is hinged on} a less physical energy dissipation mechanism.

\subsection{$H^1$-type dissipation}\label{section h1 dissipation}

In this section, we prove existence, regularity and convergence, as $\eps\to 0$, for the gradient flows of $\Eeff$ with respect to an $H^1$-type dissipation. This dissipation regularizes the evolution equation in a way that is not described by a partial differential equation anymore, but \com{instead} by an ordinary differential equation in an open subset of the Banach space $C^{2,\alpha}$, which can be solved by the Picard-Lindel\"off theorem. We do not need the parabolicity of the operator and the smallness of the perturbative term as much as in the case of $L^2$-dissipation.

The procedure is slightly non-standard since the singular perturbation acts on the dissipation mechanism rather than \com{on} the energy functional. \com{A} similar procedure has been employed in \cite{hudson} \com{using} an explicit Euler scheme. \com{Instead, here we employ abstract arguments.} We consider a more \com{specific} situation and \com{although we} obtain only short time-existence, \com{we} prove higher regularity of the evolving curves and manage to pass to the limit in $\eps$.

\begin{lemma}\label{lemma lipschitz force}
Let $\gamma_0\in C^{2,\alpha}(S^1;\Omega)$ be an embedded curve and 
\[
U := \{\gamma\in C^{2,\alpha}(S^1;\Omega)\:|\:\|\gamma-\gamma_0\|_{C^{2,\alpha}}<\rho\}
\]
for some $\rho>0$. Then the Peach-Koehler force $F^{PK, \eps}_\gamma$ is given by a Lipschitz map (uniformly in $\eps$) from $\overline U$ into the vector space $C^{0,\alpha}(S^1;\R^3)$ if $\rho$ is small enough, i.e.,
\[
\|F^{PK, \eps}_\gamma - F^{PK, \eps}_\eta \|_{C^{0,\alpha}(S^1;\R^3)} \leq L\,\|\gamma - \eta\|_{C^{2,\alpha}(S^1;\R^3)}
\]
where $L$ does not depend on $\eps>0$ (for sufficiently small $\eps$) but may depend on $\gamma_0$ and $\rho$.
\end{lemma}

\begin{proof}
\com{
In Theorem \ref{theorem asymptotics force}, we have seen that $F^{PK,\eps}_\eta$ is $C^{0,\alpha}$-continuous along $\eta$ with constants depending only on quantities which can be estimated uniformly in $U$. It remains to show the Lipschitz estimate. This proof uses the Lipschitz estimate of Corollary \ref{corollary lipschitz estimate} and a close inspection of the remainder terms in the proof of Theorem \ref{theorem asymptotics force}. According to the renormalization by a factor $|\log\eps|^{-1}$ of the Peach-Koehler force in \eqref{eq Peach-Koehler force}, the factor of $|\log\eps|$ in Corollary \ref{corollary lipschitz estimate} disappears. As we have presented many similar arguments in detail, we leave the details to the reader.
}
\end{proof}

\com{C}onsider the scalar product
\[
\langle u, v\rangle_{(L^2 + \delta H^1)(\gamma)} := \int_\gamma \left(uv + \delta\,\langle \nabla^\gamma u, \nabla^\gamma v\rangle\right)\d\H^1 = \int_\gamma \left(uv + \delta \langle \nabla u, \tau\rangle\,\langle \nabla v, \tau\right)\rangle\d\H^1,
\]
where $\nabla^\gamma$ denotes the tangential gradient along $\gamma$. With respect to this singularly perturbed $L^2$-inner product for the time-derivative, the gradient flow equation becomes
\[
(1- \delta\,\Delta_{\Gamma(t,\cdot)}) \Gamma_t = F^{PK,\eps}_{\Gamma(t,\cdot)},
\]
or, equivalently,
\begin{equation}\label{eq h1 gf eps}
 \Gamma_t = (1- \delta\,\Delta_{\Gamma(t,\cdot)})^{-1}\,F^{PK,\eps}_{\Gamma(t,\cdot)}.
\end{equation}
We keep $\delta>0$ fixed. \com{An interesting open question is w}hether or not it is possible to take $\delta\to 0$ and recover the $L^2$-gradient flow, possibly after taking $\eps\to 0$. We provide brief observations in the linear case in Appendix \ref{appendix l2 h1}. Also, technical results on the operators $-\delta \Delta_\gamma +1$ are collected there.

\begin{corollary}\label{corollary motion h1}
Solutions to the $L^2+\delta\cdot H^1$-gradient flow \eqref{eq h1 gf eps} of $\Eeff$ exist on a time interval $[0,T_\delta]$, where $T_\delta \geq T\delta$ for a constant $T$ independent of $\eps$ and $\delta$, and they converge to the $L^2 + \delta\cdot H^1$-gradient flow of the arclength functional in $C^0\big([0,T_\delta]; C^{2,\beta}(S^1)\big)$ for all $\beta<\alpha$ as $\eps \to 0$. \com{To be precise,}
\begin{align*}
\|\Gamma_{\delta,\eps}(t) - \Gamma_\delta(t)\|_{C^{2,\beta}}  &\leq \frac{C\, t}{\delta\,|\log\eps|^{1-\frac\beta\alpha}}\,\exp\left(Ct\right),\\
\text{and}\qquad \|\partial_t\Gamma_{\delta,\eps}(t) - \partial_t\Gamma_\delta(t)\|_{C^{2,\beta}}  &\leq \frac{C\, t}{\delta^2\,|\log\eps|^{1-\frac\beta\alpha}}\,\exp\left(Ct\right)
\end{align*}
for suitable constants $C$.
\end{corollary}

\begin{proof}
Let $0<\delta\leq 1$, \com{$\eps>0$} and define $\Phi: U\to C^{2,\alpha}(S^1;\R^3)$ by
\begin{equation}\label{eq definition Phi}
\Phi_{\delta,\eps}(\gamma) := (1- \delta\cdot \Delta_\gamma)^{-1}\,F^{PK,\eps}_\gamma.
\end{equation}
where $U$ is a suitable open neighborhood of $\gamma_0$ in $C^{2,\alpha}(S^1; \R^3)$, namely, all curves in $U$ need to be embedded in $\Omega$ with embeddedness radius uniformly bounded away from $0$ (and a uniform lower distance from $\partial\Omega$), satisfy uniform upper and lower bounds on $|\gamma'|$ and uniform bounds in $C^{2,\alpha}$. Then \com{we use the estimates of Theorem \ref{theorem regularizing identity} for the regularizing effect of the operator $(1-\delta\Delta)$ to show that}
\begin{align*}
\|\Phi_{\delta,\eps}(\gamma) - \Phi_{\delta,\eps}(\eta)&\|_{C^{2,\alpha}(S^1;\R^3)}  = \left\|(1- \delta\cdot \Delta_\gamma)^{-1}\,F^{PK,\eps}_\gamma - (1- \delta\cdot \Delta_\eta)^{-1}\,F^{PK,\eps}_\eta\right\|_{C^{2,\alpha}(S^1;\R^3)}\\
	&\leq \left\|(1- \delta\cdot \Delta_\gamma)^{-1}\,\left(F^{PK,\eps}_\gamma - F^{PK,\eps}_\eta\right)\right\|_{C^{2,\alpha}}\\
		&\quad + \left\| \left(( 1-\delta \cdot \Delta_\gamma)^{-1} - (1-\delta \cdot\Delta_\eta)^{-1}\right)\,F^{PK,\eps}_\eta\right\|_{C^{2,\alpha}}\\
	&\leq \frac{C}\delta\,\left\|F^{PK,\eps}_\gamma - F^{PK,\eps}_\eta\right\|_{C^{0,\alpha}}\\
		&\quad + \left\|(1- \delta\cdot \Delta_\gamma)^{-1} - (1- \delta\cdot \Delta_\eta)^{-1}\right\|_{L(C^{0,\alpha}; C^{2,\alpha})} \|F^{PK,\eps}_\eta\|_{C^{0,\alpha}}\\
	&\leq \frac{C}\delta\,\|\gamma - \eta \|_{C^{2,\alpha}} + C\, \left\|(1- \delta\cdot \Delta_\gamma)^{-1} - (1- \delta\cdot \Delta_\eta)^{-1}\right\|_{L(C^{0,\alpha}; C^{2,\alpha})}\numberthis\label{eq lemma 5.4 1}
\end{align*}
\com{due to Lemma \ref{lemma lipschitz force} and using that} the operators $\Delta_\gamma$ are uniformly elliptic with uniformly $C^{0,\alpha}$-continuous coefficients for $\gamma$ close to $\gamma_0$ and since the Peach-Koehler force on all curves $\gamma$ in $U$ is uniformly bounded in $C^{0,\alpha}$. The behavior of the constants as $\frac1\delta$ is obtained in Theorem \ref{theorem regularizing identity} in the appendix. It remains to estimate the norm difference of the regularizing operators.

\com{In view of Theorem \ref{theorem regularizing identity} w}e observe that if
$f\in C^{0,\alpha}(S^1)$ and
\[
(-\delta \Delta_\gamma+1) u^\gamma = (-\delta \Delta_\eta +1 ) u^\eta = f, 
\]
then $\|u^\gamma\|_{C^{2,\alpha}(S^1)} + \|u^\eta\|_{C^{2,\alpha}(S^1)} \leq \frac C\delta\,\|f\|_{C^{0,\alpha}}$, where $C$ is a uniform constant depending only on $U$. Furthermore,
\[
0 = (-\delta \Delta_\gamma +1) u^\gamma - (-\delta \Delta_\eta +1) u^\eta = -\delta\,\Delta_\gamma(u^\gamma-u^\eta) + \delta \,(\Delta_\gamma - \Delta_\eta) u^\eta + (u^\gamma - u^\eta),
\]
so that $v := u^\gamma-u^\eta$ satisfies
\[
-\delta\,\Delta_\gamma v + v = \delta\,(\Delta_\eta - \Delta_\gamma)\,u^\eta.
\]
\com{Since}
\[
\|\delta\,(\Delta_\eta - \Delta_\gamma)u^\eta\| \leq\,C\delta\,\|\gamma-\eta\|_{C^{2,\alpha}}\|u^\eta\|_{C^{2,\alpha}} \leq C\delta\,\|\gamma-\eta\|_{C^{2,\alpha}}\,\frac C\delta\,\|f\|_{C^{0,\alpha}} = C\,\|\gamma-\eta\|_{C^{2,\alpha}}\,\|f\|_{C^{0,\alpha}},
\]
we obtain that 
\begin{equation}\label{eq solution operator difference}
\left\|(1- \delta\cdot \Delta_\gamma)^{-1} - (1- \delta\cdot \Delta_\eta)^{-1}\right\|_{L(C^{0,\alpha}; C^{2,\alpha})} \leq C\, \|\gamma - \eta\|_{C^{2,\alpha}},
\end{equation}
and \com{by \eqref{eq lemma 5.4 1}}
\[
\|\Phi_{\delta,\eps}(\gamma) - \Phi_{\delta,\eps}(\eta)\|_{C^{2,\alpha}(S^1;\R^3)} \leq C\delta^{-1}\,\|\gamma-\eta\|_{C^{2,\alpha}(S^1;\R^3)}.
\]
By the Picard-Lindel\"off theorem \cite[Section 4.II]{konigsberger2013analysis}, there exists a time $t'>0$ (depending on $\eps$ and $\delta$) such that the ODE 
\begin{equation}\label{eq gradient flow ode}
\begin{pde}
\frac{d}{dt}\Gamma &= \Phi_{\delta,\eps}(\Gamma) &t>0,\\
\Gamma &= \gamma_0 &t=0
\end{pde}
\end{equation}
has a solution in $C^1\left([0,t'), U\right)$. \com{We observe that} solutions exist as long as $\Gamma(t)\in U$, \com{so using the fact} $\|\Phi_{\eps,\delta}(\gamma)\|_{C^{2,\alpha}}\leq C\,\delta^{-1}$ for all $\gamma \in U$, it follows that 
\[
\|\Gamma(t) - \gamma_0\|_{C^{2,\alpha}} = \left\|\int_0^t \Phi_{\eps,\delta}(\Gamma(t))\dt \right\|_{C^{2,\alpha}} \leq \int_0^t \| \Phi_{\eps,\delta}(\Gamma(t)) \|_{C^{2,\alpha}}\dt \leq \frac{C\,t}\delta \com{<\rho}.
\]
\com{Since} $U$ contains a $C^{2,\alpha}$-ball of some radius $\rho>0$ around $\gamma_0$, we \com{obtain} $\Gamma(t)\in U$ for all $t< \frac{\delta \rho}C$, such that \com{the time of existence} $T_\delta$ \com{is lower bounded by} $\frac{\rho}C\,\delta =:T\delta$. \com{In view of \eqref{eq definition Phi}}, we remark that \eqref{eq gradient flow ode} is precisely the $L^2+ \delta\cdot H^1$-gradient flow of $\Eeff$, therefore we have proved existence to the corresponding solution equation on a short time interval.
 
Finally, \com{if $\beta<\alpha$ then we have}
\begin{align*}
\bigg\|\Phi_{\eps,\delta}(\gamma) &- (1-\delta\Delta_\gamma)^{-1}\frac{|b|^2}{4\pi}\vv H_\gamma\bigg\|_{C^{0,\beta}}	\\
\leq &\left\| (1-\delta\Delta_\gamma)^{-1} \left(F^{PK,\eps}_\gamma - \frac{|b|^2}{4\pi} \vv H_{\com\gamma}\right)\right\|_{C^{0,\beta}} \\
	\leq &C\,\left\|F^{PK,\eps}_\gamma - \frac{|b|^2}{4\pi}\vv H_\gamma\right\|_{C^{0,\beta}}\\
	\leq &C\, \left(\left\|F^{PK,\eps}_\gamma - \frac{|b|^2}{4\pi}\vv H_\gamma\right\|_{L^\infty} + \left[F^{PK,\eps}_\gamma - \frac{|b|^2}{4\pi} \vv H_\gamma\right]_{C^{0,\alpha}}^\frac\beta\alpha \left\|F^{PK,\eps}_\gamma - \frac{|b|^2}{4\pi}\vv H_\gamma\right\|_{L^\infty}^{1-\frac\beta\alpha} \right)\numberthis \label{eq some equation}
\end{align*}
for all $\gamma\in U$ due to Theorem \ref{theorem regularizing identity}, so \com{denoting} by $\Gamma_{\delta,\eps}$ the solution of \eqref{eq gradient flow ode} and by $\Gamma_\delta$ the solution of 
\[
\begin{pde}
\frac{d}{dt}\Gamma &= (1-\delta\Delta_{\Gamma(t)})^{-1}\frac{|b|^2}{4\pi}\vv H_{\Gamma(t)} &t>0\\
\Gamma &= \gamma_0 &t=0,
\end{pde}
\]
we observe that
\begin{align*}
\|\Gamma_{\delta,\eps}(t) - \Gamma_\delta(t)\|_{C^{2,\beta}} &= \left\|\int_0^t \left[\Phi_{\delta,\eps}(\Gamma_{\delta,\eps}(s)) -  (1-\delta\Delta_{\Gamma(s)})^{-1}\frac{|b|^2}{4\pi}\vv H_{\Gamma_\delta(s)}\right]\ds\right\|_{C^{2,\beta}}\\
	&\leq \int_0^t \big\| \Phi_{\delta,\eps}(\Gamma_{\delta,\eps}(s)) - \Phi_{\delta,\eps}(\Gamma_{\delta}(s))\big\|_{C^{2,\beta}}\ds\\ 
	&\quad + \int_0^t \big\|(1-\delta\Delta_{\Gamma(s)})^{-1} \left( F^{PK, \eps}(\Gamma(s)) - \frac{|b|^2}{4\pi}\vv H_{\Gamma(s)}\right)\big\|_{C^{2,\beta}}\ds\\
	&\leq \int_0^t C\,\|\Gamma_{\delta}(s) - \Gamma_{\delta,\eps}(s)\|_{C^{2,\beta}}\ds\\
	&\quad + \frac{C}\delta\int_0^t \|F^{PK,\eps}(\Gamma(s)) - \frac{|b|^2}{4\pi}\vv H_{\Gamma(s)}\|_{C^{0,\beta}}\ds\\
	&\leq C\int_0^t \|\Gamma_{\delta}(s) - \Gamma_{\delta,\eps}(s)\|_{C^{2,\beta}}\ds + \frac{C}{\delta\,|\log\eps|^{1-\frac\beta\alpha}}t.
\end{align*}
\com{In the last inequality, we used \eqref{eq some equation} and the estimates on $F^{PK,\eps}-\frac{|b|^2}{4\pi}\vv H$ from Theorem \ref{theorem asymptotics force}.} \com{B}y Gr\"onwall's Lemma we deduce that
\[
\|\Gamma_{\delta,\eps}(t) - \Gamma_\delta(t)\|_{C^{2,\beta}}  \leq \frac{C\, t}{\delta\,|\log\eps|^{1-\frac\beta\alpha}}\,\exp\left(Ct\right)
\]
for all $0\leq t\leq T_\delta$. \com{Letting} $\eps\to 0$, \com{we obtain} $\Gamma_{\delta,\eps}\to \Gamma_\delta$ in $C^{0}\big([0,T_\delta]; C^{2,\beta}(S^1)\big)$. 
\com{Furthermore, since} 
\begin{align*}
\| \partial_t\Gamma_\delta - \partial_t\Gamma_{\delta,\eps}\|_{C^{2,\beta}} &= \|(1-\delta_{\Gamma_\delta})^{-1}\vv H_\Gamma - (1-\delta_{\Gamma_{\delta,\eps}})^{-1}F^{PK,\eps}\|_{C^{2,\beta}}
\end{align*}
\com{and $\Gamma_{\delta,\eps} \to \Gamma_\delta$ in $C^0([0,T_\delta],C^{2,\beta}(S^1))$ for all $\beta<\alpha$, we conclude that also the time derivatives converge, again using Theorem \ref{theorem regularizing identity} and \eqref{eq solution operator difference}.}
\end{proof}

\begin{remark}\label{remark physical h1}
\com{The} same proof could be applied to an equation of the form
\[
\begin{pde}
\partial_t\Gamma &= m\left( \frac{\Gamma'}{|\Gamma'|}, b, (1- \delta\cdot \Delta_\Gamma)^{-1}\,F^{PK,\eps}_\Gamma\right)&t>0,\\
\Gamma&= \gamma_0 &t=0,
\end{pde}
\]
with a mobility $m$ like in Remark \ref{remark physical dissipation} to yield an {\em unconditional} convergence result as $\eps\to 0$. This versatility comes from the fact \com{the regularized dissipation transformed the PDE problem into a Banach-space valued ODE. We did not use parabolicity anywhere in the proofs and the same arguments would go through if we reversed the sign of the time derivative, evolving curves in the direction of {\em increasing} energy. The key tool are error bounds in H\"older spaces.}
\end{remark}

\begin{remark}
Our choice of $H^1$-type dissipation does not preserve orthogonality of the velocity. A different $H^1$-type dissipation which automatically gives a normal velocity field is proposed in \cite{hudson}, and long time existence for the evolution equation is proved. The results \com{in \cite{hudson}} differ from ours in several ways:
\begin{enumerate}
\item The energy is regularized in a different way and considered in $\R^3$ instead of a bounded domain $\Omega$,
\item the evolving curves are only $C^{1,\alpha}$-smooth in space, and
\item relevant bounds degenerate as $|\log\eps|\,\eps^{-1}\,\delta^{-1}$. While our bounds also degenerate as $\delta\to 0$ and the logarithmic factor is a matter of normalization, we improved on the degeneracy in $\eps$ and can pass to the limit.
\end{enumerate}
\end{remark}

\section{Discussion of Results}\label{section conclusion}

While we always phrased our results in terms of a single curve embedded in a domain $\Omega$, the same arguments apply to a system of finitely many smooth embedded curves that do not touch, with potentially different Burgers vectors. Many problems in the asymptotic study of dislocation dynamics remain open.

\begin{itemize}
\item We only prove that $F^{PK,\eps} = \frac{|b|^2}{4\pi}\vv H + R$ where $R$ is small in $L^\infty$ and bounded in $C^{0,\alpha}$ for $C^{2,\alpha}$-curves. It is likely that showing that $R$ is in fact small in $C^{0,\alpha}$ would allow for more satisfying results in the time-dependent setting as well.

\item Our argument applies only to dislocation loops inside a crystal (or a union of such loops). These are one class of divergence free line defect, while another one is given by lines ending at the domain boundary. Our asymptotic expansion is invalid at such boundary points, and it is not clear whether dislocation curves meeting the boundary of a crystal at a right angle evolve by the same flow complemented with a Neumann condition.

Similarly, a dislocation loop may remain regular but move through points of self-contact under curve-shortening flow before developing singularities, especially when starting from knotted configurations. The asymptotic expansion \com{ceases} to be valid near crossings and the local evolution law of curve shortening flow does not `see' such points which become singular only in a non-local fashion. 

It is to be expected that also the core-radius regularization \com{ceases} to be valid at such configurations and that atomistic effects should influence the dynamics.

\item As we observed in Theorem \ref{theorem: limitenergy}, the elastic energy is dominated by the distortion field close to the dislocation line. For straight dislocations, this term is large but independent of the location of dislocations, so that the Peach-Koehler force only stems from the interaction of different dislocation lines and the crystal boundary \com{and thus} the fine structure of the core region is irrelevant.

On the other hand, the Peach-Koehler force \com{on curved dislocations} is also determined by the local properties of the dislocation line to highest order. In this setting, the cut-off approach may not be justified since it ignores the fine structure of the region close to the dislocation, which is precisely where the largest part of the energy and force are concentrated. Atomistic or non-linearly continuum-elastic models might be more appropriate in this region, especially if the dislocation curves are not assumed to be as smooth as in this article or when they collide or interact with other defects.

\item Instead of simplified linearized elasticity, a more realistic mathematical model would be linearized isotropic elasticity where the energy of an elastic deformation $\beta$ with dislocation measure $\mu$ is given \com{by}
\[
\qquad\qquad \E_{\lambda,\mu,\eps} (\beta) := \frac12\int_{\Omega_\eps} \left(2\mu\,|e(\beta)|^2 + \lambda\,\tr(\beta)^2\right) \dx, \qquad\ e(\beta) := \frac{\beta + \beta^T}2, \quad  \mu>0, \:\lambda+2\mu>0.
\]
Unlike the simplified linearized elastic energy, this functional distinguishes between dislocations of edge and screw type (Burgers vector orthogonal/parallel to the tangent line of the dislocation) since for a rank-one matrix we have
\[
|b\otimes v|^2 = |b|^2\,|v|^2
\]
independently of how $b, v$ are oriented, but
\[
\big|e(b\otimes v)\big|^2 = \frac14\left[|b\otimes v|^2 + 2\,\left\langle b\otimes v, v\otimes b\right\rangle + |v\otimes b|^2\right] = \frac{|b|^2|v|^2 + 2\langle b, v\rangle^2}2
\]
which takes into account the orientation between $b$ and $v$. The limiting energy becomes an anisotropic line tension where edge dislocations ($b\bot v$) are cheaper than screw dislocations ($b\parallel v$). The anisotropy stems from the choice of the Burgers vector, not the elastic energy.

Even more realistically, crystals should be assumed to be anisotropic. \com{Atoms in crystalline materials are arranged on a lattice which only admits a finite symmetry group (for example face- or body centered cubic lattices). Polycrystalline materials, which are composed of many small crystal grains which that by the orientation of the grid, may appear isotropic on larger scales, but not on the relevant scale for grid defects such as dislocations.}

\item \com{Going beyond} a pure gradient flow, \com{one may ask} how dislocations move under an applied stress. \com{A} relevant case \com{is} that of an applied surface force on the domain boundary. Assuming a slow enough evolution to be quasi-static, we can model this as the gradient flow of the time dependent energy \com{(see Remark \ref{remark forcing})}
\[
\F_\eps(\mu, t) := \inf\left\{ \E_{\lambda,\mu,\eps}(\beta) - \int_\Sigma \langle f(t,x), \beta^{lift}\rangle \d\H^2_x\:\bigg|\: \beta\in \A(\mu_\gamma)\right\},
\]
where $\Sigma \subset\partial\Omega$ is a suitable subset of the boundary and $f\in L^2(\Sigma)$ is the applied force. As in the present case, we assume that geometric effects dominate the applied force to highest order, and that in the limit $\eps\to 0$ the pure relaxation of the dislocation lines dominates. If the forcing is strong, it seems reasonable to scale it appropriately with $\eps$ and \com{instead} use the energy functional
\[
\widetilde\F_\eps(\mu, t) := \inf\left\{ \E_{\lambda,\mu,\eps}(\beta) - |\log\eps|\int_\Sigma \langle f(t,x), \beta^{lift}\rangle \d\H^2_x\:\bigg|\: \beta\in \A(\mu_\gamma) \right\}.
\]
In this scaling, we expect forcing and geometric motion to interact.

\item The asymptotic expansion \com{of Theorem \ref{theorem asymptotics force}} \com{ceases} to be valid when two curves (or two segments of the same curve) come closer than $|\log\eps|^{-1}$ to each other when the interaction effects reach the same order as curvature effects. Since $|\log\eps|^{-1}$ is only moderately small even for small $\eps$, this situation may be relevant to the study of real crystals with many dislocations.

We note that in a phase-field model for planar crystal dislocations with a single Burgers vector in simplified linearized elasticity \cite{MR1935021,MR2231783,MR2178227}, interaction with dislocations penetrating the single slip plane can have large influence on the dynamics in certain scaling limits \cite{dondl2017effect}. More precisely, it is shown that the interaction of dislocations in different gliding systems can have a dissipation-like effect that allows even slightly curved dislocations (and thus, slip regions) to persist over long time scales. This confirms the link of dislocations to plastic (rather than elastic) deformation, and suggests that the interaction between dislocations might be a key component in plasticity.

\item The $H^1$-type dissipation \com{in Section \ref{section h1 dissipation}} was not physically motivated, and was chosen \com{for mathematical purposes}. Even the $L^2$-gradient flow with constant mobility as a motion law does not capture the structure of crystalline materials. 

It is generally agreed upon that edge dislocations should only be able to move only in the plane spanned by the dislocation tangent and the Burgers vector, while screw dislocations can move in certain crystallographic directions dictated by the crystal and move along one of finitely many vectors according to a maximal dissipation criterion \cite{MR3323554}. Dislocations of mixed type can, for this purpose, be considered as edge-type unless the Burgers vector and dislocation tangent are almost parallel. 

%It would be desirable to obtain a dislocation law which approximates these criteria well but is sufficiently parabolic to repeat the existence and convergence theory outlined above. 
Note that such a motion law would have tremendous influence on the limiting motion law: Assume, for example, that $\gamma$ is a planar curve and that $\Omega$ is symmetric under reflection about the plane. Assume furthermore that $\gamma$ is an edge dislocation with a Burgers vector $b$ which is orthogonal to the plane $\gamma$ lies in. Then, by symmetry, the force acting on $\gamma$ will be everywhere in plane, but $\gamma$ can only move by tangential reparametrization and in a direction orthogonal to the force. \com{T}hus the curve is stationary for all times, while it would contract to a round point in finite time under curve shortening flow (see \cite{MR840401}).

%This may be relevant to applications in materials science since dislocations in real crystals are often not straight -- may even exhibit kinks -- and yet persist over long time scales.

\item We considered the gradient flow evolution of dislocations under the effective energy, which is quasi-stationary in the sense that the elastic deformation $\beta$ was always assumed to be energy-optimal. This \com{requires} that the relaxation time of the crystal is much faster than dislocation movement.
\end{itemize}

Some of these issues will be addressed in a forth-coming article \cite{nextone}.

\section*{Acknowledgments}

The authors thank Giovanni Leoni and Ethan O'Brien for long discussions. J.\ Ginster and S.\ Wojtowytsch thank Giovanni Leoni for his support on this project and as a mentor in general. The authors acknowledge the Center for Nonlinear Analysis where part of this work was carried
out. The research of I.\ Fonseca was partially funded by the National Science Foundation under Grants No. DMS-1411646 and DMS-1906238.

\appendix
\addtocontents{toc}{\protect\setcounter{tocdepth}{1}}

\section{Embeddedness Radius}
\label{appendix embeddedness}

\subsection{Basic ideas}
Let $\gamma: S^1 \to \R^3$ be a closed $C^{1,\alpha}$-curve, for $\alpha>1/2$, which we assume to be parametrized by arc-length. Then the curve $\gamma':\com{S^1_L}\to S^2$ is $C^{0,\alpha}$-continuous and has a compact image of Hausdorff dimension at most $\frac1\alpha <2$ \cite[Proposition 4.1.7]{attouch2014variational}. In particular, there exist a vector $N\in S^2$ and $R>0$ such that $B_R(N)\cap \gamma'(S^1) = \emptyset$, and therefore $B_{R/2}(N) \cap \eta'(S^1) = \emptyset$ for all $\eta$ which are sufficiently close to $\gamma$ in $C^{1,\alpha}(S^1; \R^3)$. For any such curve $\eta$, the vector fields $\eta', n_1, n_2:S^1\to\R^3$ defined by
\begin{equation}\label{eq normal frame}
 n_1(s) := \frac{N - \langle N, \eta'(s)\rangle \,\eta'(s)}{\big| N - \langle N, \eta'(s)\rangle \,\eta'(s)\big|}, \qquad  n_2(s) := \eta'(s) \wedge n_1(s)
\end{equation}
form an orthonormal basis of $\R^3$ at every $s\in S^1$ \com{and are} $C^{0,\alpha}$-H\"older continuous in $s$, see Figure \ref{fig: cylindricalcoordinates} for an illustration. 
\begin{figure}
\centering
 \includegraphics{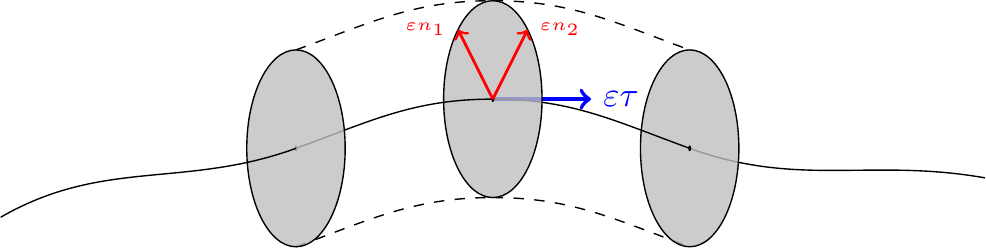}
 \caption{Sketch of the $\varepsilon$-tube around the curve $\gamma$, the tangent $\tau$, and the normal fields $n_1$ and $n_2$ as constructed below.} \label{fig: cylindricalcoordinates}
\end{figure}
%In fact, by construction the vector fields are always exactly as smooth as $\eta'$.
 Now we consider the map 
\begin{equation}\label{eq coordinates curve}
\Phi_\eta:S^1\times \R^2 \to \R^3, \qquad \Phi_\eta(s, \rho, \sigma) := \eta(s)  + \rho\,n_1(s) + \sigma\,n_2(s).
\end{equation}
\com{W}e note that $\Phi$ has the same degree of smoothness as $\eta'$, so
if $\eta\in C^2(S^1)$, then $\Phi_\eta$ is $C^1$-smooth and we can compute 
\[
\partial_s\Phi_\eta(s, 0, 0) = \eta'(s), \qquad \partial_\rho\Phi_\eta(s,0,0) = n_1(s), \qquad \partial_\sigma\Phi_\eta(s,0,0) = n_2(s).
\]
\com{Because} $D\Phi(s,0,0)$ is invertible for every $s\in S^1$, by the inverse function theorem \cite[Theorem 15.2]{deimling2010nonlinear} $\Phi$ is locally invertible at $(s,0,0)$ for every $s\in S^1$. Since $S^1$ is compact and $\eta$ is embedded, this can be strengthened to the following well-known statement:

{\em There exist open sets $U\subset S^1\times \R^2$ and $V\subset \R^3$ such that $S^1\times\{(0,0)\}\subset U$ and $\Phi:U\to V$ is a $C^1$-diffeomorphism.}

Again using the compactness of $S^1$ and a standard result in topology (sometimes known as the tube lemma \cite[Lemma 26.8]{munkres2014topology}), we may assume that $U = S^1\times D_r$ where $D_r$ denotes the disk of radius $r>0$ around the origin in $\R^2$.

\begin{definition}\label{definition embeddedness radius}
Let $\eta\in C^2(S^1;\R^3)$ and $n_1, n_2\in C^1(S^1;\R^3)$ be vector fields such that $\eta', n_1, n_2$ form an orthonormal basis at every point $s\in S^1$. \com{Consider} the map $\Phi_{\eta}$ as \com{in \eqref{eq coordinates curve}}. We define the {\em embeddedness radius} of the triple $(\eta, n_1, n_2)$ as the supremum of all $r>0$ for which 
\[
\Phi_{\eta}: S^1\times D_r\to \Phi_\eta(S^1\times D_r)\subset \R^3
\]
is a diffeomorphism.
\end{definition}

It is easy to see that the embeddedness radius is in fact independent of the choice of $n_1, n_2$ (which, for fixed $s$, only corresponds to a rotation in the plane) and thus we can speak of the {\em embeddeness radius of $\eta$}.

\subsection{Stability under convergence}

As we have seen, the embeddedness radius of every embedded $C^2$-curve $\gamma:S^1\to\R^3$ is positive. Our concern now is whether two curves $\gamma, \eta$ which are close in the $C^2$-topology also have similar embeddedness radii. The following theorem gives a positive answer to \com{this question} under an additional H\"older assumption.

\begin{theorem}
Let $\gamma\in C^{2,\alpha}(S^1; \R^3)$ \com{also} parametrized by arc-length. \com{Let} $N\in S^2$ \com{be} such that $B_R(N) \cap \gamma'(S^1) = \emptyset$ \com{for some $R>0$, and let} $r>0$ \com{be} such that $\Phi_\gamma:S^1\times D_r\to \Phi_\gamma(S^1\times D_r)$ is a diffeomorphism. 

Then for all $\eps>0$ there exists $\delta>0$, depending on $R$, $\gamma$ and $\alpha$, such that the following is true:
If $\eta \in C^{2,\alpha}(S^1; \R^3)$ \com{is} such that $\|\eta-\gamma\|_{C^{2}} < \delta$ and $\|\eta\|_{C^{2,\alpha}} \leq \|\gamma\|_{C^{2,\alpha}}+1$, then $\Phi_\eta: S^1\times D_{r-\eps}\to \Phi_\eta(S^1\times D_{r-\eps})$ is a diffeomorphism.
\end{theorem}

\begin{proof}
It suffices to find $\delta>0$ such that $\Phi_\eta: S^1 \times D_{r-\eps}\to \R^3$ is injective and has an invertible derivative at every point if $\|\eta-\gamma\|_{C^2}<\delta$.

\com{Denote by $GL(3)$ the (dense and open) set of invertible $3\times 3$ matrices and by $\partial GL(3)$ the non-invertible matrices.} Note that \com{by \eqref{eq normal frame} and \eqref{eq coordinates curve}}, $\|\Phi_\eta - \Phi_\gamma\|_{C^1(S^1\times D_r;\R^3)} \leq (1 + C_Rr)\,\|\gamma - \eta\|_{C^2(S^1;\R^3)}$ for a constant $C_R$ depending only on $R>0$, so if we choose
\[
\delta < \frac {\min_{S^1\times \overline{D_{r-\eps}}} \dist\big(D\Phi_\gamma, \partial GL(3)\big)} {2 \,(1 + C_Rr)}
\]
then $D\Phi_\eta \in GL(3)$ and $\Phi_\eta$ is locally injective. In particular, 
\[
\min_{p, v} \frac{|D\Phi_\eta(p)\cdot v|}{|v|} \geq \frac12\min_{p,v} \frac{|D\Phi_\gamma(p)\cdot v|}{|v|}.
\]

Now, assume that $\Phi_\eta(p) = \Phi_\eta(q)$ for some $p, q\in S^1\times \overline{D_{r-\eps}}$. The function
\[
(p,q) \mapsto \begin{cases}\frac{|\Phi_\gamma(p)-\Phi_\gamma(q)|}{|p-q|} &p\neq q\\ \min_{v\in S^2} \big|D\Phi(q)\,v\big| &p=q\end{cases}
\]
has a positive minimum $\tilde c$, which implies that
\begin{align*}
\tilde c\,|p-q| &\leq |\Phi_\gamma(p) - \Phi_\gamma(q)|\\
	& \leq |\Phi_\gamma(p) - \Phi_\eta(p)| + |\Phi_\eta(p) - \Phi_\eta(q)| + |\Phi_\eta(q) - \Phi_\gamma(q)|\\
	&< \delta + 0 + \delta \\
	&= 2\delta.
\end{align*}
This means that the points $p, q$ need to be close -- however, if $|p-q|< \frac{2\,\delta}{\tilde c}$, then 
\begin{align*}
\frac12\min_{p,v} \frac{|D\Phi_\gamma(p)\cdot v|}{|v|} &\leq \min_{p, v} \frac{|D\Phi_\eta(p)\cdot v|}{|v|}\\
	&\leq \frac{|D\Phi_\eta(p)\cdot (p-q)|}{|p-q|}\\
	&= \frac{\big|\Phi_\eta(p) - \Phi_\eta(q) - D\Phi_\eta(p)\cdot(p-q)\big|}{|p-q|}\\
	&\leq \frac{[D^2\Phi_\eta]_{C^{0,\alpha}}\,|p-q|^{1+\alpha}}{|p-q|}\\
	&\leq (1+C_Rr)\,[\eta'']_{C^{0,\alpha}}\,\left(\frac{2\,\delta}{\tilde c}\right)^\alpha\\
	&\leq C(R, r, \gamma)\,\delta^\alpha.
\end{align*}
Since the constant on the left hand side is positive, we find that we can choose $\delta>0$ for which the inequality becomes false.
\end{proof}

Since $r, N$ and $\Phi_\gamma$ depend explicitly on $\gamma$, we can re-write all constants above as depending on $\gamma$ only.

\begin{corollary}
Let $\gamma\in C^{2,\alpha}(S^1;\R^3)$ be an embedded curve with embeddedness radius $r>0$. Then there exists $\eps>0$ such that all curves $\eta \in C^{2,\alpha}(S^1;\R^3)$ satisfying $\|\gamma-\eta\|_{C^2}< \eps$ and $\|\eta\|_{C^{2,\alpha}}\leq \|\gamma\|_{C^{2,\alpha}}+1$ have an embeddedness radius $> r/2$. 
\end{corollary}

\begin{remark}
Of course, all such curves are $C^{2,\beta}$-close to $\gamma$ for all $\beta<\alpha$ by interpolation.
\end{remark}

\begin{remark}
It is easy to see that a map with invertible derivative at a point is invertible in a neighborhood of that point -- however, there is no a priori estimate of the size of the neighborhood if the derivative is only assumed to be continuous since there is no a priori statement how well the derivative describes the function around the point. It should therefore not be surprising that a modulus of continuity estimate of some form -- realized above as a H\"older estimate -- is needed.
\end{remark}

\begin{remark}
If $n=2$ or $n=4$, there are easy proofs for the existence of an orthonormal basis along every arc-length parametrized $C^1$-curve that do not require any further assumptions like the H\"older-continuity of the derivative. For $n=2$, the normal field is given by $n(s) = (-\dot{\gamma}_2(s), \dot{\gamma}_1(s))$; for $n=4$ we note that $S^3$ is a Lie-group and as such there are three vector-fields $X_1, X_2, X_3$ forming an orthonormal basis of $T_pS^3$ at every $p\in S^3$. The ortho-normal basis is then given by
\[
\gamma'(s), \qquad n_1(s) = X_1(\gamma'(s)), \qquad n_2(s) = X_2(\gamma'(s)), \qquad n_3(s) = X_3(\gamma'(s)).
\]
\end{remark}

\begin{remark}
A notable difference between this construction and constructions like the Frenet-Serret frame is that the frame is defined globally, while the Frenet-Serret frame may develop discontinuities at points where curvature vanishes since $\frac{\vv H}{|\vv H|}$ may become discontinuous there. A second notable difference is that all estimates depend only on the normal fields depend only on the first derivative of $\gamma$, not the first two.
\end{remark}

\subsection{Cylindrical coordinates} \label{appendix : cylindrical coordinates}

For practical purposes, we choose a slightly different parametrization of the tubular neigbourhood $B_{\delta}(\gamma)$ which resembles cylindrical coordinates. We set
\[
\Psi :S^1 \times [0,\delta) \times S^1  \to \R^3, \qquad \Psi(s, r, \theta) = \gamma(s) + r\,\left(\cos\theta\,n_1(s) + \sin\theta\,n_2(s)\right)
\]
where we identify $S^1$ with the periodic interval $[0,2\pi)$. Observe that $n_1\wedge n_2 = \gamma'$ pointwise, so
\begin{align}
\partial_s\Psi &= \gamma' + r\left(\cos\theta\,n_1' + \sin\theta\,n_2'\right),\nonumber\\
\partial_r\Psi &= \cos\theta\,n_1 + \sin\theta\,n_2,\nonumber\\
\partial_\theta\Psi &= r\left(-\sin\theta\,n_1 + \cos\theta\,n_2\right),\nonumber\\
\partial_r\Psi\wedge \partial_\theta\Psi &= \big(\cos\theta\,n_1 + \sin\theta\,n_2\big)\wedge r\left(-\sin\theta\,n_1 + \cos\theta\,n_2\right)\nonumber\\
	&= r\,\left(\cos^2\theta \,n_1\wedge n_2 - \sin^2\theta \,n_2\wedge n_1\right) \nonumber\\
	&=  r\,\gamma', \nonumber \\
\big|\det\,D\Psi\big| &= \big| \langle \partial_s\Psi, \partial_r\Psi\wedge \partial_\theta\Psi\rangle\big| \nonumber\\
	&= \big|\langle \gamma' + r\left(\cos\theta\,n_1' + \sin\theta\,n_2'\right), r\,\gamma'\rangle \big| \nonumber\\
	&= r\big| 1 +r \langle \cos\theta\,n_1' +\sin \theta\,n_2', \gamma'\rangle\big|\nonumber\\
	&= r\big| 1 - r\langle \cos\theta \,n_1 + \sin\theta\,n_2, \gamma''\rangle\big| \nonumber\\
	&= r\big| 1- r\langle \vv H, \nu\rangle\big| \nonumber\\
	&=  r\big( 1- r\langle \vv H, \nu\rangle\big), \label{eq: expansiondetpsi}
\end{align}
where $\vv H$ is the mean curvature vector of $\gamma$ at $s$ and 
\[
\nu = \nu(s,\theta) = \cos\theta\,n_1(s) + \sin\theta\,n_2(s)
\]
is the exterior normal vector to $B_r(\gamma)$ at $\Psi(s,r,\theta)$. In the derivation, we used that
\[
\langle \gamma', n_1\rangle \equiv 0\quad\Ra\quad 0 = \frac{d}{ds} \langle \gamma', n_1\rangle = \langle\gamma'', n_1\rangle + \langle \gamma', n_1'\rangle.
\]
The same argument shows that $\langle n_1,  n_1'\rangle =0$ and $\langle n_1, n_2'\rangle = - \langle n_1', n_2\rangle$.
Now, let us investigate the cylindrical hull $\partial B_r(\gamma)$ which is parametrized by 
\[
\psi_r:S^1\times S^1\to \R^3, \qquad \psi(s, \theta) = \gamma(s) + r\left(\cos\theta\,n_1(s) + \sin\theta\,n_2(s)\right).
\]
The volume element $\sqrt{\det g_{\psi_r}} = \sqrt{\det \big(D\psi_r^T\,D\psi_r\big)}$ can be obtained as follows from the area formula and co-area formula. Let $w\in C_c^\infty(B_{\overline r}(\gamma))$ be a function where $\overline r$ is the embeddedness radius. Then, since $|\nabla \dist(\cdot,\gamma)|\equiv 1$, we find that
\begin{align*}
\int_{B_{\overline r}(\gamma)} w(x)\dx &= \int_0^{\overline r}\int_{S_L^1} \int_{S^1} w(\Psi(r,s,\theta))\,r\big( 1- r\langle \vv H, \nu\rangle\big)\d\theta\ds\dr  \\
\int_{B_{\overline r}(\gamma)} w(x) \dx &= \int_{B_{\overline r}(\gamma)} w(x) \big|\nabla \dist(x,\gamma)\big| \dx \\
	&= \int_0^{\overline r} \int_{\{\dist(x,\gamma) = r\}}w \d\H^{n-1} \dr\\
	&= \int_0^{\overline r} \left(\int_{S^1}\int_{S^1} w(\psi_r(s,\theta))\,\sqrt{\det (g_{\psi_r})}\d\theta\ds\right)\dr\\
	&=  \int_0^{\overline r}\int_{S_L^1} \int_{S^1} w(\Psi(r,s,\theta))\,\sqrt{\det (g_{\psi_r}(s,\theta))}\d\theta\ds\dr,
\end{align*}
so the fundamental lemma of the calculus of variations shows that
\begin{equation}\label{eq: areadistortion}
\sqrt{\det (g_{\psi_r}(s,\theta))} = r\big( 1- r\langle \vv H, \nu\rangle\big).
\end{equation}
almost everywhere and smoothness implies that the identity holds in the strong sense.
We used fairly involved machinery in a simple situation, and
the same result can be obtained by a direct, albeit tedious calculation. In the metric tensor on $\partial B_\eps(\gamma)$, scalar products of the vectors $n_i'$ occur, which do not cancel out in an obvious way in $\det(g)$. It is necessary to use the identities 
$n_2 = \pm \gamma' \wedge n_1$ multiple times to translate scalar products back to quantities depending only on $\gamma$, together with close attention to which vectors are orthogonal or parallel.

\section{On the extension and trace constants of $\Omega_\eps$}
\label{appendix domain with hole}

In this section, we prove \com{uniform extension and trace results and uniform Poincar\'e inequalities} for domains $\Omega_\eps$ given by a fixed domain $\Omega$ from which we removed the $\eps$-tube \com{$B_\eps(\gamma)$} around an embedded $C^2$-curve $\gamma$. We make frequent use of the adapted cylindrical coordinates developed in the previous section.

We begin by proving the uniform trace theorem.

\begin{proof}[Proof of Proposition \ref{proposition uniform trace}]
Let $\Psi$ be the adapted cylindrical coordinates as introduced in Appendix \ref{appendix embeddedness}.
By \eqref{eq: areadistortion} we may choose $R>0$ smaller than the embeddedness radius such that $\det(D\Psi) = r\big( 1- r\langle \vv H, \nu\rangle\big) \geq \frac r2$ for $r<R$ (\com{Note that} both the $C^2$-norm and the embeddedness radius of $\gamma$ enter in the choice of $R$). 
Next, let $\eta\in C^\infty(\R)$ be a cut-off function such that
\[
\eta(z) = \begin{cases}1 &|z|\leq \frac{R}3,\\ 0 & |z|\geq \frac{2R}3,\end{cases} \text{ and } -\frac{4}R\leq \eta'\leq 0,
\]
and \text{set} $\tilde\eta:\Omega \to \R$, $\tilde\eta(x) := \eta(\dist(x,\gamma))$. We observe that $\tilde\eta u = u$ on $\partial B_\eps(\gamma)$ \com{for $\eps< \frac{R}3$} and that $u\tilde\eta\in H^1_0(\Omega)$ when \com{$R< \dist(\gamma,\partial\Omega)$}. We observe that
\begin{align*}
|u|\big(\psi_\eps(s,\theta)\big) &= |u\tilde\eta|\big(\Psi(s,\eps,\theta)\big)\\
	&= |u\tilde\eta|\big(\Psi(s,R,\theta)\big) - \int_\eps^R\frac{d}{dr} \, |u\tilde\eta|\big(\Psi(s,r,\theta)\big) \dr\\
	&= 0 - \int_\eps^R\left[|u| \big(\Psi(s,r,\theta)\big)\,\eta'(r) + \eta(r)\, \left\langle (Du)^T\frac{u}{|u|}\big(\Psi(s,r,\theta)\big), \frac\partial{\partial r}\Psi (s,r,\theta)\right\rangle\right]\dr\\
	&\leq \frac4R\int_{R/3}^{2R/3} |u| \big(\Psi(s,\com{r},\theta)\big)\dr + \int_\eps^R\big|Du\big|\big(\Psi(s,r,\theta)\big)\dr,
\end{align*}
since
\begin{equation}\label{eq radial derivative}
\left|\frac\partial{\partial r} \Psi(s,r,\theta)\right| = \big| \cos\theta\,n_1(s) + \sin\theta\,n_2(s)\big| = 1.
\end{equation}
\com{Recall that $S_L^1$ denotes the circle of length $L$.} Therefore, using H\"older's inequality we find that
\begin{align*}
\int_{\partial B_\eps(\gamma)}&|u|^2\d\H^2 = \int_{S^1} \int_{S^1}|u|^2\big(\psi_\eps(s,\theta)\big) \sqrt{\det g_{\psi_\eps}(s,\theta)}\d\theta\ds\\
	&\leq \int_{S^1}\int_{S^1} \left(\frac4R\int_{R/3}^{2R/3} |u| \big(\Psi(s,\com{r},\theta)\big)\dr + \int_\eps^R\big|Du\big|\big(\Psi(s,r,\theta)\big)\dr\right)^2\left[\eps + O_\gamma(\eps^2)\right]\d\theta\ds\\
	&\leq 2\int_{S^1}\int_{S^1} \left[\left(\frac4R\int_{R/3}^{2R/3} |u| \big(\Psi(s,\com{r},\theta)\big)\dr\right)^2 + \left(\int_\eps^R\big|Du\big|\big(\Psi(s,r,\theta)\big)\dr\right)^2\right]2\eps\d\theta\ds\\
	&\leq 2\int_{S_L^1}\int_{S^1}\left[\frac{4}{3}\int_{R/3}^{2R/3} |u|^2 \big(\Psi(s,\com{r},\theta)\big)\dr + (R-\eps)\int_\eps^R|Du|^2\left(\Psi(s,r,\theta)\right)\dr\right]2\eps \d\com{\theta}\\
	&\leq 2\int_{S_L^1}\int_{S^1}\int_\eps^R \left(\left[\frac4{3} |u|^2 + R\,|Du|^2\right]\circ\Psi\right)\: \frac{2\eps}{|\det(D\Psi)|}\,|\det(D\Psi)|\dr\d\theta\ds\\
	&\leq 2\int_{S_L^1}\int_{S^1}\int_\eps^R \left(\left[\frac4{3} |u|^2 + R\,|Du|^2\right]\circ\Psi\right)\: \frac{2\eps}{\frac r2}\,|\det(D\Psi)|\dr\d\theta\ds\\
	&\leq 8 \int_{S_L^1}\int_{S^1}\int_\eps^R \left(\left[\frac4{3} |u|^2 + R\,|Du|^2\right]\circ\Psi\right)\:|\det(D\Psi)|\dr\d\theta\ds\\
	&\leq 8 \max\left\{\frac43, R\right\}\,\int_{B_R(\gamma)\setminus B_\eps(\gamma)} \left[|u|^2+|Du|^2\right]\dx.
\end{align*}
\end{proof}

Let us move on to the uniform extension theorem.

\begin{proof}[Proof of Proposition \ref{proposition uniform extension}]
Let $\eta\in C^\infty(\R)$ be a cut-off function such that
\[
\eta(z) = \begin{cases}1 &|z|\leq \frac{R}2\\ 0 & |z|\geq R\end{cases}, \qquad -\frac{3}R\leq \eta'\leq 0
\]
and by abuse of notation write $\eta:\Omega \to \R$, $\eta(x) = \eta(\dist(x,\gamma))$. We define
\[
(T_\eps u)\big(\Psi(s, r,\theta)\big) = (\eta u)\left(\Psi\left(s, \frac{\eps^2}{r}, \theta\right)\right)  
\]
for $r<\eps$. By construction, $T_\eps u$ is continuous at $r=\eps$ if $u$ is continuous up to the boundary and $T_\eps u(x) = 0$ if $\dist(x,\gamma) < \frac {\eps^2}{R}$ since then $\eta\left(\frac{\eps^2}{\dist(x,\gamma)}\right) =0$. These matching conditions ensure weak differentiability of $T_\eps u$ in the entire domain $\Omega$. 

It remains to prove the uniform estimates. In standard Cartesian coordinates, we have
\[
(T_\eps u)(x) = \begin{cases} (u\eta)\big(\phi_\eps(x)\big) &\dist(x,\gamma)<\eps\\ u(x) &\dist(x,\gamma)\geq \eps\end{cases}
\]
where $\phi_\eps$ is a sequence of (local) diffeomorphisms given by
\[
\phi_\eps = \Psi \circ I_\eps \circ \Psi^{-1}, \qquad I_\eps(r,s,\theta) = \left(\frac{\eps^2}r, s, \theta\right),
\]
and $\Psi$ is again the change of coordinates constructed in Appendix \ref{appendix embeddedness}.
Consequently,
\[
\nabla T_\eps u = \eta' \,\big(\nabla \dist(\cdot,\gamma)\big) \otimes u + \eta\,Du \cdot D\phi_\eps.
\]
Without loss of generality, let us assume that $\gamma' = e_1$, $n_1 = e_2$ and $n_2 = e_3$. We see that (c.f.~Appendix \ref{appendix embeddedness})S
\[
D\Psi(s,r,\theta) = \begin{pmatrix} 1- r\langle \vv H, \nu\rangle & 0 &0\\ r\beta\,\cos\theta & \cos\theta & -r\sin\theta\\ -r\beta\,\sin\theta &\sin\theta &r\cos\theta\end{pmatrix}
\]
using as before that $\langle n_1', \gamma'\rangle = \langle n_1, -\gamma''\rangle$ and denoting $\beta:= \langle n_2', n_1\rangle$ a sort of ``torsion'' associated to the moving frame $\gamma', n_1, n_2$. The inverse of the matrix is given by
\[
%D\big(\Psi^{-1}\big)_{\Psi(s,r,\theta)} =
\big(D\Psi\big)^{-1}_{(s,r,\theta)} = 
%\frac 1{r\,(1- r\,\langle \vv H, \nu\rangle)} \begin{pmatrix} r&0 &0\\
%r^2\beta\,\left(\sin^2\theta - \cos^2\theta\right) & r\left(1-r\langle \nu, \vv H\rangle\right)\cos\theta & r\left(1-r\langle \nu, \vv H\rangle\right)\sin\theta\\
%2r\beta\,\sin\theta\cos\theta & \left(1-r\langle \nu, \vv H\rangle\right)\sin\theta & \left(1-r\langle \nu, \vv H\rangle\right)\cos\theta
%\end{pmatrix}
%=
\begin{pmatrix} \frac{1}{1-r\langle\nu,\vv H\rangle}&0 &0\\
\frac{r\beta\,\left(\cos^2\theta - \sin^2\theta\right)}{1-r\langle\nu,\vv H\rangle} & \cos\theta &\sin\theta\\
\frac{2\beta\,\sin\theta\cos\theta}{1-r\langle\nu,\vv H\rangle} &-\frac{\sin\theta }r&\frac{\cos\theta}r
\end{pmatrix}
\]
and thus
\begin{align*}
(D\phi_\eps)_{\Psi(s,r,\theta)} &= D\Psi_{I_\eps(s,r,\theta)} \cdot (D I_\eps)_{s,r,\theta} \cdot D(\Psi^{-1})_{\Psi(s,r,\theta)}\\
	&= D\Psi_{(s,\frac{\eps^2}r,\theta)} \cdot (D I_\eps)_{(s,r,\theta)} \cdot (D\Psi)^{-1}_{(s,r,\theta)}.\\
%	&= D\Psi_{(s,\frac{\eps^2}r,\theta)} \cdot \begin{pmatrix}1&0&0\\ 0&-\frac{\eps^2}{r^2} &0 \\ 0&0&1\end{pmatrix}\cdot
%		\begin{pmatrix} \frac{1}{1-r\langle\nu,\vv H\rangle}&0 &0\\
%		\frac{r\beta\,\left(\cos^2\theta - \sin^2\theta\right)}{1-r\langle\nu,\vv H\rangle} & \cos\theta &\sin\theta\\
%		\frac{2\beta\,\sin\theta\cos\theta}{1-r\langle\nu,\vv H\rangle} &-\frac{\sin\theta }r&\frac{\cos\theta}r\end{pmatrix}\\
	&= \begin{pmatrix} 1- \frac{\eps^2}r\langle \vv H, \nu\rangle & 0 &0\\ \frac{\eps^2}r\beta\,\cos\theta & \cos\theta & -\frac{\eps^2}r\sin\theta\\ -\frac{\eps^2}r\beta\,\sin\theta &\sin\theta &\frac{\eps^2}r\cos\theta\end{pmatrix}\cdot \begin{pmatrix}1&0&0\\ 0&-\frac{\eps^2}{r^2} &0 \\ 0&0&1\end{pmatrix}\cdot
		\begin{pmatrix} \frac{1}{1-r\langle\nu,\vv H\rangle}&0 &0\\
		\frac{r\beta\,\left(\cos^2\theta - \sin^2\theta\right)}{1-r\langle\nu,\vv H\rangle} & \cos\theta &\sin\theta\\
		\frac{2\beta\,\sin\theta\cos\theta}{1-r\langle\nu,\vv H\rangle} &-\frac{\sin\theta }r&\frac{\cos\theta}r\end{pmatrix}\\
	&=\begin{pmatrix} 1- \frac{\eps^2}r\langle \vv H, \nu\rangle & 0 &0\\ \frac{\eps^2}r\beta\,\cos\theta &\frac{-\eps^2}{r^2} \cos\theta & -\frac{\eps^2}r\sin\theta\\ -\frac{\eps^2}r\beta\,\sin\theta &-\frac{\eps^2}{r^2}\sin\theta &\frac{\eps^2}r\cos\theta\end{pmatrix}\cdot
		\begin{pmatrix} \frac{1}{1-r\langle\nu,\vv H\rangle}&0 &0\\
		\frac{r\beta\,\left(\cos^2\theta - \sin^2\theta\right)}{1-r\langle\nu,\vv H\rangle} & \cos\theta &\sin\theta\\
		\frac{2\beta\,\sin\theta\cos\theta}{1-r\langle\nu,\vv H\rangle} &-\frac{\sin\theta }r&\frac{\cos\theta}r\end{pmatrix}\\
	&= \begin{pmatrix}
		\frac{1- \frac{\eps^2}r\langle \vv H, \nu\rangle}{1-r\langle\nu,\vv H\rangle}&0 &0\\
		0 & \frac{\eps^2}{r^2}(\sin^2\theta - \cos^2\theta) & -\frac{2\eps^2}{r^2}\sin\theta\cos\theta\\
		0 & -\frac{2\eps^2}{r^2}\sin\theta\cos\theta & \frac{\eps^2}{r^2}(\cos^2\theta-\sin^2\theta)
		\end{pmatrix}.
\end{align*}
It follows that $||D\phi_\eps||_{L^\infty} \leq \frac{C\,\eps^2}{r^2}$ and thus
\begin{align*}
\int_{B_\eps(\gamma)} &|\nabla T_\eps u|^2\dx = \int_{S_L^1}\int_\frac{\eps^2}R^\eps \int_{S^1} |\nabla T_\eps u|^2(\Psi(s,r,\theta)) \,r\,\big(1-r\langle \vv H,\nu\rangle\big)\d\theta\dr\ds \\
	&\leq 2\int_{S_L^1}\int_\frac{\eps^2}R^\eps \int_{S^1} \left[\frac{9}{R^2}\,|u|^2 + C\left(\frac{\eps^2}{r^2}\right)^2\,|\nabla u|^2\right]\left(\Psi\left(s,\frac{\eps^2}r,\theta\right)\right)\,r\,\big(1-r\langle \vv H,\nu\rangle\big)\d\theta\dr\ds\\
	&\leq C\int_{S_L^1}\int_\frac{\eps^2}R^\eps \int_{S^1}\left[|u|^2 + |\nabla u|^2\right]\left(\Psi\left(s,\frac{\eps^2}r,\theta\right)\right) \cdot \frac{\eps^2}{r}\d\theta \frac{\eps^2}{r^2}\dr\ds\\
	&=  C\int_{S_L^1}\int_\eps^R \int_{S^1}\left[|u|^2 + |\nabla u|^2\right]\left(\Psi\left(s,r',\theta\right)\right) \cdot (r')\d\theta\dr'\ds\\
	&\leq C\int_{B_R(\gamma)\setminus B_\eps(\gamma)} |u|^2 + |\nabla u|^2\dx.
\end{align*}
A simpler version of the same argument applies to the $L^2$-norm.
\end{proof}

Finally, we uniformly bound the Poincar\'e constants of the sets $\Omega_\eps$.

\begin{proof}[Proof of Proposition \ref{proposition uniform poincare}]
%Again, it suffices to prove the proposition in the scalar-valued case. 
Assume that $u\in H^1(\Omega)$ and, without loss of generality, that $\int_{\Omega_\eps} u=0$. Use the extension operator $T_\eps$ from above and observe that 
\begin{align*}
\int_\Omega T_\eps u\dx %&= \int_{\Omega_\eps} T_\eps u\dx + \int_{B_\eps(\gamma)}T_\eps u\dx\\
	&= \int_{\Omega_\eps}u_\eps\dx + \int_{B_\eps(\gamma)}T_\eps u\dx%\\&
	= \int_{B_\eps(\gamma)} T_\eps u\dx
\end{align*}
where
\begin{align*}
\left|\int_{B_\eps(\gamma)} T_\eps u\dx\right| &\leq \left(\int_{B_\eps(\gamma)}1\dx\right)^\frac12 \left(\int_{B_\eps(\gamma)}|T_\eps u|^2\dx\right)^\frac12%\\&
	\leq C_\gamma\,\eps^{1/2} \left(\int_\Omega u^2\dx\right)^\frac12.
\end{align*}
It follows that
\begin{align*}
\int_{\Omega_\eps}|u|^2\dx &\leq \int_\Omega |T_\eps u|^2\dx\\
	%&= \int_\Omega \left|T_\eps u - \frac1{|\Omega|}\int_\Omega T_\eps u + \frac1{|\Omega|}\int_\Omega T_\eps u \right|^2\dx\\
	&\leq 2\int_\Omega \left|T_\eps u - \frac1{|\Omega|}\int_\Omega T_\eps u\right|^2 + \left|\frac1{|\Omega|}\int_\Omega T_\eps u \right|^2\dx\\
	&\leq 2\int_\Omega C_P\, |\nabla\, T_\eps u|^2 + \frac1{|\Omega|^2}\,C_\gamma^2\eps \left(\int_\Omega |u|^2\dx\right)\\
	&\leq 2C\int_{\Omega_\eps} |\nabla u|^2\dx + \frac{C_\gamma^2\,\eps}{|\Omega|}\int_\Omega |u|^2\dx
\end{align*}
which means that
\[
\int_{\Omega_\eps}|u|^2\dx\leq \frac{2C}{1- \frac{C_\gamma^2\eps}{|\Omega|}}\:\int_{\Omega_\eps} |\nabla u|^2\dx.
\]
For small $\eps$, the denominator is approximately $1$ and the constant asymptotically depends on the Poincar\'e constant of $\Omega$ as well as the extension constant, but nothing else.
\end{proof}

\section{$H^1$- and $L^2$-gradient flows}\label{appendix l2 h1}

We used an $H^1$-type dissipation for physical convenience where the dependence on the gradient of the velocity was weighted by a small parameter $\delta$. As a proof of concept, we show that in a simple and linear situation, it is possible to pass to the limit $\delta\to 0$ in the evolution equation and present first properties of solutions to such equations. 

Let $F:\overline\Omega\times\R\times \R^n\times \R^{n\times n} \to \R$ be any Lipschitz-continuous function. Then the maps
\[
A_\delta: C^{2,\alpha}(\overline\Omega) \to C^{2,\alpha}(\overline \Omega), \qquad A_\delta u = (1-\delta \Delta)^{-1}F(x,u, Du, D^2u)
\]
are Lipschitz-continuous for all $\delta>0$ with Lipschitz-constants scaling as $\frac1\delta$ and
\[
A_0: C^{2,\alpha}(\overline\Omega) \to C^{0,\alpha}(\overline \Omega), \qquad A_0 u = F(x,u, Du, D^2u)
\]
is Lipschitz-continuous. For $\delta>0$, we can easily prove existence for the evolution equation
\[
\begin{pde}
u_t &= A_\delta u &t>0\\
u &= u_0 &t=0
\end{pde}
\]
by the Picard-Lindel\"off theorem, even if the operator $A_0$ induced by $F$ fails to be elliptic. If $F$ is elliptic (and well-enough behaved), in the case $\delta=0$ existence follows from the existence theory of parabolic equations. We are interested in the question how well gradient flows of suitable energy functional $\E$ with respect to the inner product
\[
\langle u, v\rangle_\delta = \langle u,v\rangle_{L^2} + \delta \,\langle \nabla u, \nabla v\rangle_{L^2}
\]
(so Banach-space valued ODEs) approximate solutions to the (parabolic) $L^2$-gradient flow as $\delta\to0$. We briefly discuss the $H^1$-gradient flow of the Dirichlet energy for some $\delta>0$ as a model problem, i.e.\ the model equation
\[
(1-\delta\Delta)\,u_t = \Delta u.
\]
On the circle, this problem admits a direct treatment using Fourier series:
\begin{align*}
u(t,x) &= \sum_{n=0}^\infty a_n(t)\,\sin(nx) + b_n(t)\,\cos(nx)\\
\Delta u(t,x) &= -\sum_{n=0}^\infty n^2a_n(t)\,\sin(nx) + n^2b_n(t)\,\cos(nx)\\
u_t(t,x) &= \sum_{n=0}^\infty \dot a_n(t)\,\sin(nx) + \dot b_n(t)\,\cos(nx)\\
(1-\delta\Delta)u_t(t,x) &= \sum_{n=0}^\infty (1+\delta n^2)\dot a_n(t)\,\sin(nx) + (1+\delta n^2)\dot b_n(t)\,\cos(nx)
\end{align*}
which results in the ODEs for the coefficients
\[
(1+\delta n^2)\dot a_n = - n^2\,a_n\quad \Ra\quad a_n(t) = a_n(0)\,\exp\left(\frac{-n^2}{1+\delta n^2}\,t\right)
\]
and the same for $b_n$. While the exponential decay of coefficients in time resembles the behavior of solutions to the heat equation, there is a significant difference: the rate of exponential decay of the higher order coefficients is `capped' at $\frac1\delta$ since
\[
\frac{n^2}{1+\delta n^2} = \frac{1}{\delta} \frac{\delta n^2}{1+\delta n^2} \leq \frac1\delta.
\]
In particular, the rate of the exponential decay approaches $\frac1\delta$ for large $n$. While this `smoothes out' functions over time by decreasing oscillations in a non-rigorous sense, we do not observe any increase in regularity since 
\[
a_n(t) \approx a_n(0) \cdot \exp\left(\delta^{-1}t\right)
\] 
for large enough $n$, so the same summability properties hold at times $t>0$ and $t=0$, while for the heat equation with $\delta=0$ instantaneously $n^k\,a_n(t)$ becomes square-summable for all $k\in\N$, leading to $H^k$- and thus $C^\infty$-regularity. Here, an increase in regularity is not expected since the equivalent formulation of the equation as
\[
u_t = (1-\delta\Delta)^{-1}\Delta u = \Delta\circ (1-\delta\Delta)^{-1}u
\]
does not require any particular degree of spacial regularity -- in fact, the equation in this form makes sense in $L^p$- or H\"older spaces and does not require functions to have derivatives even in the weak sense.
However, we observe that, if the initial condition $u(0)$ lies in $H^k$ for $k\geq 0$, ($H^0 = L^2$), we find
\begin{align*}
\|u^\delta(t) - u^0(t)\|_{H^k} &= \sum_{n=0}^\infty \left[n^k\,\big(a_n^\delta(t)- a_n^0(t)\big)\right]^2 + \left[n^k\,\big(b_n^\delta(t)- b_n^0(t)\big)\right]^2 \\
	&= \sum_{n=0}^\infty n^{2k} \left(a_n(0)^2 + b_n(0)^2\right)\cdot \left[\exp\left(\frac{-n^2}{1+\delta n^2}\,t\right) - \exp\big(-n^2t\big)\right]^2\\
	&\to 0
\end{align*}
due to the dominated convergence theorem. Thus we note:
\begin{enumerate}
\item Solutions to the $L^2+ \delta H^1_0$-gradient flow for positive $\delta$ are not expected to increase regularity of solutions, but
\item as $\delta\to 0$, solutions $u^\delta$ approach the solution $u^0$ of the heat equation in every Hilbert Sobolev space which contains the initial condition. 
\end{enumerate}

The H\"older theory for this problem seems more complicated since the operator
\[
(1-\delta\Delta)^{-1}: C^{0,\alpha}\to C^{0,\alpha}
\]
does not approach the identity map even pointwise in the natural topology. Since $u_\delta \in C^{2,\alpha}$ by regularity arguments, $u_\delta$ lies in the small H\"older space $c^{0,\alpha}$ (the closure of $C^\infty$ in the $C^{0,\alpha}$-norm) and thus there exist right hand sides $f\in C^{0,\alpha}$ which mandate that $\|u_\delta - f\|_{C^{0,\alpha}}\geq 1$ for all $\delta$. 

We provide a statement which is needed above on the behavior of solutions in H\"older spaces. To keep things simple, we only treat the periodic case (which is all we need on the circle) and do not consider boundary effects. We believe these results to be well known but have been unable to find a reference.

\begin{theorem}\label{theorem regularizing identity}
Let $L$ be a uniformly elliptic operator with $C^{0,\alpha}$-H\"older continuous coefficients $Lu:= -a^{ij}D_iD_ju + b^iu$ on a flat $d$-dimensional torus $\mathbb T^d$ and $f\in C^{0,\alpha}(\mathbb T^d)$. Denote by $u_\delta$ the unique solution of
\[
(\delta L +1)\,u_\delta = f.
\]
Then for all $\delta<1$ we have
\[
 \|u_\delta\|_{L^\infty} \leq \|f\|_{L^\infty},\qquad  [u_\delta]_{C^{0,\alpha}} \leq [f]_{C^{0,\alpha}}
 \]
and
\[
\|u_\delta - f\|_{L^\infty} \leq C_1\,[f]_{C^{0,\alpha}}\,\delta^{\alpha/2}, \qquad \|u_\delta\|_{C^{2,\alpha}} \leq \frac{C_2}\delta\,\|f\|_{C^{0,\alpha}}.
\]
The constant $C_1$ depends only on the $L^\infty$-norm of the coefficients of $L$ while $C_2$ depends on the H\"older norm of the coefficients of $L$ and its uniform ellipticity constant. 
\end{theorem}

\begin{proof}
{\bf First $L^\infty$-estimate.} Take a point $x$ where $u_\delta$ achieves its maximum. Then $Lu_\delta(x) \geq 0$ and thus
\[
u_\delta(x) = f(x) - \delta Lu_\delta(x) \leq f(x) \leq \|f\|_{L^\infty}.
\]
The same can be done at a minimum.

{\bf $C^{0,\alpha}$-estimate.} Consider $u_{\delta, z} = u_\delta(x+z)$ which solves $(\delta L + 1) u_{\delta,z} = f_z = f(\cdot +z)$. Using the $L^\infty$-estimate, we find that
\[
(\delta L +1) (u_\delta - u_{\delta,z}) = (f - f_z) \qquad\Ra\qquad \|u_\delta - u_{\delta,z}\|_{L^\infty} \leq \|f-f_z\|_{L^\infty} \leq [f]_{C^{0,\alpha}} |z|^\alpha
\]
which implies that 
\[
|u_\delta (x) - u_\delta(x+z)| \leq [f]_{C^{0,\alpha}}\,|z|^\alpha
\]
which is a re-arrangement of the usual H\"older estimate.

{\bf $C^{2,\alpha}$-estimate.} The $C^{2,\alpha}$-estimate is now easy to obtain. We re-arrange the equation to find that
\[
\delta (L+1) u_\delta = f - (1-\delta) u_\delta \qquad\Ra\qquad \left(L+1\right)\,u_\delta = \frac{f - u_\delta + \delta\,u_\delta}\delta
\]
which by elliptic regularity theory for the operator $L+1$ implies that
\begin{align*}
\big\|u_\delta\big\|_{C^{2,\alpha}} \leq C_L \left\|\frac{f - u_\delta + \delta\,u_\delta}\delta\right\|_{C^{0,\alpha}}
	\leq \frac{C_L}\delta \left(\|f\|_{C^{0,\alpha}} + (1-\delta)\,\|u_\delta\|_{C^{0,\alpha}}\right)
	\leq \frac{C_L}\delta \,\|f\|_{C^{0,\alpha}}
\end{align*}
by the $C^{0,\alpha}$-estimate.

{\bf Second $L^\infty$-estimate.} We prove $C^0$-convergence by constructing sub- and super-solutions for $u_\delta$.
Take a standard mollifier $\eta$ and consider the convolution $f_\rho = f *\eta_\rho$ for some parameter $\rho>0$. Then we see that
\begin{align*}
\left|D^2 f_\rho\right| &= \left|f * D^2\eta_\rho\right|\\
	&= \left|\rho^{-2}\int_{\T^d}f(y)\:(D^2\eta)_\rho(x-y)\dy \right|\\
	&= \left|\rho^{-2}\int_{\T^d} \left(f(y) - f(x)\right)\,(D^2\eta)_\rho(x-y)\dy\right|\\
	&\leq [f]_{C^{0,\alpha}} \rho^{-2} \int_{\T^d} |x-y|^\alpha\, \left| D^2\eta\left(\frac{x-y}\rho\right)\right|\dy\\
	&\leq [f]_{C^{0,\alpha}} \rho^{\alpha-2} \int_{\T^d} |D^2\eta|(z)\dz
\end{align*}
since $\int_{\T^d}D^2\eta = 0$. The same argument applies to first derivatives. Hence we calculate that
\begin{align*}
(\delta L+1) \left(f_\rho - \bar c\rho^\beta\right) &\leq C\delta\,\left(\|(a^{ij})_{ij}\|_{L^\infty} \rho^{\alpha-2} + |b|_{L^\infty}\rho^{\alpha-1} \right) [f_\rho]_{C^{0,\alpha}} +  \left( f_\rho  - \bar c\rho^\beta\right)\\
	&\leq C_L\,[f]_{C^{0,\alpha}}\,\rho^{\alpha-2}\delta + f + [f]_{C^{0,\alpha}}\,\rho^\alpha - \bar c\,\rho^\beta
\end{align*}
where the constant $C_L$ depends only on the $L^\infty$-norm of the coefficients of $L$. To make sure that this is negative, we need to balance
\[
\rho^\beta = \rho^\alpha = \delta\,\rho^{\alpha-2} \qquad \Ra \qquad \alpha= \beta, \quad \rho^{2} = \delta
\]
and choose $\bar c> (C_L+1)\,[f]_{C^{0,\alpha}}$. Then
\[
(\delta L+1)\left(f_{\delta^{1/2}} - \bar c\,\delta^{\alpha/2}\right) \leq f = (\delta L+1) u_\delta ,
\] 
and since the operator $\delta L+1$ allows a comparison principle, it follows that
\[
u_\delta \geq f_\delta - \bar c\,\delta^\frac\alpha2 \geq f - [f]_{C^{0,\alpha}}\,\rho^\alpha - \bar c\,\delta^{\alpha/2}.
\]
Similarly, we construct a super-solution and obtain
\[
f - C_1\,[f]_{C^{0,\alpha}}\,\delta^{\alpha/2} \leq u_\delta \leq f+C_1\,[f]_{C^{0,\alpha}}\,\delta^{\alpha/2}.
\]
\end{proof}

\begin{remark}
Of course, it would have been possible to include a zeroth order term in $L$ or suitable boundary conditions which are close enough to $f$ and for parts of the theorem to relax the continuity condition on the coefficients.
\end{remark}

\begin{remark}
The same proof based on the translation of $f$ shows that if $f\in c^{0,\alpha}$ where 
\[
c^{0,\alpha}(\T^d) = \left\{ f\in C^{0,\alpha}(\T^d)\:\bigg|\: \lim_{\rho\to 0} \sup_{|x-y|<\rho} \frac{|f(x) - f(y)|}{|x-y|^\alpha} = 0\right\}
\]
then $u_\delta \in c^{0,\alpha}$ as well since
\[
\sup_{|x-y|<\rho} \frac{u_\delta(x) - u_\delta(y)}{|x-y|^\alpha} \leq \sup_{|x-y|<\rho} \frac{f(x) - f(y)}{|x-y|^\alpha}
\] 
by the same argument. It is easy to see that this implies convergence $u_\delta \to f$ in $C^{0,\alpha}$.
\end{remark}

%\bibliographystyle{alpha}
%\bibliography{LaplacianElasticity}

\end{document}